\RequirePackage{color}
\documentclass[bj,authoryear, round, noshowframe]{imsart}

\usepackage{amssymb}
\usepackage{amsfonts}
\usepackage{amsmath}
\usepackage{amsthm}
\usepackage{mathrsfs}
\usepackage{mathtools}
\usepackage{graphicx}
\usepackage{caption} 
\usepackage{booktabs}
\usepackage{graphicx}
\usepackage{verbatim}
\usepackage{array}
\usepackage{natbib}
\usepackage{relsize}
\usepackage{url}
\usepackage{dsfont}
\usepackage{tikz}
\usetikzlibrary{positioning}
\usepackage{hyperref}
\usepackage{algorithm}
\usepackage{tabularx}
\usepackage[noend]{algpseudocode}
\usepackage[justification=centering]{caption}
\usepackage{cancel}
\usetikzlibrary{calc,tikzmark}
\usetikzlibrary{arrows, automata}

\startlocaldefs

\renewcommand{\thetable}{\arabic{section}.\arabic{table}}
\numberwithin{equation}{section}

\newtheorem{theorem}{Theorem}[section]
\newtheorem{lemma}[theorem]{Lemma}
\newtheorem{remark}[theorem]{Remark}
\newtheorem{example}[theorem]{Example}
\newtheorem{proposition}[theorem]{Proposition}
\newtheorem{definition}[theorem]{Definition}

\newtheorem{corollary}[theorem]{Corollary}
\newtheorem{fig}[theorem]{Figure}

\def\tablename{\footnotesize Table}

\newcommand{\bthe}{\begin{theorem}}
	\newcommand{\ethe}{\end{theorem}}

\newcommand{\ben}{\begin{enumerate}}
	\newcommand{\een}{\end{enumerate}}

\newcommand{\bit}{\begin{itemize}}
	\newcommand{\eit}{\end{itemize}}

\newcommand{\beq}{\begin{equation}}
\newcommand{\eeq}{\end{equation}}

\newcommand{\ble}{\begin{lemma}}
	\newcommand{\ele}{\end{lemma}}

\newcommand{\bde}{\begin{definition}\rm}
	\newcommand{\ede}{\halmos\end{definition}}

\newcommand{\bco}{\begin{corollary}}
	\newcommand{\eco}{\end{corollary}}

\newcommand{\bpr}{\begin{proposition}}
	\newcommand{\epr}{\end{proposition}}

\newcommand{\brem}{\begin{remark}\rm}
	\newcommand{\erem}{\halmos\end{remark}}

\newcommand{\bproof}{\begin{proof}}
	\newcommand{\eproof}{\end{proof}}

\newcommand{\bexam}{\begin{example}\rm}
	\newcommand{\eexam}{\halmos\end{example}}

\newcommand{\bfi}{\begin{fig}}
	\newcommand{\efi}{\end{fig}}

\newcommand{\btab}{\begin{tab}}
	\newcommand{\etab}{\end{tab}}

\newcommand{\beao}{\begin{eqnarray*}}
	\newcommand{\eeao}{\end{eqnarray*}\noindent}

\newcommand{\beam}{\begin{eqnarray}}
\newcommand{\eeam}{\end{eqnarray}\noindent}

\newcommand{\barr}{\begin{array}}
	\newcommand{\earr}{\end{array}}

\newcommand{\bdis}{\begin{displaymath}}
\newcommand{\edis}{\end{displaymath}\noindent}

\def\N{{\mathbb N}}

\def\P{{\mathbb P}}
\def\E{{\mathbb E}}
\def\R{{\mathbb R}}

\def\P{\mathbb{P}}

\newcommand{\bs}{\boldsymbol}

\newcommand{\bsz}{\boldsymbol{Z}}
\newcommand{\bsx}{\boldsymbol{X}}
\newcommand{\bsy}{\boldsymbol{Y}}

\newcommand{\bst}{\boldsymbol{T}}

\newcommand{\stp}{\stackrel{P}{\rightarrow}}
\newcommand{\std}{\,\,\stackrel{\mathcal{D}}{\rightarrow}\,\,}

\newcommand{\stv}{\stackrel{v}{\rightarrow}}
\newcommand{\stw}{\stackrel{w}{\rightarrow}}

\newcommand{\eps}{\varepsilon}

\newcommand{\var}{{\rm Var}}

\newcommand{\DAG}{{\rm DAG}}

\newcommand{\an}{{\rm an}}
\newcommand{\pa}{{\rm pa}}

\newcommand{\des}{{\rm de}}
\newcommand{\An}{{\rm An}}
\newcommand{\Pa}{{\rm Pa}}

\newcommand{\Des}{{\rm De}}

\newcommand{\ch}{{\rm ch}}

\let\norm\undefined 
\DeclarePairedDelimiter\norm{\lVert}{\rVert}
\newcommand{\halmos}{\quad\hfill\mbox{$\Box$}}  

\def\P{{\bf {\mathbb{P}}}}

\newcommand{\D}{\mathcal{D}}

\definecolor{red-brown}{rgb}{0.65, 0.16, 0.16}
\definecolor{red(ncs)}{rgb}{0.77, 0.01, 0.2}

\newcommand{\CK}[1]{{\color{blue} #1}}

\newcommand{\Mario}{\textcolor{red(ncs)}}

\endlocaldefs

\begin{document}
\begin{frontmatter}
\title{Heavy-tailed max-linear structural equation models in networks with hidden nodes} 

\runtitle{Heavy-tailed max-linear structural equation models in networks with hidden nodes}

\begin{aug}
\author[A]{\inits{M.}\fnms{Mario}~\snm{Krali}\ead[label=e1]{mario.krali@epfl.ch}}
\author[B]{\inits{A.}\fnms{Anthony}~\snm{C. Davison}\ead[label=e2]{anthony.davison@epfl.ch}}
\author[C]{\inits{C.}\fnms{Claudia}~\snm{Kl\"uppelberg}\ead[label=e3]{cklu@ma.tum.de}}
\address[A]{Institute of Mathematics, \'Ecole Polytechnique F\'ed\'erale de Lausanne (EPFL)\printead[presep={,\ }]{e1}}

\address[B]{Institute of Mathematics, \'Ecole Polytechnique F\'ed\'erale de Lausanne (EPFL)\printead[presep={,\ }]{e2}}

\address[C]{Department of Mathematics, Technical University of Munich \printead[presep={,\ }]{e3}}

%
%

\end{aug}

\maketitle

\begin{abstract}
Recursive max-linear vectors provide models for causal dependence between large values of random variables that are supported on directed acyclic graphs, but the standard assumption that all nodes of such a graph are observed can be unrealistic. We give necessary and sufficient conditions for a partially observed recursive max-linear vector to be representable as a recursive max-linear (sub-)model and provide a graphical algorithm to construct the latter. Our conditions concern the max-weighted paths of a directed acyclic graph and its minimal representation, which play a key role for such models. In the framework of regular variation we translate these conditions into checkable criteria and establish a connection between max-weighted paths and the extremal dependence measure of transformed variables for pairs of nodes. We propose a statistical algorithm to detect bivariate regularly varying recursive max-linear models among the node variables of a directed acyclic graph and show consistency and asymptotic normality of the estimators of the extremal dependence measure under a thresholding procedure. Simulations show that our algorithm performs satisfactorily.  We apply it to nutrition intake data.

\end{abstract}
\begin{keyword}
\kwd{Bayesian network}
\kwd{directed acyclic graph}
\kwd{extreme value theory}
\kwd{recursive max-linear model}
\kwd{regular variation}
\kwd{structure learning}
\end{keyword}
\end{frontmatter}

\section{Introduction and motivation}\label{sec:intro}

Extreme value theory has become indispensable for studying rare events, with applications to high temperatures, rainfall and flooding \citep{DHT}, storms and hurricanes {\citep{davisetal,fondev}, and financial crises \citep{McNeil,poonetal}. 
Statistical modelling of such events can improve our understanding of the underlying mechanisms and thus can suggest how to mitigate their effects.  
The modelling of multivariate extremes is an area of high activity in which the nonparametric character of joint distributions~\citep[Ch.~8, 9]{beirlant} has generally restricted applications to fairly low dimensions.
A variety of dimension reduction methods for extremes have been proposed, including clustering approaches \citep{Chautru,JanWan}, a principal components-like decomposition \citep{cooley}, factor analysis \citep{HKK}, and support detection \citep{goix}.
Recent directions of research include spatial extremes \citep{DHT}, Bayesian approaches \citep[e.g.,][]{opitz2018inla}, and graphical modelling \citep[]{engelke:hitz:18, gk}.   


{Graphical models have proven useful in studying high-dimensional data, and we contribute to this area by} considering a class of graphical models for extremes using max-linear structural equation models \citep{pearl} called recursive max-linear models (RMLMs) \citep{gk} or max-linear Bayesian networks \citep{amendola,gkl}.
An RMLM $\boldsymbol{X}$ supported on a directed acyclic graph (DAG) $\mathcal{D}=(V,E)$ with nodes $V=\{1,\dots,D\}$ and edges $E$ is defined through the formula 
\begin{align}\label{semequat1}
X_i\coloneqq {\underset{k\in \pa(i)}{\bigvee}}c_{ik}X_k\vee c_{ii}Z_i,\hspace{5mm} {i\in V},
\end{align}
where the innovations $Z_1,\dots,Z_{D}$ are independent atom-free random variables with support $\mathbb{R}_+$ and the edge weights $c_{ik}$ are positive for all $i \in V$ and $k\in\pa(i)$, which denotes the parents of node $i$. For later use we assemble the innovations and weights in the innovations vector $\boldsymbol{Z}=(Z_1,\dots,Z_{D})$ and the edge-weight matrix $C=(c_{ik})_{D\times D}$.  

Max-linearity offers an analogue to linear operations when analysing how the largest shocks affect a system.  Models such as~\eqref{semequat1} have the appealing property of downplaying weaker shocks, since it is mainly the extreme shocks that disseminate through the network, and these models allow certain nodes to have key influences on other nodes.  Max-linear models can also approximate any max-stable dependence structure between extremes arbitrarily well as the number of factors grows, making them valuable objects of study \citep{foug,Wang2011}.

Identifiability and estimation for RMLMs are studied by \citet{gkl} and \citet{KL2017}, and, under the assumption of one-sided multiplicative noise, in \citet{BK}. Conditional independence under the RMLM is studied by \citet{amendola}, who introduce a new separation concept.  
\citet{TBK} propose a machine-learning algorithm for identifying the edges of an RMLM {with two-sided noise} supported on a root-directed tree. Other work studying a max-linear model in the context of trees of transitive tournaments can be found in \citet{AS}.

Engelke and coauthors have taken a different approach to graphical modelling for extremes. \citet{engelke:hitz:18} use conditional independence relations in undirected graphical models when the exponent measure has a density, based upon which \citet{eglearn} propose a graph learning procedure for the H\"usler--Reiss model. \citet{ES} develop a structure learning method for estimating the edges of an undirected graphical tree structure. \citet{gnecco:19} perform causal discovery using conditional means of the integral transforms of pairs of node variables in linear structural equation models with heavy tails, and \citet{mhalla} propose a method for causal discovery using Kolmogorov complexity and extreme conditional quantiles. \citet{EI} survey methods from some of these references. 

Analogous to the Gaussian setting, but focusing on extremes, \citet{leecooley} and \citet{Huseretal} use partial tail correlations to infer undirected graphical structures. 

\subsection{Problem statement}\label{sec:1.1}

Structural equation models require assumptions about the observed variables. One key assumption is 
that the innovations, often referred to as noise or error variables in the literature, are independent, which implies that the model is Markovian \citep[p.30]{pearl} and ensures that the model satisfies the parental Markov condition, which allows conditional independence relations to be discerned from the graphical structure \citep[Theorem~1.4.1]{pearl}.  This assumption is widely criticised because one may not observe all relevant variables, resulting in unmeasured causes \citep[p. 252]{pearl}. In graphical terminology, hidden variables may correspond to hidden confounders, i.e., variables which are unknown, undiscovered or unmeasured, and which may falsify conclusions if ignored.

Although it is hoped that structure learning procedures will find a causal order, i.e., a graph structure of the variables, few such procedures can handle hidden nodes.  To the best of our knowledge, there are just two publications on extremal graphical models dealing with similar problems. 
The approach in \citet[after equation~(9)]{gnecco:19} may be used to identify a causal order even when there are hidden nodes. Focusing on confounders only, 
\citet{PCD} propose testing for a causal link between two variables by using regression to assess whether the scale parameters depend on confounders, which must therefore be observed.

Hidden confounders, often associated with latent variables, are a lively topic of research in graphical modelling. The objective is to construct a DAG which both marginalizes out any hidden variables and preserves conditional independence relationships from the larger DAG with the hidden nodes.
Ancestral graphs, particularly maximal ancestral graphs (MAGs) \citep{richardson,colombo}, are prominent concepts targeting this problem; they allow hidden confounding effects to be embedded via so-called $m$ or  $m*$ separation, thus accounting for hidden confounders while retaining global Markov properties. \citet{richardson} apply this theory to Gaussian graphical models, whose relative simplicity enables a convenient parametrization of the MAG. In extreme value theory, and particularly in RMLMs, however, it is unclear how this theory could be applied, as RMLMs do not satisfy  the faithfulness assumption \citep{amendola}. In the extremal framework the dependence structure of the regularly varying vector $\bs X\in \mathbb{R}_+^{D}$ is characterised by its angular measure $H_{\bs X}$, which {has $D^2$ parameters for an RMLM}. In the Gaussian case the dependence is completely captured by pairwise covariances whatever the number of hidden confounders, but this is untrue for RMLMs unless the conditions given in Section~\ref{sec:RMLM} are satisfied.
	
We take a different approach, using properties of RMLMs to find conditions under which one can ignore the effects of hidden nodes.  Certain paths of an RMLM may be irrelevant because extreme shocks will never pass through them owing to the max-linearity, and this property may extend to nodes along irrelevant paths, whose lack of visibility is inconsequential.  Below we show that the irrelevance of hidden confounders is equivalent to modelling the observed nodes via an RMLM, and provide a graphical algorithm to select a (sub)set of the observed nodes that can be modelled in this way.

Under a regularly varying framework we use extremal dependencies between pairs of observed node variables to derive necessary and sufficient conditions that can be used to decide whether the observed variables can be modelled as an RMLM in the presence of hidden variables, and provide a statistical algorithm to detect bivariate regularly varying RMLMs among pairs of node variables~in~a~DAG.

Our work uses the scaling methodology of \citet{KK}, who, starting with the assumption that a regularly varying vector $\boldsymbol{X}$ can be modelled as an RMLM, first estimate a causal ordering of the nodes based on estimated scaling parameters, and then perform inference on the dependence parameters of the angular measure of $\boldsymbol{X}$}. 
 
\subsection{Terminology}\label{sec:1.2}

We use standard terminology for directed graphs \citep{Lauritzen1990}.
Let $\mathcal{D}=(V, E)$ be a DAG with node set $V=\{1,\dots,{D}\}$  and edge set $E\subset V\times V$. The parents, ancestors and descendents of a node $i\in V$ are respectively $\pa(i)=\{j\in V: e_{ji}\in E\}$, $\an(i)$ and $\des(i)$; we write $\Pa{(i)}=\pa(i)\cup \{i\} $,  $\An(i)=\an(i)\cup \{i\}$ and $\Des(i)=\des(i)\cup\{i\}$.
If $U\subseteq V$, then  $\an(U)$ denotes the ancestral set of all nodes in $U$, and $\An(U)=\an(U)\cup U$, $\des(U)$ and $\Des(U)$ are defined analogously to those for a single node.
A node $i\in V$ is a {source node} if $\pa(i)=\emptyset$, and $V_0$ denotes the set of all source nodes.
 {We write $j\to i$ to denote the edge $e_{ji}$ from node $j$ to $i$}. Then
a path $p_{ji}\coloneqq[\ell_0=j\to \ell_1 \to\cdots\to \ell_{m}=i]$ from $j$ to $i$ has length $|p_{ji}|=m$, and the set of all such paths is denoted by $P_{ji}$. 
Instead of $p_{ji}$ we also write $j\rightsquigarrow i$, and say that $X_j$ causes $X_i$ (or $j$ causes $i$) whenever there is a path $j \rightsquigarrow i$ between the corresponding nodes.

Given nodes $i,j,k\in V$, we say that $X_i$ is a confounder of $X_j$ and $X_k$ (or $i$ is a confounder of $j$ and $k$) if there exist paths $i \rightsquigarrow j$ and $i\rightsquigarrow k$ which do not pass through $k$ and $j$, respectively. 

A DAG $\mathcal{D}=(V,E)$ is called {well-ordered} if for all $i\in V$ it is true that $i<j$ for all $j\in \pa(i)$. We refer to such an order as a {causal order}. 
A graph {$\D_1=(V_1,E_1)$} is a subgraph of $\D$ if $V_1\subseteq V$ and $E_1\subset (V_1\times V_1)\cap E$. If $\D$ is a DAG, then  $\D_1$ is also a DAG.

We finish by describing endogenous and exogenous variables in structural equation models \citep[pp.~23--24]{Peters2014b}.    Endogenous variables are those that the modeler tries to understand, and exogeneous variables are independent and influence the endogenous variables, but not conversely.  In an RMLM such as~\eqref{semequat1} these are the $X_i$ and $Z_i$ respectively.  

\subsection{A motivating example}\label{ex:intro}

Figure~\ref{introex} shows three DAGs with node set $V=\{1,2,3\}$ but with node 3 hidden; we do not observe this node, and may be unaware that it exists.  The presence of an edge between two nodes $j\to i$ ($j\in\pa(i)$) indicates that $c_{ij}>0$. We use dotted edges to capture the effects of hidden nodes. Writing out equation~\eqref{semequat1} for the three DAGs gives 
 \begin{align}\label{ex:3graphs}
 X_1=c_{11}Z_1\vee c_{12}X_2\vee c_{13}X_3, \quad X_2=c_{22}Z_2\vee c_{23}X_3,\quad {X_3=c_{33}Z_3},
 \end{align}
 with $c_{13}=0$ for $\D_1$ and $c_{12}=0$ for $\D_3$. 
In this paper we seek conditions whereby $(X_1,X_2)$ can be expressed as an RMLM of the form
	$$X_1=c_{11}Z_1\vee c_{12}X_2,\,\quad \quad X_2= \tilde c_{22}\tilde{Z}_2,$$
	for independent innovations $Z_1, \tilde{Z}_2$. 
In Section~\ref{sec:RV} we prove that, under a regular variation condition on the innovations, the random variable $\tilde{Z}_2$ shares certain tail properties with $Z_2$.
		
			\begin{figure}[t]
				\begin{center}
				\resizebox{3.4cm}{2.5cm}{\begin{tikzpicture}[
					> = stealth,
					shorten > = 1pt, 
					auto,
					node distance = 2cm, 
					semithick 
					]
					\tikzstyle{every state}=[
					draw = black,
					thick,
					fill = white,
					minimum size = 4mm,scale=1
					]
					\node[state, dotted] (3) {$3$};
					\node[state] (2) [below right=1cm and 1cm of 3] {$2$};
					\node[state] (1) [below left=1cm and 1cm of 3] {$1$};
					
					\path[->][dotted, blue] (3) edge node {} (2);
					\path[->][blue] (2) edge node {} (1);
					\node (4) [below = 1.2cm of 3] {$\D_1$};
					\end{tikzpicture}}	
					\hspace{1cm}
				\resizebox{3.4cm}{2.5cm}{\begin{tikzpicture}[
					> = stealth,
					shorten > = 1pt, 
					auto,
					node distance = 2cm, 
					semithick 
					]
					\tikzstyle{every state}=[
					draw = black,
					thick,
					fill = white,
					minimum size = 4mm,scale=1
					]
					\node[state, dotted] (3) {$3$};
					\node[state] (2) [below right=1cm and 1cm of 3] {$2$};
					\node[state] (1) [below left=1cm and 1cm of 3] {$1$};
					
					\path[->][dotted, blue] (3) edge node {} (2);
					\path[->][dotted, blue] (3) edge node {} (1);
					\path[->][blue] (2) edge node {} (1);
					\node (4) [below = 1.2cm of 3]  {$\D_2$};
					\end{tikzpicture}}		
					\hspace{1cm}
					\resizebox{3.4cm}{2.5cm}{\begin{tikzpicture}[
					> = stealth,
					shorten > = 1pt, 
					auto,
					node distance = 2cm, 
					semithick 
					]
					\tikzstyle{every state}=[
					draw = black,
					thick,
					fill = white,
					minimum size = 4mm,scale=1
					]
					\node[state, dotted] (3) {$3$};
					\node[state] (2) [below right=1cm and 1cm of 3] {$2$};
					\node[state] (1) [below left=1cm and 1cm of 3] {$1$};
					
					\path[->][dotted ,blue] (3) edge node {} (2);
					\path[->][ dotted, blue] (3) edge node {} (1);
                 \node (4) [below = 1.2cm of 3] {$\D_3$};
					\end{tikzpicture}}
				
				\end{center}
				\caption{Three-dimensional DAGs in which the hidden node 3 is a confounder in $\D_2$ and $\D_3$.}\label{introex}\vspace{-.5cm}
			\end{figure}

To illustrate exogeneity and endogeneity, we reformulate~\eqref{ex:3graphs} as
\begin{align}\label{simp_ex}
X_1&=\big(c_{11}Z_1\vee c_{13}X_3\big) \vee c_{12}X_2\eqqcolon f_{13}(Z_1,X_3) \vee c_{12}X_2,\\ 
 	X_2&=c_{22}Z_2\vee c_{23}X_3\eqqcolon f_{23}(Z_2, X_3)\nonumber,
  \end{align}
and briefly discuss the three DAGs of Figure~\ref{introex}.

If $c_{13}=0$, corresponding to $\D_1$, then $f_{13}=f_{13}(Z_1)$ and $f_{23}=f_{23}(Z_2, c_{33}Z_3)$: both functions are exogenous, only depending on different innovations, which makes them independent.
Extending the notion of an innovation slightly, we call $f_{13}(Z_1)$ and $f_{23}(Z_2,  c_{33}Z_3)$ innovations of $(X_1, X_2)$. 
Indeed, $(X_1, X_2)$ is an RMLM on the smaller DAG $(\{1,2\}, 2\to 1)$. 

If $c_{13}\neq 0$ and $c_{12}\neq 0$, corresponding to $\D_2$, then
both $f_{13}$ and $f_{23}$ are functions of $X_3$ and thus are not exogenous. 
Indeed, they can be written in terms of the innovations $f_{13}(Z_1,c_{33}Z_3)$ and $f_{23}(Z_2, c_{33}Z_3)$. 

If $c_{12}=0$, corresponding to $\D_3$, then 
we have a situation similar to that for $\D_2$.

For $\D_2$ and $\D_3$, node 3 is a confounder of nodes 1 and 2 in the classical sense. In $\D_2$ and under certain conditions on the path from 3 to 1 passing through 2, we can express $(X_1,X_2)$ as an RMLM with independent innovations. It is the goal of this paper to investigate when this is possible.

Distinguishing between $\D_2$ and $\D_3$ is essential in practice because the pair $(X_1,X_2)$ is causally dependent only in $\D_2$, which implies that extreme values of $X_2$ cause extreme values of $X_1$. This arises in the application in Section~\ref{sec:data}, in which such observations correspond to unusually high nutrient levels, which can lead to toxicity or have other harmful effects. In this setting it may be important to distinguish whether simultaneous high quantities across several nutrients are causally related to each other, or whether the underlying cause is a hidden (nutrient) confounder, not present in the data. 

Focusing on $\D_3$, we briefly discuss the implications of MAGs applied to Gaussian graphical models \citep[Section 8]{richardson}, in which the dependence structure of $(X_1, X_2)$ is fully characterised by their covariance, a single parameter, which could be captured by a single edge in a bivariate graph. Since such models are marginalizable, this is true whether or not there is a single hidden confounder (as in $\D_3$), or several such: in all cases the MAG yields a bidirected edge $1\leftrightarrow2$.
	
The extremal dependence captured by the angular measure of a $D$-dimensional RMLM is more complex than in Gaussian graphical models, because the angular measure is supported on the $D$-dimensional unit simplex.  This has $2^D-1$ faces, each of which might capture important features of the dependence structure. For $\D_3$, these faces are the nodes $\{1\}, \{2\}$, and the interior of the simplex. In a more general setting with $D-2$ {independent} confounders, the support in the interior is determined by the dependence {of $(X_1, X_2)$} on each of the confounders, and the number of the parameters of the angular measure of $(X_1, X_2)$ is $2D$ (Proposition~\ref{discspect}). Many  configurations of the angular measure may yield the same extremal dependence measure, so, as the angular measure changes with the number of hidden confounders, we know of no viable way to embed this in a general way in the context of MAGs. 

\subsection{Recursive max-linear models}\label{secRMLM}

{An RMLM $\bsx$ as defined in~\eqref{semequat1}} has a unique solution, which can be derived via tropical algebra \citep[see, e.g.,][]{But2010}; i.e., linear algebra with arithmetic in the max-times semiring $(\R_+,\vee,\times)$ defined by $a\vee b:=\max(a,b)$ and $a\times b := ab$ for $a,b\in\R_+:=[0,\infty)$. 
These operations extend to $\R_+^D$ coordinatewise and to corresponding matrix multiplication $\times_{\max}$. 
In this paper vectors are column vectors; we write $\bsz=(Z_1,\dots,Z_D)$ for the column vector of innovations. 
Tropical multiplication of the max-linear (ML) coefficient matrix $A$ with $\bsz$ yields the unique solution  \citep[Theorem~2.2]{gk} 
\begin{align}\label{Rmlmequat}
X_i=(A\times_{\max} \bsz)_i={{\underset{j\in \An(i)}{\bigvee}}}a_{ij}Z_j,\hspace{5mm} i\in\{1,\dots,D\}.
\end{align}
The coefficient matrix $A=(a_{ij})_{D\times D}$ is defined by the {path weights} $d(p_{ji})= c_{jj} c_{k_1j}\cdots c_{ik_{\ell-1}}$ for each path $p_{ji}=[j\to k_1\to\cdots \to k_{\ell}=i]$.
The entries of $A$ are defined for $i\in\{1,\dots,{D}\}$ by
\begin{align*}
a_{ij}=\underset{p_{ji}\in P_{ji}}{\bigvee}d(p_{ji}) \mbox{ for } j\in {\An}(i),\quad a_{ij}=0 \mbox{ for }  j\in V\setminus {\An}(i),\quad a_{ii}=c_{ii},
\end{align*}
and a path $p_{ji}$ from $j$ to $i$ such that $a_{ij}$ equals $d(p_{ji})$ is called max-weighted path (mwp).

We use the following notion throughout the paper.

\begin{definition}\label{mwpair}
A pair of nodes $(i,j)$ {is a max-weighted pair}, 
if for all $k\in \An(i)\cap\An(j)$ there are max-weighted paths $k\rightsquigarrow j \rightsquigarrow i$; if so, we write $(i,j)\in$ {\em MWP}. 
Note that $k=j$ is possible, so this includes max-weighted paths $j\rightsquigarrow i$.
\end{definition}

\subsection{Organisation}\label{sec:orga}

In Section~\ref{sec:RMLM} we give conditions that ensure that a partially observed vector of node variables can be modelled as an RMLM, and present a graphical algorithm to construct this model. 
In Section~\ref{sec:RV} we provide criteria based on max-weighted paths in a regularly varying RMLM to ensure the representation of the observed pairs of node variables as bivariate RMLMs, and derive conditions for the identification of these paths.
In Section~\ref{sec:est} we translate previous theoretical results into a statistical algorithm to detect max-weighted paths {between} observed node variables, and apply it to nutrient intake data. 
Section~\ref{funcCLT} provides a functional central limit theorem for random sample sizes, which {results from a} two-step thresholding procedure.

The Supplement has seven appendices. Appendix~\ref{sec:3.4} extends the methodology for {bivariate} regularly varying RMLMs by leveraging the 
identified ancestors among the observed nodes.
Appendix~\ref{Ap:ProofsRMLM} contains proofs of the main theorems of Sections~\ref{sec:RMLM},~\ref{sec:RV} and Appendix~\ref{sec:3.4}.
Appendix~\ref{sec:ARV} summarises standard definitions and results on regular variation, and specifies those relevant for RMLMs.
Consistency and asymptotic normality of the estimated scalings and extremal dependence measures are proved in Appendix~\ref{sec:stat}, where we propose an intermediate thresholding procedure that requires an apparently novel functional CLT~for a random sample.
In Appendix~\ref{scalest} these results are used to estimate the inputs of our Algorithm~\ref{datdalg2}.
Appendix~\ref{sec:performance} investigates its performance in a simulation study based on true and false positive rates for the estimated max-weighted paths enriched by various categories of causal dependence. 
Appendix~\ref{Ap:sim} studies the sensitivity of our algorithm numerically.

\section{Constructing RMLMs from DAGs with hidden nodes }\label{sec:RMLM}

We now investigate the problem posed in Section~\ref{sec:1.1} and motivated by the example of Section~\ref{ex:intro}: given an RMLM $\bsx$ on a DAG $\D$ with $D$ nodes and certain nodes hidden, can we construct a lower-dimensional RMLM on the observed {$d<D$} nodes? We shall see that this may be possible under certain assumptions {on the max-weighted paths.}
{All proofs are deferred to Appendix~\ref{Ap:ProofsRMLM}.}
 
 \subsection{Max-weighted paths and RMLMs on DAGs with hidden nodes}\label{sec:2.1}
 
 Suppose that the true DAG has $|V|={D}$ nodes, of which {only} $d<D$ are observed. In this section we denote the set of observed nodes  and its complement by $O\subset V$ and $O^c$, respectively, and write $\bsx_O$ for the vector of observed variables. 
 
 A key tool is the minimal representation of the components of $\boldsymbol{X}$ \citep[Theorem~6.7]{gk}, in which we replace the arbitrary subset $U$ by the subset $O\subset V$ of observed nodes and make {some} modifications. 
 For $i\in V$ we reformulate the sets as 
 \begin{align}\label{rmlmequat}
{\an^O(i)} &= \Big\{j\in\an(i)\cap O : a_{ij}> \bigvee_{k\in O\cap \textrm{an}(i)\cap \des(j)} \frac{a_{ik}a_{kj}}{a_{kk}}\Big\},\nonumber\\
 \An^{O^c}(i) &= \{i\} \cup \Big\{j\in\an(i)\cap O^c : a_{ij}> \bigvee_{k\in O\cap \an(i)\cap \des(j)} \frac{a_{ik}a_{kj}}{a_{kk}}\Big\}.
 \end{align} 
 The set $\an^O(i)$, originally defined as $\An_{\rm low}^O(i)$ for $i\notin O$ in equation~(6.3) of \citet{gk}, contains the lowest {max-weighted} ancestors of $i$ in $O$, i.e., those nodes $j$ such that no max-weighted path from $j$ to $i$ passes through any other node in $O$. 
 However, $\An_{\rm low}^O(i)=\{i\}$  for $i\in O$, and as we are interested in the innovations, we extend the use of $\an^O(i)$ in~\eqref{rmlmequat} for all $i\in V$ to avoid such trivial representations, so that $i\notin \an^O(i)$.
 Analogously, $\An^{O^c}(i)$, denoted by $\An_{\rm nmw}^{O}(i)$ in \citet[Theorem~6.7]{gk}, consists of the lowest max-weighted ancestors of $i$ in $O^c$, {including now $i$ itself.}
 This allows us to express the variable for  node $i$ in terms of the minimum number of observed ancestors and of innovations, giving the minimal representation 
 \begin{align}\label{minrep}
 X_{i} &= \bigvee_{k\in \an^{O}(i)} \frac{a_{ik}}{a_{kk}} X_k \vee \bigvee_{k\in \An^{O^c}(i)} a_{ik}Z_k,\quad {i\in V.}
 \end{align} 
 As we are interested in this representation for $\boldsymbol{X}_O$ only, we use it for $i\in O$; as a result, the set $O$ is the node set of a unique DAG $\D^O$ \citep[Theorem~5.4]{gk}, and the edges of $\D^O$ correspond to the relations in~\eqref{minrep}. 


The following example illustrates the minimal representation~\eqref{minrep}.

 \begin{example}\label{ex:2.1}
Consider $\mathcal{D}_2$ of Figure~\ref{introex} with confounder 3 of nodes 1 and 2.
  
If the path from $3\to 2\to1$ is max-weighted, {the path $3\to 1$ is irrelevant} and~\eqref{minrep} allows us to write $(X_1,X_2)$ as an RMLM with 
 	$$X_1=a_{11}Z_1\vee \frac{a_{12}}{a_{22}}X_2, \quad {X_2= a_{22}Z_2\vee a_{23}Z_3}=f_{23}(Z_2, Z_3).$$
Here $Z_1$ and $f_{23}(Z_2, Z_3)$ depend on the independent original innovations in the two equations and are both exogenous for $(X_1, X_2)$, which represents an RMLM on the smaller DAG $(\{1,2\}, 2\to 1)$. 
  
  If the path $3\to 2\to1$ is not max-weighted, then equation~\eqref{simp_ex} implies that the new innovations $f_{13}(Z_1, Z_3)\coloneqq a_{11}Z_1\vee a_{13}Z_3$ and $f_{23}$ both depend on $Z_3$.  This contradicts the independence of the innovations, so {$(X_1, X_2)$} cannot be represented as an RMLM.
 \end{example}

The next theorem characterises when a vector $\bsx_O$ of observed node variables can be represented as an RMLM. Part (i) characterises the source nodes and part (ii) their descendants.   
The proof, given in Appendix~\ref{Ap:ProofsRMLM}, uses representation~\eqref{minrep}.

{We recall Definition~\ref{mwpair} and define two further sets that restrict max-weighted paths to certain subsets of nodes. For $B, C\subseteq V$ we write}
{\begin{align}\label{mwps}	
	{\rm MWP}&= \{(i,j):  \An(i)\cap\An(j)\neq\emptyset \text{\,and\,} \forall k\in \An(i)\cap\An(j), \exists 
    \text{\, an mwp\,}  k\rightsquigarrow j\rightsquigarrow i\},\nonumber\\ 
	{\rm MWP}(B)&= \{(i,j):  \An(i)\cap\An(j)\cap B\neq\emptyset \text{\,and\,}\forall k\in \An(i)\cap\An(j)\cap B, \exists  \text{\, an mwp\,} k\rightsquigarrow j\rightsquigarrow i\},\nonumber\\
	{\rm MWP}^C(B)&= \{i : \An(i)\cap B\neq \emptyset \text{\,and\,} \forall k\in \An(i)\cap B, \exists\,  j\in C \text{\,and an mwp\,} k\rightsquigarrow j\rightsquigarrow i\}.
	\end{align}}
 \vspace*{-.1mm}
 Note that ${\rm MWP}={\rm MWP}(V)$, implying that if $(i, j)\in {\rm MWP}$, then $(i, j)\in {\rm MWP}(B)$ for all $B\subseteq V$ such that $\An(i)\cap \An(j)\cap B\neq \emptyset$. This proves useful for the tree graphical structures in Section~\ref{trees}. 
{For the sets $C=\{j\}$ and $B=\{k\}$, we simply write  ${\rm MWP}(B)={\rm MWP}(k)$ and ${\rm MWP}^C(B)={\rm MWP}^j(k)$.}
 
\bthe\label{prop.dat} 
Let $\boldsymbol{X}\in\R_+^D$ be an RMLM {on a DAG $\D$} with coefficient matrix $A\in \mathbb{R}_+^{D\times D}$. 
{For observed nodes $O\subset V$,} the vector $\bsx_O$ 
can be represented as an RMLM if and only if:
\hspace{-10mm} 
\begin{enumerate}
 	\item[(i)]  
for every $\ell\in O$ such that $O\cap \an(\ell)=\emptyset$, both\\
(a) $(i, \ell)\in {\rm MWP}$ for all $i\in O\cap\emph{de}(\ell)$, and \\
(b) $\an(\ell)\cap\an(j)=\emptyset$ for all $j\in O\cap\Des(\ell)^c$ hold;  \\
 let $V_0^O$ denote the set of nodes $\ell$ satisfying these properties, then 
 	\item[(ii)] 
for every $u\in O^c\setminus\emph{an}(V_0^O)$, and any nodes $i,j\in O$ such that  $u\in \an(i)\cap\an(j)$, then\\
(a) if $j\in \an(i)$, either $i\in{\rm MWP}^{j}(u)$ or there exists $k\in\an(j)\cap O$ such that $i, j\in{\rm MWP}^{k}(u)$; and\\ 
(b) if $j\notin\an(i)$ and $i\notin\an(j)$, there exists $k\in\an(i)\cap\an(j)\cap O$ such that $i, j\in{\rm MWP}^{k}(u)$.
 \end{enumerate}
 \ethe

Theorem~\ref{prop.dat} covers all three DAGs in Figure~\ref{introex}.
For $\D_1$ there is a unique (hence max-weighted) path {$3\rightsquigarrow 2\rightsquigarrow 1$} and Theorem~\ref{prop.dat}(i, ii) is trivially satisfied.
When node 3 is a confounder, then for $\D_2$, provided {$3\rightsquigarrow 2\rightsquigarrow 1$} is max-weighted, the vector $(X_1,X_2)$ can be represented as an RMLM, but for $\D_3$ no such representation is possible. 

{We now implement the conditions of Theorem~\ref{prop.dat} into a graphical algorithm,
and illustrate how it can be employed to model $\bsx_O$ as an RMLM. 
If this is impossible, Algorithm~\ref{datdalgrmlm} identifies a subset $K$ of $O$ such that $\bsx_K$ can be modelled as an RMLM. The general result is stated in Proposition~\ref{graphalgprop}.
Algorithm~\ref{datdalgrmlm} cannot be applied in a data analysis as it needs the~matrix~$A_O\in\mathbb{R}^{d\times D}$ as input.
More precisely, $A_O$ contains $D$-dimensional rows indexed in $O$ from the original~coefficient matrix~$A$.
Note that $A_O$ determines the marginal distribution of the observed variables and also~contains information about the maximum path weights from the hidden nodes to the observed ones.
We emphasize that Algorithm~\ref{datdalgrmlm} serves as a transparent reformulation of Theorem~\ref{prop.dat}, which outputs a node set $K\subseteq O$ and the corresponding DAG $D^K$, which we call the minimal representation DAG of the RMLM $\bsx_K$.
For better understanding, we also highlight in blue the connections of the conditions of Algorithm~\ref{datdalgrmlm} to the respective~ones~in~Theorem~\ref{prop.dat}.}

\begin{algorithm}[t]
	\caption{Identification of a set of nodes in $O$ that can be modelled as an RMLM, and of their DAG} 
	\label{datdalgrmlm}
	\textbf{Input}: $\bs{z}=\boldsymbol{0}\in\mathbb{R}^d$, { $A_O\in\mathbb{R}^{d\times D}$}, $\D=(O, E=\emptyset)$, $K=\emptyset$  \\
	\textbf{Output}: Well-ordered set $K$ {of nodes forming an RMLM}, and the minimal representation DAG $\D^K$ \\
	\textbf{Procedure:} 
	\begin{algorithmic}[1]
		\State \textbf{for} $j \in O$
		 	\State \hspace{4mm}  \textbf{for} $i \in O\setminus \{j\}$
		 	 \State \hspace{11mm} \textbf{if} $(i, j)\in {\rm MWP}$, \textbf{set} $z_j= z_j+1$, \textbf{add } $j\to i$ \textbf{in} $E$,  \quad[\textcolor{blue}{Theorem~\ref{prop.dat} (i)(a)}]
		\State \hspace{14mm} \textbf{else if} $\An(i)\cap \An(j)=\emptyset$, \textbf{set} $z_j = z_j+1$, \quad[\textcolor{blue}{Theorem~\ref{prop.dat} (i)(b)}]
		
	\State \hspace{4mm} \textbf{end for}
		\State \textbf{end for}
		\State $K\leftarrow\text{select one index from } \underset{j\in O}{\rm argmax}~ z_j$;\quad {$\bs{z}\leftarrow\boldsymbol{0}$}
		\State $V^K\leftarrow O\cap ({\rm MWP}^K({\An(K)})\cup \{r: \An(r)\cap \An(K)=\emptyset\})\setminus
		K$
			\State \textbf{while} $V^K\neq \emptyset$
			\State\hspace{8mm} \textbf{for} $j \in V^K$
			\State\hspace{8mm} \hspace{3mm}  \textbf{for} $i \in V^K\setminus \{j\}$

			\State\hspace{8mm} \hspace{8mm} \textbf{if} $(i, j)\in {\rm MWP}(\An(K)^c)$, \textbf{set} $z_j=z_j+1$,  \textbf{add} $j\to i$ \textbf{in} $E$,\quad  [\textcolor{blue}{Theorem~\ref{prop.dat} (i, ii)(a)}]
						\State\hspace{8mm} \hspace{11mm} \textbf{else if} $\An(i)\cap \An(j)\cap \An(K)^c=\emptyset$, \textbf{set} $z_j=z_j+1$, \quad [\textcolor{blue}{Theorem~\ref{prop.dat} (i, ii)(b)}]
			\State\hspace{8mm} \hspace{3mm} \textbf{end for}
			\State\hspace{8mm} \textbf{end for}
			\State\hspace{8mm} $k\leftarrow \text{select an index from  }\underset{j\in V^K}{\rm argmax}~ z_j$; 
\text{update $K$ by adding $k$ as its first element, }$K\leftarrow(k, K)$  \State\hspace{8mm} $V^K\leftarrow O\cap ({\rm MWP}^K({\An(K)})\cup \{r: \An(r)\cap \An(K)=\emptyset\})\setminus K$;
 \State\hspace{8mm} {$\bs{z}\leftarrow\boldsymbol{0}$}
	\State \textbf{end while}
				\State  \textbf{for} $i\in K$ ;
				\State  \hspace{5mm}\textbf{for} $j\in \an(i)\cap K$ ;
				\State  \hspace{10mm}\textbf{for} $k\in \an(j)\cap K$ ;
					\State  \hspace{14mm} \textbf{if}  $(i,j)\in {\rm MWP}(\An(k)\setminus (\an(k) \cap K))$, \textbf{remove} $k\to i$ \textbf{from} E		
				\State  \hspace{10mm}\textbf{end for} 
				\State  \hspace{5mm}\textbf{end for} 
				\State  \textbf{end for} 
                    \State  \textbf{Return} $K$, $\D^K=(K, E)$
	\end{algorithmic}
\end{algorithm}\vspace*{-.25cm}
\begin{proposition}\label{graphalgprop}
The observed node variables $\bsx_O$ can be modelled as an RMLM if and only if Algorithm~\ref{datdalgrmlm} outputs a causal order of $O$ and {the minimal representation} DAG $\D^O$. 
If only subsets of node variables~in~$O$ can be modelled as RMLMs, Algorithm~\ref{datdalgrmlm} outputs a well-ordered set $K\subset O$ of such nodes with {the minimal representation} DAG~$\D^K$.
\end{proposition}

The set $K\subset O$ produced by Algorithm~\ref{datdalgrmlm} is augmented by one element at every iteration {step} and also provides a causal ordering. {Although it seems natural to write the ordered node set $K$ as a vector, we write it as a set, since we often apply set manipulations to $K$.}
The {successive} augmentation of $K$ avoids cases when $z_i=z_j$ for pairs $(i,j)$ that share hidden confounders, {as for node 9 }in Figure~\ref{partialdag}. 


\begin{example}	\label{exdag} 
{We illustrate Algorithm~\ref{datdalgrmlm} when $\bsx_O$ can be modelled as an RMLM.}\\
Suppose that an RMLM is supported on the DAG in Figure~\ref{fig:M1} with $D=13$ nodes, of which $d=8$ are observed and five are unobserved, that $X_{13}=Z_{13},\, X_{12}=Z_{12}$ and $X_{11}=Z_{11}$, and that the paths $13 \to 10 \to 6$, {$13 \to 10 \to 5\to 3$} and  $7 \to 4 \to i$ ($i\in\{1, 2\}$) are max-weighted. We apply Algorithm~\ref{datdalgrmlm}. 
 	\begin{figure}[ht]
 		\centering
 		\resizebox{7.cm}{4.2cm}{\begin{tikzpicture}[
 			> = stealth,
 			shorten > = 1pt, 
 			auto,
 			node distance = 2cm, 
 			semithick 
 			]
 			\tikzstyle{every state}=[
 			draw = black,
 			thick,
 			fill = white,
 			minimum size = 4mm,scale=1
 			]
 			
 			\node[state] (12)[dotted] {\footnotesize $12$};
 			\node[state] (13)[dotted] [left=1.5cm of 12] {\footnotesize $13$};
 			\node[state] (11)[dotted] [right=2.2cm of 12] {\footnotesize $11$};
 			\node[state] (10) [below =.7cm  of 13] {\footnotesize $10$};
 			\node[state] (9) [below=.7cm  of 12] {$9$};
 			\node[state] (8) [below=.7cm  of 11] {$8$};
 			\node[state] (6) [below left=.7cm and .7cm  of 10] {$6$};
 			\node[state, dotted] (5) [below right=.7cm and .7cm  of 10] {$5$};
 			\node[state, dotted] (7) [right=.8cm   of 9] {$7$};
 			\node[state] (4) [below =.7cm  of 7] {$4$};
 			\node[state] (2) [below left=.7cm and 1cm  of 4] {$2$};
 			\node[state] (1) [below right=.7cm and 1cm  of 4] {$1$};
 			\node[state] (3) [below right =.7cm  and .73 cm of 6] {$3$};
 			
 			\path[->][blue, dotted] (7) edge node {} (4);
 			\path[->][blue, dotted] (13) edge node {} (10);
 			\path[->][blue, dotted] (5) edge node {} (3);
 			\path[->][blue] (6) edge node {} (3);
 			\path[->][blue] (4) edge node {} (1);
 			\path[->][blue] (4) edge node {} (2);
 			\path[->][blue, dotted, bend right=45] (13) edge node {} (6);
 			\path[->][blue, dotted, bend right=0] (7) edge node {} (2);
 			\path[->][blue, dotted, bend left=0] (7) edge node {} (1);
 			\path[->][blue, dotted] (9) edge node {} (7);
 			\path[->][blue, dotted] (9) edge node {} (5);
 			\path[->][blue] (10) edge node {} (6);
 			\path[->][blue, dotted] (10) edge node {} (5);
 			\path[->][blue, dotted] (11) edge node {} (8);
 			\path[->][blue, dotted] (12) edge node {} (9);
 			\path[->][blue, dotted] (8) edge node {} (7);
 			\end{tikzpicture}}
 		\caption{Partially observed DAG with $13$ nodes, of which the five dotted nodes are hidden.} \label{fig:M1}\vspace{-.5cm}
 	\end{figure}

-- Start with $j=10$ {(line 1)}. 
Now check $i\in O\setminus\{10\}$.
{Running through these $i$, we obtain
$(i,10)\in {\rm MWP}$ for $i\in\{6, 3\}$, so line 3 gives $z_{10}=1$ and then $z_{10}=2$, and we add the edges $10\to 6$ and $10\to 3$.
As $\An(10)\cap\An(i)=\emptyset$ for $i\in\{9, 8, 4, 2, 1\}$, such pairs $(i,10)$ satisfy the condition of line 4, giving consecutively $z_{10}=3$, $z_{10}=4$, $z_{10}=5$, $z_{10}=6$ and $z_{10}=7$ (line 5);}

-- for $j\in\{8, 9\}$ {(line 1)}, similar reasoning gives $z_9=z_8=7$ (line 5);

{note that by Theorem~\ref{prop.dat} (i), $\{10,9,8\}\subseteq V_0^O$;}

-- for $j\notin\{10, 9, 8\}$ we find $z_j<7$; for instance, $z_6=5$ (in line 5) because $\An(6)\cap\An(i)=\emptyset$ for $i\in\{9, 8, 4, 2, 1\}$, {but $(3,6), (10,6)\notin {\rm MWP}$; hence, $V_0^O=\{10,9,8\}$;}

-- {as $\max_{j\in O} z_j= 7$, select one node out of $\{10,9,8\}$ and we take $K=\{10\}$ (line 7)}, then $V^K=O\setminus\{10\}$ {(line 8)};
a similar analysis as above gives in lines 12 and 13  that $z_9=z_8=z_6=6$, while $z_j<6$ for all remaining nodes in $V^K$; for instance, $z_4=2$ because the conditions in lines 12 and 13 are satisfied with $\An(6)\cap\An(4)\setminus\{10\}=\emptyset$ and $\An(3)\cap\An(4)\setminus\{10\}=\emptyset$;
{line 16 then  updates $K$ by first selecting $k$ from $\{9, 8, 6\}$ and then adding it as the first element of $K$;} 

-- suppose we have identified $K=\{8, 9, 10\}$ (line 16), so $V^K=\{1, 2, 3, 4, 6\}$ {(line 17)};
we continue the iteration in lines 12 and 13 giving $z_1=2$, since $\An(i) \cap \An(1) \cap An(K)^c = \emptyset$ for  $i\in \{6, 3\}$, and  $z_2=2$, $z_3=3$;
{this continues with} $z_4=4$, and $z_6=4$, since $(3,6)\in {\rm MWP}(\An(K)^c)$, because $\An(3)\cap \An(6)\cap \An(K)^c=\{6\}$ (line~12) and, similarly, $\An(i)\cap \An(6)\cap \An(K)^c=\emptyset$ for $i\in\{1, 2, 4\}$ (line~13);

-- suppose now that we have identified $K=\{4, 6, 8, 9, 10\}$ and $V^K=\{1 , 2, 3\}$; similar steps to those above give $z_3=z_2=z_1=2$;

-- for the final iteration, suppose that $K=\{2,3, 4, 6, 8, 9, 10\}$ and $V^K=\{1\}$; then $z_1=0$ (lines 12 and 13), which indicates that $1$ is the last node that can be added to the DAG;  the RMLM can be extended no further, since adding $1$ gives $K\{1, 2, 3, 4, 6, 8, 9, 10\}$ and $V^K=\emptyset$.

The DAG of the RMLM, found by applying lines 18--24 of Algorithm~\ref{datdalgrmlm}, is shown in Figure~\ref{fig:M2}.
 	\begin{figure}[H]
 		\centering
 		\resizebox{5.2cm}{2.7cm}{\begin{tikzpicture}[
 			> = stealth,
 			shorten > = 1pt, 
 			auto,
 			node distance = 2cm, 
 			semithick 
 			]
 			\tikzstyle{every state}=[
 			draw = black,
 			thick,
 			fill = white,
 			minimum size = 4mm,scale=1
 			]
 			
 			\node[state] (10) {\footnotesize{$10$}};
 			\node[state] (9) [right=1cm  of 10] {$9$};
 			\node[state] (6) [below left=.7cm   of 10] {$6$};
 			\node[state] (4) [below right =.7cm  of 9] {$4$};
 			\node[state] (8) [above right=.7cm  of 4] {$8$};
 			\node[state] (2) [below left=.7cm  of 4] {$2$};
 			\node[state] (1) [below right=.7cm  of 4] {$1$};
 			\node[state] (3) [below right =.7cm   of 6] {$3$};
 			
 			\path[->][blue] (6) edge node {} (3);
 			\path[->][blue] (4) edge node {} (1);
 			\path[->][blue] (4) edge node {} (2);
 			\path[->][blue] (9) edge node {} (3);
 			\path[->][blue] (9) edge node {} (4);
 			\path[->][blue] (10) edge node {} (6);
 			\path[->][blue] (10) edge node {} (3);
 			\path[->][blue] (8) edge node {} (4);
 			\end{tikzpicture}}
 		\caption{DAG $\D^O$ with the eight observed nodes.} \label{fig:M2}\vspace{-.5cm}
 	\end{figure}
 \end{example}

The dimension of the DAG output by Algorithm~\ref{datdalgrmlm} need not be the highest attainable. 

\begin{example}\label{subsets}
{We illustrate Algorithm~\ref{datdalgrmlm} when only a subvector of $\bsx_O$ can be modelled as an RMLM.}\\
	Consider the nine-dimensional RMLM supported on the DAG in Figure~\ref{partialdag}:
	
	-- in the initial step of the algorithm we have $z_j=0$ for $j\in\{1, 2, 3, 4, 5\}$, $z_7=2$, because $(4, 7), (5,7)\in {\rm MWP}$, and $z_8=3$, by similar arguments;
	
	-- letting $K=\{8\}$ gives $V^K=\{1, 2, 3\}$, excluding $7$ and its descendants, because $9$ is a hidden confounder for both $7$ and $8$. At this step, however, we have $z_j=0$ for $j\in\{1, 2, 3\}$. If we were to pick, say $1$, then $K=\{1, 8\}$, but then $V^K=\emptyset$, because  $6$ is hidden and non-exogenous among $1, 2$ and $3$. Thus, we can only obtain the two-dimensional RMLMs with nodes $\{1,8\}$, $\{2, 8\}$ and $\{3, 8\}$;
	
	-- by contrast, had we chosen $K=\{7\}$, then $V^K=\{4,5\}$ and $z_j=1$ for both nodes in $V^K$, so we would have obtained a DAG with nodes $\{4 ,5, 7\}$.
	
	Having identified one possible subset $K$ for the RMLM, we could repeat the procedure with the nodes in $O\setminus K=\{4, 5, 7\}$. This would yield a second RMLM consisting of nodes $\{4 ,5, 7\}$.	
\begin{figure}[ht]
	\centering
	\resizebox{5.3cm}{3cm}{\begin{tikzpicture}[
		> = stealth,
		shorten > = 1pt, 
		auto,
		node distance = 2cm, 
		semithick 
		]
		\tikzstyle{every state}=[
		draw = black,
		thick,
		fill = white,
		minimum size = 4mm,scale=1
		]

		\node[state, dotted] (9) {$9$};
		\node[state] (8) [ right=.7cm  of 9] {$8$};
		\node[state] (7) [left=.7cm  of 9] {$7$};
		\node[state, dotted] (6) [below=.6cm  of 8] {$6$};
		
		\node[state] (5) [ below right=.7cm and .4cm  of 7] {$5$};
		\node[state] (4) [  below left=.7cm and .4cm  of 7] {$4$};
		\node[state] (3) [ below =.54cm  of 6] {$3$};
		\node[state] (2) [ below right=.7cm and .7cm of 6 ] {$2$};
		\node[state] (1) [ below left=.7cm  and .7cm of 6 ] {$1$};
		
		\path[->][blue, dotted, bend right=0] (9) edge node {} (8);
		\path[->][blue, dotted, bend left=0] (9) edge node {} (7);
		\path[->][blue, dotted] (8) edge node {} (6);
		\path[->][blue, bend left=0] (7) edge node {} (5);
		\path[->][blue, bend left=0] (7) edge node {} (4);
		\path[->][blue, dotted, bend left=0] (6) edge node {} (3);
		\path[->][blue, dotted, bend left=0] (6) edge node {} (2);
		\path[->][blue, dotted, bend left=0] (6) edge node {} (1);
		
		\end{tikzpicture}}
	\caption{Partially observed DAG, of which the two dotted nodes $6$ and $9$ are hidden.} \label{partialdag}\vspace{-.7cm}
\end{figure}
\end{example}
The following remark elucidates the relationship between Theorem~\ref{prop.dat} and representation~\eqref{minrep}, which is also the main tool used to prove the former.

\begin{remark}\label{ex:rel} We now use the representation~\eqref{minrep} to show the implications of conditions (i) and (ii) in Theorem~\ref{prop.dat} for the DAGs in Figures~\ref{fig:M1} and~\ref{partialdag}.   Suppose that none of the paths $13 \to 10 \to 6$, $13 \to 10 \to 5\to 3$, and  $7 \to 4 \to i$, $i\in\{1,2\}$, in Figure~\ref{fig:M1} is max-weighted. 

For each of the conditions of Theorem~\ref{prop.dat} we select a pair of nodes, and show that the innovations are non-exogeneous:
\\[2mm]
\,
(i)(a) and Figure~\ref{fig:M1}, with the pair $(6,10)$ and 10 a  candidate for a source node. Representation~\eqref{minrep} yields
\begin{align*}
X_{10} = a_{10,10}Z_{10}\vee a_{10,13}Z_{13},\quad
X_6 = \frac{a_{6,10}}{a_{10,10}} X_{10} \vee a_{66}Z_6\vee  a_{6,13}Z_{13},
\end{align*}
which both have innovations depending on $Z_{13}$, so they cannot be represented by an RMLM;\\[2mm]
 \,
(i)(b) and Figure~\ref{partialdag}, with $\{7,8\}$  candidates  for source nodes. Here $7\notin\Des(8)$ , $8\notin\Des(7)$, and $\An(7)\cap\An(8)\cap O^c =\{9\}$, so~\eqref{minrep} yields
\begin{align*}
X_7 =a_{77}Z_7\vee a_{79}Z_{9},\quad
X_8 =a_{88}Z_8\vee a_{89}Z_{9},
\end{align*}
which both have innovations depending on $Z_{9}$, so they cannot be represented by an RMLM;\\[2mm]
\,
(ii)(a) and Figure~\ref{fig:M1}: let $u=7$ and consider the pair $(4,2)$. By representation~\eqref{minrep},
\begin{align}\label{x2}
X_4 = \frac{a_{48}}{a_{88}}X_8\vee\frac{a_{49}}{a_{99}}X_9\vee  a_{44}Z_4\vee  a_{47}Z_7, \quad 
X_2 =  \frac{a_{24}}{a_{44}}X_4\vee  \frac{a_{28}}{a_{88}}X_8 \vee \frac{a_{29}}{a_{99}}X_9\vee  a_{22}Z_2\vee a_{27}Z_7,
\end{align}
and as both have innovations depending on $Z_7$, they cannot be represented as an RMLM. This contradicts (ii)(a) with $u=7$ and $4\in\an(2)$;\\[2mm]
\,
(ii)(b) and Figure~\ref{fig:M1}. Let $u=7$ and consider the pair $(1,2)$, which has no ancestral relation.  By~\eqref{minrep}, 
\begin{align*}
X_1 =  \frac{a_{18}}{a_{88}}X_8\vee\frac{a_{19}}{a_{99}}X_9\vee \frac{a_{14}}{a_{14}}X_4\vee{a_{11}}Z_1\vee  a_{17}Z_7,
\end{align*}
with $X_2$ given in~\eqref{x2}, which both depend on $Z_7$, again contradicting (ii)(b) of Theorem~\ref{prop.dat}.
\end{remark}

\subsection{MWP for directed trees}\label{trees}

We now focus on two simple but important graphical structures, namely directed trees and directed spanning trees, in which a unique path connects any causally dependent pair of nodes. The identification of the set MWP turns out to be crucial in such instances, because the edges of the DAG induced by the minimal representation~\eqref{minrep} of the RMLM can be identified directly from the pairs in MWP: if $(i, j)\in {\rm MWP}$, then $(i, j)\in {\rm MWP}(B)$ for all $B\subseteq V$ such that the set $\An(i)\cap \An(j)\cap B\neq\emptyset$.

\begin{remark}\label{extree}
Suppose that for $O_t\subseteq O$ the RMLM $\bs X_{O_t}$ is supported on {the} directed tree $\D^{O_t}$. 
\begin{enumerate}
\item[(i)]
As there is a unique path between pairs $(i,j)$ such that $j\in \An(i)$, all such pairs lie in ${\rm MWP}$. 
For any other pair there exists a well-ordered set $K$ at some iteration of Algorithm~\ref{datdalgrmlm}, with $i,j\notin K$, and  $\an(i)\cap\an(j)\subset K$, such that $\An(i)\cap\An(j)\cap \An(K)^c=\emptyset$. 
\item[(ii)]
The edges of the tree can be identified as follows. Let $i$, $j$, $k$ be such that $j\in \an(i)$ and $k\in \an(j)$. Then, by path uniqueness in a tree, $(i, j), (i, k), (j,k) \in 
{\rm MWP}$, and these pairs of nodes correspond to adding the edges $j\to i$, $k\to i$ and $k\to j$ in the first iteration of Algorithm~\ref{datdalgrmlm}. However, the membership  of $(i, j)$ in ${\rm MWP}$ implies that there are max-weighted paths $k\to j\to i$ for all $k\in\an(j)\cap\an(i)$.  Thus, following the minimal representation~\eqref{minrep}, we remove the edge $k\to i$, accomplished in line 23 of Algorithm \ref{datdalgrmlm}.
\end{enumerate}
\end{remark}
This remark shows that the properties of a tree and of the set {\rm MWP} imply that we need only that $(i, j)\in {\rm MWP}$ to verify the condition in line 12 of Algorithm~\ref{datdalgrmlm}, and likewise for the set in line 23. 

A directed spanning tree has a unique path between any pair of nodes and all edges directed towards a single, sink, node. We call such a graphical structure a (sink-)directed spanning tree.

\begin{proposition}\label{dst}
    A vector $\bsx_{O_t}$ is an RMLM supported on a sink-directed spanning tree of $O_t\subseteq O$ if and only if 
    \begin{enumerate}
        \item[(i)]\, for all pairs of nodes $i,j\in O_t$, either $(i,j)$ or $(j,i)$ lie in {\rm MWP}, or  $\An(i)\cap\An(j)=\emptyset$;
        \item[(ii)] there exists $i\in O_t$ such that $(i, j)\in {\rm MWP}$ for all $j\in O_t\setminus\{i\}$.
    \end{enumerate}
    The edges of such a tree can be identified from the pairs in {\rm MWP}.
\end{proposition}

\begin{proof}
    The equivalence is a direct consequence of the definition of a directed spanning tree. Its edges can be identified in a similar fashion to the procedure outlined in Remark~\ref{extree}.
\end{proof}

Finally, we state a simple consequence of Theorem~\ref{prop.dat} for a sink-directed spanning tree.
 \begin{corollary}
If the RMLM  $\boldsymbol{X}=(X_1,\dots,X_D)\in\R_+^D$  is supported on a well-ordered sink-directed spanning tree and we observe $(X_{o_1},\ldots,X_{o_d})$ for $d<D$, ordered so that $o_i<o_j$ for $i<j$, then $(X_{o_1},\ldots,X_{o_d})$ is also an RMLM.
 \end{corollary}
 
\begin{proof}
	Every node of such a tree has a single child, so a hidden node is never a confounder, and therefore (i) and (ii) in Theorem~\ref{prop.dat} are both satisfied. If the sink node is also observed, then the observed RMLM is also supported on a sink-directed spanning tree.
\end{proof}

\section{Regular variation of a recursive max-linear vector}\label{sec:RV}

In this section we connect the theory in Section~\ref{sec:RMLM} with that of regular variation. 
{We focus on regularly varying RMLMs with hidden nodes and investigate whether the observed nodes have again a representation as an RMLM. Section \ref{MLreparamet} confirms this, if Theorem~\ref{prop.dat} applies.
As the latter involves MWP, in Section~\ref{sec:MWP} we characterise max-weighted pairs $(i,j)$ for regularly varying RMLMs, resulting in a bivariate RMLM for $(X_i,X_j)$. We generalise such results in Appendix~\ref{sec:3.4} to extend an RMLM $\bsx_K$ for $K \subset V$ to other observed nodes.}
  
\subsection{Extremal dependence}\label{sec_exdep}

In the rest of the paper we suppose that the vector of innovations $\boldsymbol{Z}\in \R_+^D$ is regularly varying with index $\alpha>0$, written $\boldsymbol{Z}\in{\rm RV}_+^{D}(\alpha)$, and that it has independent and standardised components, with $n\mathds{P}(n^{-1/\alpha} Z_i>z)\to z^{-\alpha}$ ($z>0$) as $n\to\infty$ for all $i\in\{1,\dots,{D}\}$.
Then RMLMs in~\eqref{Rmlmequat} belong~to~the more general class of max-linear models 
with independent regularly varying innovations, which has a long history; see, e.g., \citet{Wang2011}. 
Such models are multivariate regularly varying and have a finite discrete {angular measure} $H_{\bsx}$ on the non-negative unit sphere $\Theta_+^{{ {D}}-1}=\{\boldsymbol{\omega}\in \mathbb{R}^{ {D}}_+:\norm{\boldsymbol{\omega}}=1\}$, for some norm $\norm{\cdot}$. 
For completeness we define multivariate regular variation, its angular representation and its angular measure in Appendix~\ref{sec:ARV}, referring to \citet{sres,ResnickHeavy} for more details.

Our results depend on the following measure of extremal dependence, introduced in Propositions~3 and~4 of \citet{lars}; see also \citet[Section 4]{cooley} and \citet[Section 2.2]{KK}. 


We now derive certain properties of the finite angular measure $H_{\boldsymbol{X}}$ and of its (non-normalised) second order moments.
Let $f\colon\Theta_+^{{d}-1}\to\mathbb{R}_+$ be a function, continuous outside a null set, bounded and compactly supported. 
Then the following moment exists \cite[eq. (3)]{lars} and, to keep notation simple, we define 
\begin{align}\label{p2empdist}
\mathbb{E}_{{H}_{\bsx}}[f(\boldsymbol{\omega})] := \int_{\Theta_+^{{d}-1}} f(\boldsymbol{\omega}) {\rm d}{ H}_{\bsx}(\boldsymbol{\omega}).
 \end{align}
In Appendix~\ref{sec:ARV} we will standardise the angular measure $H_{\bsx}$ to a probability measure such that the expectation notation for the finite integral in \CK{\eqref{p2empdist}} makes sense.
Defining $f(\boldsymbol{\omega})=\omega_i \omega_j$ we find the second order moments of the angular measure $H_{\bsx}$. 
 
 \begin{definition}\label{scaledef}
 Let $\bsx\in {\rm RV}^D_+(2)$ and consider its angular representation $(R,\boldsymbol{\omega})=(\norm{\boldsymbol{X}}, \bs X/ \norm{\boldsymbol{X}})$ as in Definition~\ref{mrv}(ii), setting $\omega_i={X_i}/{R}$ for $i\in\{1,\dots, D\}$, and $\boldsymbol{\omega}=(\omega_1,\dots,\omega_{ {D}})\in \Theta_+^{{ {D}}-1}$. 
 \begin{itemize}
     \item[(i)] We define the second order moments 
 \begin{align*}
 \sigma_{ij}^2 = \sigma_{X_i, X_j}^2&\coloneqq\int_{\Theta_+^{{ {D}}-1}}\omega_i \omega_j \,{\rm d}H_{\boldsymbol{X}}(\boldsymbol{\omega}), \quad  i,j\in \{1,\ldots, {D}\}. 
 \end{align*}
 \item[(ii)]
 $\sigma_{ij}^2$ is the extremal dependence measure of $(X_i,X_j)$ \cite[eq. (8)]{lars}.
\item[(iii)]
$\sigma_i =\sigma_{{X}_{i}}=\sigma_{{X}_{i},X_i}$ {is the scaling (parameter) of $X_i$ \cite[Section~4]{cooley}.}
 \item[(iv)]
The matrix $\Sigma\coloneqq(\sigma_{ij}^2)_{D\times D}$ summarises the second-order properties of $H_{\bsx}$. 
  \end{itemize}
 \end{definition}

 We also define the standardised RMLM obtained from~\eqref{Rmlmequat} by standardising the rows of $A$. 

 \begin{definition}
Let $A$ be a coefficient matrix with row vectors $A_i$ and column vectors $a_k$, then the standardised coefficient matrix $\bar{A}$ is defined as 
 \begin{align}
 \bar{A}=(\bar{a}_{ij})_{{D}\times{D}}
{ =\bigg(\frac{a_{ij}}{\|A_{i}\|}\bigg)_{{D}\times{D}}}
 \coloneqq \bigg(\frac{a_{ij}^2}{\sum_{k\in {\An(i)}}a_{ik}^2}\bigg)_{{D}\times{D}}^{1/2}
 =\bigg(\frac{a_{ij}^2}{\sum_{k=1}^{D}a_{ik}^2}\bigg)_{{D}\times{D}}^{1/2}.\label{abar}
 \end{align}    
 \end{definition}

 In the rest of the paper we make the following assumptions.\smallskip 
 
 \noindent
 {\bf {Assumptions A:}}
 \begin{enumerate}
 	\item[(A1)]
 	The innovations vector $\boldsymbol{Z}\in {\rm RV}^{D}_+(2)$ has independent and standardised components.
 	\item[(A2)]
 	The norm $\|\cdot\|$ denotes the Euclidean norm.
 	\item[(A3)]
 	The coefficient matrix $A$ is standardised as in~\eqref{abar}; i.e., the components of $\bsx$ are standardised. 
 \end{enumerate}
 

The following proposition collects mainly results (formulated for the non-standardised coefficient matrix) from \citet[Lemma~3]{foug} and \citet[Proposition~5]{cooley}; see also \citet[Section 2.2]{KK}.

\begin{proposition}\label{completedep}
Let ${\boldsymbol{X}}$ be an RMLM satisfying Assumptions $A$.
Then for $i, j\in\{1,\ldots, D\}$,
  \begin{enumerate}
   \item[(i)]
 the angular measure $H_{\bsx}$ of $\bsx$ is discrete with atoms $(a_{i}/\norm{a_i})$, which are the normalised columns of {$A$},
 \item[(ii)]
  $\sigma_{ij}^2=(AA^\top)_{ij}=\sum_{k=1}^D a_{ik} a_{jk}$ and
 $\sigma^2_{i}=({A}{A}^\top)_{ii}=1$, 
 \item[(iii)] 
 the (sub)vector $(X_i,X_j)$ from~$\bs X$ can be represented via the matrix $A_{ij}\in\R^{2\times D}$ having only rows $i$ and $j$ of $A$, and $\sigma_{ij}^2=(A_{ij}A_{ij}^\top)_{12}$,
\item[(iv)]  
(a) $\sigma_{ij}^2=1$ for $i\neq j$, (b)  $ {\bar a_{ik} = \bar a_{jk}}$ for all $k\in\{1,\dots,D\}$, and (c) $ X_i, X_j$ are asymptotically fully dependent,  are equivalent,
\item[(v)]  
$\sigma^2_{i}=\lim_{n\to \infty} n\mathbb{P}({{X_{i}^2}}>{n})=1$.
 \end{enumerate}
\end{proposition}

\begin{proof}
(iv) As $\sigma_{ij}^2$ is defined via the limiting angular measure,  asymptotic full dependence implies that the variables $X_i,X_j$ become completely dependent as $\bs X$ becomes more extreme. Appendix~\ref{sec:3.2} gives a precise definition and proof of the last equivalence of part~(iv) of Proposition~\ref{completedep}. \\
(v) is a consequence of the standardisation and Lemma~\ref{tailMk} for $|K|=1$. 
\end{proof}

{The choice of $\alpha=2$ and the Euclidean norm allows for the representation of the
scalings and extremal dependence measures by the entries of $A$. In addition, these measures are invariant with respect to the dimensionality of the angular measure (see Remark~\ref{invariant}). 
Results can be extended for $\alpha \neq 2$, but then expression for the $\sigma_{ij}$’s contain $\alpha$ \citep{kirilzhou}, complicating the notation. }

%
 
{The main result of this section (Theorem~\ref{incdagmod} below) relies on transformations and scalings of the maximum, $M_{K}\coloneqq\max(X_k:k\in K)$, over components of an RMLM $\boldsymbol{X}$ indexed in $K$. 
For the next definition, take $f(\boldsymbol{\omega})=\bigvee_{k\in K}\omega_k^{2}$ and use \eqref{p2empdist} (note that $f$ satisfies the required conditions).}
 
 \begin{definition}\label{scalMK}
 Let $\bsx$ be an RMLM satisfying Assumptions A with angular representation $(R,\boldsymbol{\omega})=(\norm{\boldsymbol{X}}, \bs X/ \norm{\boldsymbol{X}})$ as in Definition~\ref{mrv}(ii); set $\omega_i={X_i}/{R}$ for $i\in\{1,\dots, D\}$, and $\boldsymbol{\omega}=(\omega_1,\dots,\omega_{ {D}})\in \Theta_+^{{ {D}}-1}$. 
 Let $M_{K}\coloneqq\max(X_k:k\in K)$ for $K\subseteq \{1,\dots, D\}$.
 We define the scaling of $M_K$ as
	\begin{align}\label{def:scalMK}
	\sigma_{M_{K}}^2&\coloneqq\int_{\Theta_+^{ {d}-1}} \underset{k\in K}{\bigvee}\omega_k^{2} {\rm d}H_{{\boldsymbol{X}}}(\boldsymbol{\omega}). 
	\end{align}
 \end{definition}


Lemma 6 in \citet{KK} characterises the scalings of such objects in terms of the coefficient matrix $A$ and are restated in parts (i,ii) of the next lemma; part (iii) is proved in Lemma~\ref{tailMk}.

\ble\label{scalcoll}
Let $\boldsymbol{X}$ be an RMLM satisfying Assumptions A. 
Then $M_{K}$ is a max-linear combination of the innovations that lies in ${\rm RV}_+(2)$ and has squared scalings as follows:
\begin{itemize}
    \item[(i)] if $K\subset\{1,\dots,{D}\}$, {then} $\sigma_{M_{K}}^2 =\sum_{k\in \An(K)}\bigvee_{i\in K} a_{ik}^2$,
\item[(ii)] if $K =\{1,\dots,{D}\}$, then $\sigma_{M_{K}}^2 = \sum_{k \in \An(K)}  a_{kk}^2$,
\item[(iii)] {$\sigma_{M_{K}}^2=\lim\limits_{n\to \infty} n\mathbb{P}({{M_{K}}}>\sqrt n)$.}
\end{itemize}
\ele

\subsection{The max-linear matrix of an RMLM with hidden nodes} \label{MLreparamet}

{Now we turn to the main questions of this paper.}
Assume that only $d<D$ of the nodes of the DAG supporting $\bsx$ are observed, corresponding to a max-linear vector $\boldsymbol{X}_O=A_O\times_{\max}\boldsymbol{Z}\in\R_+^d$ with coefficient matrix $A_O\in\R^{d\times D}$, whose rows correspond to the observed node variables in $O$.  
{We shall use the fact that   $\boldsymbol{X}_O$ as a subvector of $\bsx$ is again max-linear and $\boldsymbol{X}_O \in{\rm RV}_+^d(2)$.}

 It is  natural to ask whether one can rewrite $A_O$ as a square $d\times d$ matrix and the innovations vector as a $d$-vector of exogenous random variables with independent components in ${\rm RV}_+(2)$, {which would entail that $\boldsymbol{X}_O$
 can be represented as an RMLM.} To address this we start with a useful lemma.

\begin{lemma}\label{mlcomb}
If the innovations $Z_1,\ldots,Z_p$ for $p\in\{1,\dots,D\}$ satisfy Assumption (A1) and $(a_1,\ldots, a_p)$ $\in \mathbb{R}_+^p$, then the maximum $M\coloneqq\bigvee_{i\in\{1,\ldots,p\}}a_{i}Z_i$ belongs to ${\rm RV}_+(2)$ with squared scaling $\sigma^2=\sum_{i=1}^p a_{i}^2$. 
In particular, $M=\sigma Z$ where $Z\in {\rm RV}_+(2)$ and has unit scaling.
\end{lemma}

\begin{proof}
Recall that ${\rm RV}_+(2)$ is closed with respect to max-linear combinations.
The scaling follows as in Lemma~\ref{scalcoll} (i), and defining $Z:=M/\sigma$ implies that $Z\in {\rm RV}_+(2)$ with unit scaling.
\end{proof}

We now illustrate how closure under max-linear combinations can reduce the dimension of $A_O$.

\begin{example}\label{AOex}
Consider an RMLM supported on the DAG $\mathcal{D}_2$ in Figure~\ref{introex} and satisfying Assumptions~A, with a max-weighted path $3\to 2\to 1$. Here $O=\{1,2\}$, $D=3$ and $d=2$ and node 3 is hidden, and we have the reduced max-linear representation 
\begin{align}\label{same.rep}
	\begin{bmatrix}
	X_{1}\\
	X_{2}
	\end{bmatrix}
	& {=} 
	\begin{bmatrix}
	a_{11}Z_1\vee a_{12} Z_{2}\vee a_{13}Z_{3}\\
	a_{22}Z_{2}\vee a_{23}Z_{3}
	\end{bmatrix}
	\overset{}{=}  
	\begin{bmatrix}
	a_{11}Z_1\vee (a_{12}/a_{22})X_{2}\\
	a_{22}^{*}Z^{*}_{2}
	\end{bmatrix}\overset{}{=}\begin{bmatrix}
	a_{11}Z_1\vee (a_{12}a_{22}^{*}/a_{22})Z^{*}_{2}\\
	a_{22}^{*}Z^{*}_{2}
	\end{bmatrix} ,
	\end{align}
	where $a_{22}^{*2}=a_{22}^2+a_{23}^2$, and $Z_2^{*}=(a_{22}Z_{2}\vee a_{23}Z_{3})/a_{22}^{*}$ is a standardised innovation (Lemma~\ref{mlcomb}). Hence the new innovation vector is $(Z_1, Z_2^{*})$ and the reduced coefficient matrix lies in $\mathbb{R}^{2\times 2}$.
\end{example}

The following result shows how to re-parametrise the max-linear vector $\boldsymbol{X}_O=A_O\times_{\max}\boldsymbol{Z}$ with $|O|=d$, coefficient matrix $A_O\in\R^{d\times D}$ and $\bsz\in\R_+^D$ as an RMLM with reduced and upper-triangular coefficient matrix, under the conditions of Theorem~\ref{prop.dat}. 

	\begin{proposition}\label{prop.incmat} \,
		Suppose that $\boldsymbol{X}\in{\rm RV}_+^D(2)$ is an RMLM {satisfying Assumptions~A with coefficient matrix $A$}, and that the observed node variables $\bsx_O$ have $|O|=d < D$, where the {nodes in $O$ are well-ordered}, and satisfy conditions (i) and (ii) of Theorem~\ref{prop.dat}.
  Then 
  \begin{align}\label{eq:triang}
  \boldsymbol{X}_O=A_O^{*}\times_{\max}\boldsymbol{Z^{*}}
  \end{align} 
  with standardised innovations vector $\boldsymbol{Z}^{*}=(Z_1^{*},\dots,Z_d^{*})$ and reduced coefficient matrix 
		\begin{align}\label{Am}
		A_{O}^{*}=\begin{bmatrix}
		a^{*}_{11} & a^{*}_{12}& \hdots  & a^{*}_{1d} \\
		0 &   a^{*}_{22} &\hdots   & a^{*}_{2d}\\
		\vdots& \vdots & \ddots & \vdots\\
		0&   0 &\hdots    &a^{*}_{dd}
		\end{bmatrix},
		\end{align}
		where, with $\an^{O}(i)$ and $\An^{O^c}(i)$ defined in~\eqref{rmlmequat}, 
  $$
  a_{ii}^{*}=\Big(\sum_{k\in  \An^{O^c}(i)} a_{ik}^2\Big)^{1/2},\qquad a_{ij}^{*}=\bigvee_{k\in {\Des(j)}\cap\an^{O}(i)} \frac{a_{ik}}{a_{kk}}a_{kj}^{*},\quad {j > i}.
  $$
  	\end{proposition}

\subsection{Identifying max-weighted paths: the set  MWP} \label{sec:MWP}

In Section~\ref{MLreparamet}  we have confirmed that an observed vector $\bsx_O$ that satisfies conditions (i) and (ii) of Theorem~\ref{prop.dat} {with a well-ordered set of nodes $O$} can be represented as $A_O^{*}\times_{\max} \bsz^{*}\in\R_+^d$ for a triangular matrix $A_O^{*}\in\R^{d\times d}$ as in~\eqref{Am} and a vector of innovations $\bsz^{*}\in\R_+^d$.

We now turn our attention to condition (i) of Theorem~\ref{prop.dat}, or equivalently the identification of MWP, as addressed in Algorithm~\ref{datdalgrmlm}, which requires that max-weighted paths from hidden confounders in the $D$-dimensional DAG pass through observed ancestral nodes.
To this end, we investigate the extremal dependence measure of certain transformations of $(X_i,X_j)$ for two observed nodes $i$ and $j$ that have a common ancestor, or are such that $j\in\an(i)$. It turns out that such a pair of variables can be represented as an RMLM if and only if the extremal dependence measure between the transformed random variables equals unity. This provides a way to reduce the dimension of the max-linear representation, and, in particular, to verify whether hidden confounders can be ignored.

Assume for now that {$i\notin\an(j)$, so that the pair $(i, j)$ can be ordered as $i<j$ on a well-ordered DAG}. This is ensured by Lemma~\ref{lemmaid1}(iii), or condition~\eqref{cond1} of Theorem~\ref{incdagmod} below. 
We start with the submatrix of the $i$-th and $j$-th rows of $A$, 
\begin{align}\label{eq:Aim}
A_{ij}= \begin{bmatrix}
0& \cdots & a_{ii} & \cdots & a_{i,j-1} & a_{ij} & \cdots & a_{iD}\\
0& \cdots & 0 & \cdots & 0 & a_{jj} & \cdots & a_{jD}
\end{bmatrix}\in\R^{2\times D},
\end{align}
where $a_k=(a_{ik},a_{jk})$ is the $k$-th column of $A_{ij}$ for $k\in\{1,\dots,D\}$.

The random variable 
\begin{align}\label{def:twomax}
M_{c_1 i, c_2 j}\coloneqq\max(c_1 X_i,c_2 X_j), \quad c_1,c_2>0, 
\end{align}
is a max-linear combination of $Z_1,\dots,Z_D$ 
with coefficient matrix in $\mathbb{R}^{1\times D}$ with entries $c_1 a_{ik}\vee c_2 a_{jk}$ for $k\in\{1,\ldots, D\}$. 

The next lemma establishes a connection between the max-weighted path property for $(i,j)\in~{\rm MWP}$ and certain linear transformations between the entries of the coefficient matrix $A_{ij}$. These new transformed entries then appear in the angular measure of $(X_i, X_j)$ by applying similar transformations to the bivariate random vector, given in~\eqref{def:tim}. 

\begin{lemma}\label{lemmaid1}
   Consider a subvector $(X_i, X_j)\in{\rm RV}_+^2(2)$ from an RMLM and let $(a_{ij},a_{jj})$ and $(a_{ik},a_{jk})$ denote the $j$-th and $k$-th columns of the coefficient matrix $A_{ij}$ in \eqref{eq:Aim}. Then
   \begin{enumerate}
  \item[(i)] for $0<c_1\leq 1, c_2>0$, define 
$\tilde{a}_k=(\tilde{a}_{ik},\tilde{a}_{jk}) \coloneqq
(a_{jk}-c_1a_{ik},a_{jk}+c_2a_{ik})$ for $ k\in\An(j)$.
Then 
\begin{align}\label{eq:equiv}
a_{jj}a_{ik}=a_{jk}a_{ij}\quad\iff\quad\tilde{a}_{jj}\tilde{a}_{ik}=\tilde{a}_{jk}\tilde{a}_{ij},
\end{align}
and both equalities are equivalent to the {existence of a} max-weighted path $k\rightsquigarrow j \rightsquigarrow i$; 
\item[(ii)] $(i,j)\in{\rm MWP}$ if and only if the row vectors $(\tilde{a}_{i1},\ldots, \tilde{a}_{iD})$ and $(\tilde{a}_{j1},\ldots, \tilde{a}_{jD})$, defined in $(i)$ for $k\in\An(j)$ and $\tilde a_k=(0,0)$ for $k\notin\An(j)$, are linearly dependent. In contrast, the row vectors $A_i$ and $A_j$ are linearly independent; and 
\item[(iii)] if there exists $a>1$ such that 
$$\sigma_{M_{i,aj}}^2=\sigma_{M_{ij}}^2 +a^2-1\quad\mbox{and}\quad
\sigma_{M_{ai,j}}^2<\sigma_{M_{ij}}^2 +a^2-1,$$
then {$i\notin \an(j)$}, $\An(i)\cap\An(j)\neq\emptyset$, and $a_{jk}\geq a_{ik}$ for all $k\in\An(j)$;
otherwise, either $\An(i)\cap \An(j)=\emptyset$  or $(i,j)\notin {\rm MWP}$.
\end{enumerate}
\end{lemma}

\begin{remark} 
As is immediate from Definition~\ref{mwpair}, membership of $(i,j)$ in MWP requires that~there~are max-weighted paths $u\rightsquigarrow j \rightsquigarrow i$ for all $u\in \An(i)\cap\An(j)$,  and therefore ignores the effect of nodes outside $\An(j)$. This enables us to deduce {membership in MWP} for each pair of nodes $(i,j)$ from the linear dependence in Lemma~\ref{lemmaid1} $(ii)$, with the latter motivating the transformation \eqref{def:tim}~below.
\end{remark}

For $0<c_1\le 1$ and $c_2>0$ define the vector 
\begin{align}\label{def:tim}
\boldsymbol{T}^{ij}= (T_1^{ij}, T_2^{ij}) \coloneqq(M_{c_1 i,j}-c_1X_i, (1+c_2)X_j+c_2X_i-c_2M_{ij});
\end{align}
this is a linear function of $\boldsymbol{T}^{ij}_2\coloneqq(M_{c_1 i,j}, M_{ij}, X_i, X_j)$.
Table~\ref{table:atoms} is a consequence of Proposition~\ref{prop:easy}, Lemma~\ref{breimanlemma}, Corollary~\ref{cor:easy}, and Lemma~\ref{genmaxspectmainCK} provides the atoms of the angular measures of transformations of $(X_i,X_j)$ used in~\eqref{def:tim} and, in particular, of the angular measure of $\boldsymbol{T}^{ij}\in{\rm RV}^2_+(2)$. 
The atoms of this angular measure contain only indices $k\in \An(j)$, since, by a version of Breiman's Lemma~\ref{breimanlemma}, $\tilde{a}_k=0$ if $k\notin \An(j)$, corresponding to those $\tilde{a}_k$ defined in Lemma~\ref{lemmaid1}~(i)--(ii). 

\begin{table}[H]
	{\centering\caption{Vectors $\tilde a_k$ used to obtain the atoms of transformations of $(X_i, X_j)$ in \eqref{def:tim}. 
    \label{table:atoms}}}
	\centering
	\fbox{%
		\begin{tabular}{l|l|l}
			Notation & Vector & $\tilde a_k $  \\
			\hline
			$\boldsymbol{T}^{ij}_1 $ & $(M_{ij},X_i,X_j)$ & $(a_{ik}\vee a_{jk}, a_{ik}, a_{jk})$ \\ \rule{0pt}{3ex}
			$\boldsymbol{T}^{ij}_2$ & $(M_{ci,j}, M_{ij}, X_i, X_j)$ & $(ca_{ik}\vee a_{jk} , a_{ik}\vee a_{jk}, a_{ik}, a_{jk})$ \\\rule{0pt}{3ex}
			$\boldsymbol{T}^{ij}$ & $(M_{c_1i,j}-c_1X_{i},$ & $(c_1a_{ik}\vee a_{jk} -c_1a_{ik}$, \\
			& $(1+c_2)X_j+c_2X_i-c_2M_{ij})$ & $(1+c_2)a_{jk}+c_2a_{ik}-c_2(a_{ik}\vee a_{jk}))$	\\
		\end{tabular}}
	\end{table}
\vspace*{-.75cm}

\begin{lemma}\label{lemmaid2}
Let $\boldsymbol{T}^{ij}$ be as in~\eqref{def:tim} for $0<c_1\leq 1$ and $c_2>0$.
If the condition of Lemma~\ref{lemmaid1} $(iii)$ holds, then $\boldsymbol{T}^{ij}\in {\rm RV}_+^
2(2)$ has discrete angular measure {with atoms $\tilde{a}_k/\norm{\tilde{a}_k}$ derived from the non-zero vectors $\tilde{a}_k=(a_{jk}-c_1a_{ik}, a_{jk}+c_2a_{ik})$ for $k\in\An(j)$.}
 Moreover, $\tilde a_k=(0,0)$ for $k\notin\An(j)$. 
\end{lemma}

\begin{proof}
{This is a consequence of Lemma~\ref{genmaxspectmainCK}, the choice of $c_1,c_2$, which gives the formula in the third row of Table~\ref{table:atoms},
and the fact that $a_{jk}\geq a_{ik}$ for $k\in\An(j)$ by Lemma~\ref{lemmaid1}(iii).} 
\end{proof}

It is important below that the scalings of the components of 
$\boldsymbol{T}^{ij}$ (see Definition~\ref{scaledef}) may not equal unity. 
To adjust for this, we define the standardised random vector and the extremal dependence measure 
\begin{align}
\bs{\tilde{T}}^{ij} &=(\tilde{T}^{ij}_{1}, \tilde{T}^{ij}_{2}),\quad \mbox{where}\quad  \tilde{T}^{ij}_{k}\coloneqq {T}^{ij}_{k}/\sigma_{{T}^{ij}_{k}} \quad\mbox{for}\quad k\in\{1,2\},\label{eq:Tscaled} \\
\smash{\tau^2_{ij}}&\coloneqq\smash{\sigma^2_{\tilde{T}^{ij}_{1},\tilde{T}^{ij}_{2}}}\label{tau}
\end{align}

The following result characterises when a pair $(i,j)$ belongs to MWP.
\begin{theorem}\label{incdagmod}
	{Let $\boldsymbol{X}\in {\rm RV}_+^D(2)$  be an RMLM} 
    {satisfying Assumptions~A.}
	Suppose that we observe nodes $i,j\in V$, and that for some $a>1$,
	\begin{equation}\label{cond1}
	\sigma_{M_{i,aj}}^2=\sigma_{M_{ij}}^2 +a^2-1,
	\quad\quad 
	\sigma_{M_{ai,j}}^2<\sigma_{M_{ij}}^2 +a^2-1.
	\end{equation}
  Consider $\bs{{T}}^{ij}$ as in~\eqref{def:tim} for $0<c_1\leq 1$ and $c_2>0$, and {$\bs{\tilde{T}}^{ij}$ as in~\eqref{eq:Tscaled}.}
  Then: 
  \begin{enumerate}
  \item[(i)] {$(i,j)\in{\rm MWP}$}
 if and only if 
	$\tau_{ij}^2=1$.
 In this case, $(X_i,X_j)$ can be represented as an RMLM; 
	\item[(ii)] if $i\notin \emph{de}(j)$,
 then
	$\tau_{ij}^2<1$ and $(X_i,X_j)$ cannot be represented as an RMLM.
 \end{enumerate}
\end{theorem}

\begin{remark}\label{rem:two_obs}
In the situation of Theorem~\ref{incdagmod}(i), {$\tilde{T}^{ij}_{1},\tilde{T}^{ij}_{2}$ are asymptotically fully dependent} by Corollary~\ref{completedep} (iii). 
In the situation of Theorem~\ref{incdagmod}(ii), there exists a confounder $u\in\An(i)\cap\An(j)$, but no max-weighted path $u \rightsquigarrow j \rightsquigarrow i$.
\end{remark}

The following corollary is particularly useful for the statistical applications in Section~\ref{sec:est}. 
Fix a pair $(i,j)$ of observed nodes and assume that there exists some $a>1$ such that~\eqref{cond1} holds. Consider ${\tilde{T}^{ij}_{1}},{\tilde{T}^{ij}_{2}}$ as in \eqref{eq:Tscaled} 
and let $0<c_1,c'_1\leq 1$ and $c_2>0$.
Furthermore, define ${\tilde{T}^{ij'}_{1}}$ analogously to ${\tilde{T}^{ij}_{1}}$, but replacing the scalar $c_1$ by $c'_1$. 
Similar to \eqref{tau}, we define  ${\tau_{ij'}^2\coloneqq\sigma^2_{\tilde{T}^{ij'}_{1},\tilde{T}^{ij'}_{2}}}$.

\begin{corollary}\label{cor37}
{If $(i,j)\in {\rm MWP}$}, then $\tilde{\Delta}_c\coloneqq|\tau_{ij'}^2-\tau_{ij}^2|=0.$
\end{corollary} 

The following remark concerns asymptotically independent node variables $(X_i,X_j)$, which by symmetry lead to an asymptotically fully dependent vector $\boldsymbol{\tilde T}^{ij}$.

\begin{remark}\label{remarkai}
Conditions~\eqref{cond1} of Theorem~\ref{incdagmod} exclude asymptotically independent pairs $(X_i,X_j)$, i.e., pairs for which $\sigma_{M_{ij}}^2=2$, since then the inequality in~\eqref{cond1} becomes an equality.  
Using Table~\ref{table:atoms}, the angular measures for $(X_i,X_j)$ and for the transformed pair $\tilde{\boldsymbol{T}}^{ij}$ 
in \eqref{eq:Tscaled}, 
respectively, consist of 
the normalised non-zero columns of the coefficient matrices
$$A_{ij}=\begin{bmatrix} 1& 0 \\ 0& 1
\end{bmatrix}, \quad A_{\tilde{\boldsymbol{T}}^{ij}}=\begin{bmatrix} 0& 1 \\ 0& 1
\end{bmatrix}.$$
Then Proposition~\ref{completedep}~(ii) implies that $\tau_{ij}^2={\tau}_{ji}^2=1$.
\end{remark}

\begin{example}\label{ex1mwp}
Consider RMLMs supported on the DAGs in Figure~\ref{introex}, with coefficient matrix $A$ and innovations $\boldsymbol{Z}$	satisfying Assumptions~A.	We consider the DAGs separately:
    
--\, for  $\D_1$ we apply Theorem~\ref{incdagmod} for $(i,j)=(1,2)$.  
Since there is a unique (max-weighted) path $3\to2\to1$, the pair $(X_1, X_2)$ can be represented as an RMLM with node $2$ behaving as a source node in the observed DAG, and so condition~\eqref{cond1} holds by Theorem~2 of \citet{KK}. Computing the standardised vector  $\tilde{\boldsymbol{T}}^{12}$ and its extremal dependence measure gives $\tau_{12}^2=1$, so $\tilde{T}^{ij}_{1},\tilde{T}^{ij}_{2}$ are asymptotically fully dependent. 
Hence, by Theorem~\ref{prop.dat} and Proposition~\ref{prop.incmat}, 
$(X_1, X_2)$ can be represented as an RMLM, similar to Example~\ref{ex:2.1}, with reduced coefficient matrix $A_O^{*}$ computed as in Example~\ref{AOex};

--\, for $\D_2$ we first note that if the path $3\to 2 \to 1$ were max-weighted, then steps similar to those above lead to the same result as for $\D_1$.  There are two possibilities if the edge $3\to 1$ is the only max-weighted path: condition~\eqref{cond1} may not hold, in which case $(X_1, X_2)$ cannot be represented as an RMLM, whereas if the condition is satisfied, Theorem~\ref{incdagmod}(i) implies that $\tau_{12}^2<1$. If so, $(X_1, X_2)$ cannot be represented as an RMLM due to the existence of the confounder node 3;

--\, for $\D_3$ we note that there are also two possibilities. Condition~\eqref{cond1} may fail. Alternatively, if it is satisfied, Theorem~\ref{incdagmod} (ii) yields $\tau_{12}^2<1$. In both cases $(X_1, X_2)$ cannot be represented  as an RMLM, due to the existence of the confounder node $3$. 
\end{example}

\section{Estimation}\label{sec:est}

We now employ the results of Section~\ref{sec:MWP} in a statistical algorithm to detect RMLMs between pairs of node variables by estimating scalings and extremal dependence measures from 
Definition~\ref{scaledef} for appropriately transformed observations. 
This uses the link  established in Theorem~\ref{incdagmod} between the extremal dependence measure and the max-weighted path property.
{By Remark~\ref{rem:two_obs}} the max-weighted path property equivalently determines whether confounders can be ignored. 
We recall from Definition~\ref{mwpair} that 
{$(i,j)\in$ MWP} if for all $u\in \An(i)\cap\An(j)$ there are max-weighted paths $u\rightsquigarrow j \rightsquigarrow i$. Appendix~\ref{scalest} outlines the estimation procedure for the scalings and extremal dependence measures used in Section~\ref{statMWP}, which are special cases of~\eqref{finalest11} below. 
Section~\ref{funcCLT} and Appendix~\ref{sec:stat} establish the asymptotic properties of the estimators. 

\subsection{A statistical algorithm to estimate MWP}\label{statMWP}

{Assume $n$ observations of a subvector $X_O\in{\mathbb{R}}_+^{d}$ from an RMLM  satisfying Assumptions~A.}
We define the following real $d\times d$ matrices, whose entries are estimated non-parametrically from the empirical angular measure (see Appendix~\ref{scalest}), based on pairs of components of  $\boldsymbol{X}_O$:
\begin{itemize}
\item[-] $\hat{C}^{(1)}$ with entries $\hat{C}^{(1)}_{ij}=\min\Big(0.1+(\hat{\sigma}_{i}^2+\hat{\sigma}_{j}^2-\hat{\sigma}_{M_{ij}}^2)^{1/2}, 0.8\Big)$; 
\item[-] $\hat{\Delta}^{(1)}$ with entries $\hat{\Delta}^{(1)}_{ij}=(\hat{\sigma}_{M_{i,a j}}^2-\hat{\sigma}_{M_{ij}}^2-a^2+1)/(a^2-1)$;
\item[-] $\hat{\Delta}^{(2)}$ with entries $\hat{\Delta}^{(2)}_{ij}=\hat{\tau}_{ij}^2$;
\item[-] $\hat{\Delta}^{(3)}$ with entries $\hat{\Delta}^{(3)}_{ij}=|\hat{\tau}_{ij'}^2-\hat{\tau}_{ij}^2$|; 
\item[-] $\hat{\Delta}^{(4)}$ with entries $\hat{\Delta}^{(4)}_{ij}=\hat{\sigma}_{M_{ij}}^2$.
\end{itemize}

The last four matrices are used by Algorithm~\ref{datdalg2} below to identify all pairs of observed nodes in MWP. 
The entries of $\hat{C}^{(1)}$ are used for the coefficients $c_1$, $c_1'$, $c_2$ of the linear transformation $\boldsymbol{T}^{ij}$ defined in~\eqref{def:tim} and for $\boldsymbol{T}^{ij'}$ in Corollary~\ref{cor37}. These coefficients are chosen based on the discussion in Section~\ref{sec:c} and Example~\ref{ex:c1}: for each pair $(i,j)$ we set
$c_1=\hat{C}^{(1)}_{ij}$, $c'_1=0.1\,c_1$,  and $c_2=1/c_1$.
\subsubsection{Motivating  the structure of the algorithm}
We now briefly discuss Algorithm~\ref{datdalg2}, particularly its lines 1--4, which are linked to the findings of Section~\ref{sec:MWP}:
\begin{enumerate}
	\item[Line 1:] the set $S_1$ identifies pairs $(i, j)$  satisfying condition~\eqref{cond1};
	\item[Line 2:] the set $S_2$ identifies those pairs satisfying Theorem~\ref{incdagmod} (i);
	\item[Line 3:]  the set $S_3$ excludes those pairs that are asymptotically independent (Remark~\ref{remarkai});  
	\item[Line 4:]  the set $S_4$ excludes those pairs for which the necessary condition of Corollary~\ref{cor37} fails.
\end{enumerate}

Theorem~\ref{incdagmod} implies that the sets $S_1$ and $S_2$ suffice to identify MWP. 
However, the pre-asymptotic regime in the simulation study of Appendix~\ref{sec:performance} indicates many false positives, whereby non-max-weighted pairs are estimated as max-weighted pairs. 
This motivates the use of two additional necessary conditions, contained in Remark~\ref{remarkai} and Corollary~\ref{cor37}, to reduce the number of false positives. 
If $(X_i,X_j)$ are asymptotically independent, Remark~\ref{remarkai}  gives ${\tau}_{ij}^2={\tau}_{ji}^2=1$, so the conditions in~\eqref{cond1} are violated, though the algorithm outputs $(i,j)\in S_1$;  $S_3$ eliminates such pairs. 
Corollary \ref{cor37} states that transformations with different $c_1,c'_1$ result in the same extremal dependence measure for pairs in MWP, $S_4$ removes such pairs.



The matrix $\hat{\Delta}^{(4)}$, used specifically for the data application,  deals with those asymptotically strongly dependent pairs that do not satisfy~\eqref{cond1}. We consider such pairs to be indistinguishable. 

Algorithm~\ref{datdalg2} relies on small constants to allow for estimation errors that arise because the data come from a pre-asymptotic regime resulting from the finite sample size. The simulation results described in Appendices~\ref{sec:performance} {and~\ref{Ap:sim2}} indicate that small changes in the error terms, particularly $\eps_3$, $\eps_4$ and $\eps_5$, can lead to significant changes {in levels of both true positive and false discovery rates}, making it difficult to use the asymptotic normality of the estimated extremal dependence measure, particularly because their confidence intervals exclude asymptotic full dependence in nearly all cases. Although the estimators remain consistent,  their asymptotic normal distribution is degenerate when the extremal dependence measure approaches unity, as noted also in the remark to Theorem~1 of~\citet{lars}. 
Hence, we select the $\eps$ terms based on a simulation study, aiming for a reasonable trade-off between true positive and false discovery rates. We discuss the choice of $a>0$ and the $\eps$ terms in Appendices~\ref{simul} {and~\ref{Ap:sim2}.} Appendix~\ref{sec:performance}  studies the performance of the algorithm by simulation. 

%
\begin{algorithm}[t]
		\caption{Estimation of pairs of nodes in MWP, and of indistinguishable nodes} 
		\label{datdalg2}
\flushleft 
\textbf{Parameters}:\quad $a>1, \eps_1, \eps_2, \eps_3, \eps_4, \eps_5, \eps_6\geq0$  \\
\textbf{Input}:\quad ${\hat C^{(1)}}, \hat{\Delta}^{(1)}, \hat{\Delta}^{(2)}, \hat{\Delta}^{(3)},  \hat{\Delta}^{(4)}$  \\
\textbf{Output}: Matrix $\hat{P}\in \{0,1\}^{d\times d}$ {indicating pairs in MWP}\\
\hspace{1.2cm} Matrix $\hat{P}^{*}\in \{0,1\}^{d\times d}$ {indicating indistinguishable pairs} \\
\textbf{Procedure:} 
		\begin{algorithmic}[1]
                \State \textbf{set} $S_1 = \{(i,j)\in O\times O:\hat{\Delta}^{(1)}_{ij}\geq-\eps_1 \quad {\rm and} \quad \hat{\Delta}^{(1)}_{ij}-\hat{\Delta}^{(1)}_{ji}\geq-\eps_2\}$ \quad[{\rm \textcolor{blue}{Conditions~\eqref{cond1}}}]
                \State \hspace{4.5mm} $S_2=\{(i,j)\in O\times O: \hat{\Delta}^{(2)}_{ij}>1-\eps_3\}$ \quad [{\rm \textcolor{blue}{Theorem~\ref{incdagmod} (i)}}]
                \State \hspace{4.5mm} $S_3=\{(i,j)\in O\times O: \hat{\Delta}^{(2)}_{ij}>\hat{\Delta}^{(2)}_{ji}+\eps_4\}$ \quad [{\rm \textcolor{blue}{Remark~\ref{remarkai}}}]
                \State \hspace{4.5mm} $S_4=\{(i,j)\in O\times O: \hat{\Delta}^{(3)}_{ij}<\eps_5 \, {\hat C^{(1)}_{ij}}\}$ \quad [{\rm \textcolor{blue}{Corollary~\ref{cor37}}}]
			\State\textbf{for} $(i,j)\in O\times O$  
			
			\State\hspace{5mm}\textbf{if} 
			$(i,j)\in S_1\cap S_2 \cap S_3 \cap S_4$, \textbf{set} $\hat{P}_{ij}=1$
			\State\hspace{10mm}\textbf{else if} $\Delta^{(4)}_{ij}<1+\eps_6$, \textbf{set} $\hat{P}_{ij}=0 , \, \hat{P}^{*}_{ij}=1$
                \State\hspace{20mm}\textbf{else set} $\hat{P}_{ij}=0 , \, \hat{P}^{*}_{ij}=0$
			\State \textbf{end for}
			
		\end{algorithmic}
 \textbf{Return} $\hat{P},\hat{P}^{*} $ 
	\end{algorithm}\vspace{-.4cm}



\subsection{A data example}\label{sec:data}

We now apply Algorithm~\ref{datdalg2} to nutrient intake data for $n=9544$ individuals from the NHANES survey available at \url{https://wwwn.cdc.gov/Nchs/Nhanes/2015-2016/DR1TOT_I.XPT}, which also gives more details of the 168 variables. We treat individuals as independent and identically distributed and work with the $d=39$ observed components shown in Figure~\ref{fig:foodpairs}; see also \citet{JanWan}.  

We focus on the causal dependence between high nutrient intakes, which might damage the health~of an individual, resulting in toxicity or other complications. Due to similarities in their chemical compositions, large amounts of a specific nutrient are likely to be associated with those of related nutrients, so a better understanding of the causal mechanisms driving such extreme dependencies will be valuable.

 As in the simulation study, our {first} goal is to identify pairs of variables for which a two-dimensional RMLM is feasible and the effect of {possible} confounders of the two nodes can be ignored.

\citet{KK} and \citet{BK} `manually' select four such components to model extremal causal dependence under the questionable assumption that the dependence structure can be approximated by an RMLM with no hidden confounders. In contrast, our procedure serves as a building block for selecting the nodes that compose an RMLM. Instead of assuming the irrelevance of hidden confounders, Algorithm~\ref{datdalg2} ensures that this approximately holds. 

We standardise the data to Fr\'echet$(2)$ margins using the empirical integral transform~\eqref{empstand} and then proceed as for the standardised setting described in Appendix~\ref{sec:estscal}. {Appendix~\ref{discuss} discusses the influence of innovations with different tail indices to the extremal causal dependence.}
The parameters are chosen as in the simulation study of Appendices~\ref{sec:performance} and~\ref{Ap:sim}: we fix $\eps_1=0.25, \,\eps_2=0.01,\, \eps_3=0.07,\, \eps_4=0.01,\, \eps_5=0.07,\, \eps_6=0.2, a=1.0001$, and we use $k_1=500$ for the intermediate thresholding and $k_2=200$ for the final threshold.
If the conditions in line 6 of Algorithm~\ref{datdalg2} are satisfied, we set those entries $(i,j)$ of the matrix $\hat{P}$ to unity, indicating that $(i,j)\in$~MWP. 
Appendix~\ref{simustud} considers the sensitivity of the estimated max-weighted pairs to changes in $a$, $\eps_3$, $\eps_4$~and~$\eps_5$.
 
The  matrices $\hat{\Delta}^{(2)}$ in Figure~\ref{fig:foodpairs} show non-zero values of ${\hat{\tau}_{ij}^2}$, for which Algorithm~\ref{datdalg2} outputs $\hat{P}_{ij}=1$, indicating that $(i,j)\in$~MWP.
By Theorem~\ref{incdagmod} (i), for each non-zero entry $(i,j)$~we~obtain the two-dimensional RMLM $(X_i,X_j)$ with edge $j\to i$.~As we learnt from the simulation results for DAGs of dimension $D=40$ and $p=0.1$ in Appendix~\ref{Ap:sim}, some of the estimated nodes in MWP correspond~to false positives related only by confounders.~The left-hand matrix $\hat{\Delta}^{(2)}$ reveals a large number of estimated MWPs, but in most cases the dependence is rather weak. 
By Remark~\ref{remarkai} and by symmetry, asymptotic independence implies that $\Delta^{(2)}_{ij}=\Delta^{(2)}_{ji}$, so we set $\eps_4=0.05$ to further filter out weakly dependent pairs. 
This lowers the estimated number of pairs of nodes in MWP, as shown by the right-hand panel of Figure~\ref{fig:foodpairs}. 

\begin{figure}[t]\centering 
	\vspace{-0cm}\hspace*{-0cm}\includegraphics[ height=7cm, width=14cm, clip]{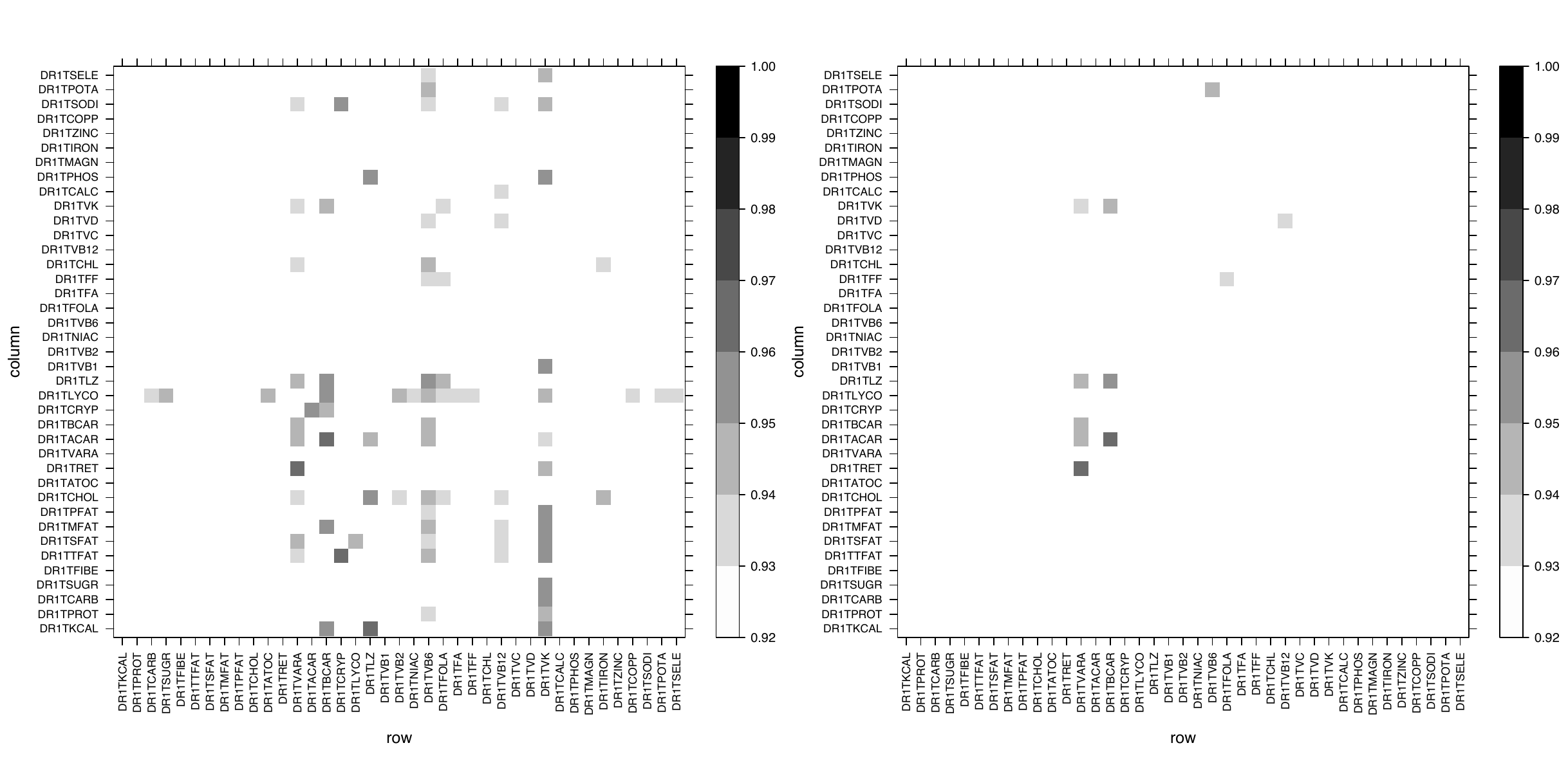}\vspace{-.5cm}
	\caption{Matrices $\hat{\Delta}^{(2)}$ {with entries $\hat{\tau}_{ij}^2$} for all pairs $(i,j)$, where Algorithm~\ref{datdalg2} outputs $\hat{P}_{ij}=1$  for $\eps_4=0.01$ (left) and $\eps_4=0.05$ (right).  
\label{fig:foodpairs}}\vspace{-.5cm}
\end{figure}

The matrix $\hat{\Delta}^{(2)}$ shown on the right-hand side of Figure~\ref{fig:foodpairs} has non-zero entries only for the following twelve nutrients:
AC (DR1TACAR, $\alpha$-Carotene),
BC (DR1TBCAR, $\beta$-Carotene),
VA (DR1TVARA, Vitamin A),
LZ (DR1TLZ, Lutein+Zeaxanthin),
VK (DR1TVK, Vitamin K),
VB12 (DR1TVB12, Vitamin B12), FF (DR1TFF, Food Folate), FA (DR1TFA, Folic Acid), P (DR1TP, Potasium), VB6 (DR1TVB6, Vitamin B6), VD (DR1TVD, Vitamin D), and VB12 (DR1TVB12, Vitamin B12).


{The pairs identified in MWP allow us to construct a larger DAG than in \citet{KK}, with six of the 39 observed nutrients, thus revealing further potential causal relations between the extremes of nutrient intake. Some of these relations align with prior  knowledge, for instance, that the carotenoids BC and LZ are precursors of VA  \citep[Section~1]{giordano2018lutein}. }

We also identify some separate bivariate RMLMs, for instance FF $\to$ FA, P $\to$ VB6, and VD $\to$ VB12, but we cannot establish connections between the corresponding DAGs due to the asymptotic dependence between the source nodes of the respective bivariate DAGs. This dependence indicates the possibility of hidden confounders, analogous to the DAG $\mathcal{D}_2$ in Figure~\ref{introex} or to that in Figure~\ref{partialdag}.

A closer look shows that the pair (LZ,VK) does not appear in any of the sets $S_1, \ldots,S_4$, but exhibits very strong extremal dependence and symmetry, with $\sigma_{M_{ij}}^2= 1.19$. For this pair the chosen $\eps$'s seem too small, but larger values led to many false positives. In this particular case, we draw the undirected red edge to represent their indistinguishability.

We use the matrix $\hat{P}$ and the pairs in MWP to construct the DAG depicted in the left-hand panel~of Figure~\ref{founddag3},~with directed edges between the nutrients~$(i,j)$~for non-zero entries~$\hat{P}_{ij}$~of the matrix $\hat{P}$.~Proposition~\ref{dst} implies that we can identify a directed spanning tree solely from the set MWP and the asymptotic independence relations. The DAG on the left of Figure~\ref{founddag3} shows that all causally dependent pairs lie in MWP and that VA is a candidate for the sink node.
No pairs among RET, AC and LZ lie~in~MWP. By Proposition~\ref{dst} (i), pairs of nodes that are not causally related on a directed spanning tree must be asymptotically independent, and we estimate $\hat{\sigma}_{M_{ij}}^2\approx 1.9$ for these three nodes; recall that $\sigma_{M_{ij}}^2=2$ corresponds to asymptotically independent extremes.  
The pairs  (AC, VK), (RET, BC) and (RET, VK) also show weak extremal dependence; combined with the set MWP this helps us verify conditions (i) and (ii) of Proposition~\ref{dst}, and thus shows that the directed spanning tree with VA as sink~could~support~the~RMLM.

The right-hand image in Figure~\ref{founddag3} corresponds to the minimal representation DAG $\mathcal{D}^O$ 
and to the directed spanning tree generated from the minimal representation~\eqref{minrep}, and is obtained by drawing only the estimated max-weighted paths.

	\begin{figure}[t]
		\centering
\begin{tikzpicture}[
					> = stealth, 
					shorten > = 1pt, 
					auto,
					node distance = 1.5cm, 
					semithick 
					]
					\tikzstyle{every state}=[
					draw = black,
					thick,
					fill = white,
					minimum size = 4mm
					]
					\node[state,  ] (4) {AC};
					
					\node[state] (5) [left of=4]{\footnotesize RET};
					\node[state] (3) [right of=4]{LZ};
					\node[state] (2) [below of=4] {BC};
					\node[state] (6) [below of=3]{VK};
					\node[state] (1) [left of=2] {VA};
					
%
%
%
					\path[->][blue] (4) edge node {} (2);
                    \path[->][blue] (4) edge node {} (1);
                    \path[->][blue] (3) edge node {} (1);
                    \path[->, bend right][blue] (6) edge node {} (1);
                    \path[->][blue] (3) edge node {} (2);
					\path[->][blue] (5) edge node {} (1);
					\path[-][red] (3) edge node {} (6);
					\path[->][blue] (6) edge node {} (2);
%
%
					
					\path[->][blue] (2) edge node {} (1);
					
					\end{tikzpicture}
  \qquad 
	\begin{tikzpicture}[
					> = stealth, 
					shorten > = 1pt, 
					auto,
					node distance = 1.5cm, 
					semithick 
					]
					\tikzstyle{every state}=[
					draw = black,
					thick,
					fill = white,
					minimum size = 4mm
					]
					\node[state,  ] (4) {AC};
					
					\node[state] (5) [left of=4]{\footnotesize RET};
					\node[state] (3) [right of=4]{LZ};
					\node[state] (2) [below of=4] {BC};
					\node[state] (6) [below of=3]{VK};
					\node[state] (1) [left of=2] {VA};
					
%
%
%
					\path[->][blue] (4) edge node {} (2);
					\path[->][blue] (5) edge node {} (1);
					\path[-][red] (3) edge node {} (6);
					\path[->][blue] (6) edge node {} (2);
     \path[->][blue] (3) edge node {} (2);
%
%
					
					\path[->][blue] (2) edge node {} (1);
					
					\end{tikzpicture}
		\caption{Left: DAG corresponding to $\hat{\Delta}^{(2)}$ as estimated in the right-hand panel of Figure~\ref{fig:foodpairs}, with directed edges (blue) having $\hat P_{ij}=1$ and undirected edge (red) for the indistinguishable pair (VK, LZ) having $\hat P_{ij}=0$ and $\hat P^{*}_{ij}=1$.
  Right: \DAG \, with only max-weighted paths, corresponding to the minimal representation DAG $\D^O$.}
  \label{founddag3}\vspace{-.51cm}
	\end{figure}

\section{A functional central limit theorem for random sample sizes} \label{funcCLT}

{We close the paper} by stating a functional central limit theorem which aids in establishing consistency and asymptotic normality of the estimators derived from the empirical angular measure for a random sample size. 
Below $\bs Y\in {\rm RV}^2_+(\alpha)$  
with angular decomposition $(R, \bs \omega)$ and {normalised angular measure $\tilde{H}_{\boldsymbol{Y}}(\cdot)={H}_{\boldsymbol{Y}}(\cdot)/{H}_{\boldsymbol{Y}}({\Theta}^{1}_+)$, a probability measure  on ${\Theta}^{1}_+$.}
Appendix~\ref{sec:stat} contains the motivation, proofs and details about the random sample size $N_n$ resulting from the two-step (or intermediate)~thresholding~procedure.  

The estimators of the extremal dependence measure developed in Appendices~\ref{sec:stat},~\ref{scalest}, and used in Algorithm~\ref{datdalg2}, are {empirical versions of \eqref{p2empdist} and by vague convergence \citep[e.g.,][Proposition 4]{lars}, we find that
$${\mathbb{E}}_{\tilde{{H}}_{\bsy}}[f(\boldsymbol{\omega})]=\lim_{y\to\infty}
 {\mathbb{E}}_{\tilde{{H}}_{\bsy}}[f(\bs \omega)\mid R>y]= \int_{\Theta_+^{1}} f(\bs \omega) {\rm d}{ \tilde H}_{\bs Y}(\bs \omega).$$}
{Thus the estimators take the form}
\begin{align}\label{finalest11}
\hat{\mathbb{E}}_{\tilde{{H}}_{\bsy}}[f(\boldsymbol{\omega})]=\frac{1}{k_2}\sum_{\ell=1}^{N_n}f(\boldsymbol{\omega}_\ell)\mathds{1}\{R_\ell \geq R^{(k_2)}\},
\end{align}
where $N_n$ is a random process in $\mathbb{N}$ and  $R^{(k_2)}$ is the $k_2$-th largest order statistic amongst $R_1,\ldots, R_{N_n}$. 

We apply \eqref{finalest11} to $\bs Y=\bs{\tilde{T}}^{ij}\in {\rm RV}^2_+(2)$ and set $f(\boldsymbol{\omega})=2\omega_1\omega_2$. 
This yields the estimate $\hat{\tau}_{ij}^2={\hat{\mathbb{E}}_{\tilde{{H}}_{{\tilde{\bs T}}^{ij}}}}[2\omega_1\omega_2]$ of the extremal dependence measure in \eqref{tau}, which is used in Algorithm~\ref{datdalg2} to estimate MWPs.


The next result shows that, {for $\sigma^2={\rm Var}_{\tilde{{H}}_{\bsy}}(f({{\boldsymbol{\omega}}_1}))$}, under appropriate conditions and as $n\to\infty$, 
\begin{align}\label{w11}
W_{k_1}(t,s)=\frac{1}{\sigma\sqrt{k_2}}\sum_{i=1}^{\lfloor k_1t\rfloor}\Big(f(\boldsymbol{\omega}_i)-\mathds{E}_{\tilde{H}_{\boldsymbol{Y}}}[f(\boldsymbol{\omega}_1)]\Big)\mathds{1}\{R_i/{b_{{k_1}/{k_2}}}\geq s^{-1/\alpha}\},
\end{align}
converges weakly in the Skorokhod space $D([0,\infty)^2)$ to a Brownian sheet $W$, i.e., 
	a Wiener process on $[0,\infty)^2$ with covariance function $(t_1\wedge t_2)(s_1\wedge s_2)$ for $(t_1,s_1), (t_2,s_2)\in\R_+^2$. 
 We write $\overset{P}{\to}$ for convergence in probability, $\overset{w}{\to}$ for weak convergence in $D([0,\infty)^2)$, and $\overset{\D}{\to}$ for convergence in distribution. 


\begin{theorem}\label{asnorm1}
	Let $\{\bsy_i: i\geq 1\}$ be independent replicates of the standardised vector $\bsy\in {\rm RV}^2_+(\alpha)$ with angular decomposition $(R,\boldsymbol{\omega})$, and let $R$ have distribution function $F$ and survival function $\bar{F}=1-F$. 
	As $n\to\infty$, choose $k_1(n), k_2(k_1)\to\infty$ such that $k_1=o(n), k_2=o(k_1)$, and choose normalising constants $b_{k_1/k_2}$ such that $\bar F(b_{ k_1/k_2} ) \sim k_2/k_1$.
	Then if 
	\begin{align}\label{assump1}
	\lim\limits_{n\to \infty}\sqrt{k_2}\bigg(\frac{k_1}{k_2}
	\mathbb{E}[f( \boldsymbol{\omega}_1 )\mathds{1} {\{R_1\geq b_{ k_1/k_2} s^{-1/ \alpha} \}}] 
	- \mathbb{E}_{\tilde{H}_{\boldsymbol{Y}}} [f(\boldsymbol{\omega}_1)] \frac{k_1}{k_2} \bar{F}(b_{ k_1/k_2}
	s^{-1/\alpha}) \bigg)=0
	\end{align}
	locally uniformly for $s\in[0, \infty)$, {and for $W_{k_1}(t,s)$ as in \eqref{w11} with $\sigma^2={\rm Var}_{\tilde{{H}}_{\bsy}}(f({{\boldsymbol{\omega}}_1}))>0$,} 
	\begin{align}\label{eq:limitest1}
	W_{k_1}(t,s)\overset{w}{\to}W(t,s), \hspace{5mm} n\to\infty.
	\end{align}
\end{theorem}

Consistency and asymptotic normality of~\eqref{finalest11}, established in the next theorem, follow on noting that for sufficiently large $n$,  $(\sqrt{k_2}/\sigma) (\hat{\mathbb{E}}_{\tilde{{H}}_{\bsy}}[f(\boldsymbol{\omega})]-\mathbb{E}_{\tilde{H}_{\boldsymbol{Y}}} [f(\boldsymbol{\omega}_1)])=W_{k_1}(N_n/k_1, (R^{(k_2)}/b_{k_1/k_2})^{-\alpha})$, and that, as shown in Lemma~\ref{randomRb}, $R^{(k_2)}/b_{k_1/k_2}\overset{P}{\to}1$.

\begin{theorem}
	In the setting of Theorem~\ref{asnorm1}, let $N_n\in\mathbb{N}$ be a random process for which ${N_n/k_1 \stp 1}$, let $\bsy_1,\dots,\bsy_{N_n}$ be a random number of independent replicates of $\bsy$, and take $\hat{\mathbb{E}}_{\tilde{{H}}_{\bsy}}[f(\boldsymbol{\omega})]$ as in~\eqref{finalest11}.~Then \begin{align}\label{eq:limitest2}
	\sqrt{k_2}(\hat{\mathbb{E}}_{\tilde{H}_{\boldsymbol{Y}}}[f({\boldsymbol{\omega}_1})]-\mathds{E}_{\tilde{H}_{\boldsymbol{Y}}}[f({\boldsymbol{\omega}_1})])\,{\overset{\D}{\to}}\,\mathcal{N}(0, \sigma^2), \hspace{5mm} n\to\infty.
	\end{align}
\end{theorem}


\section{Conclusions}\label{sec:summary}
Under the realistic assumption that certain nodes of an RMLM supported on a DAG may be unobserved, 
we have given necessary and sufficient conditions for modelling a reduced RMLM 
and a graphical algorithm to construct RMLMs for of a (sub)set of the observed nodes. 
The relation of max-weighted paths between pairs of nodes and the extremal dependence measure of transformed observations is crucial in constructing such reduced models for regularly varying RMLMs. We also provide a statistical algorithm to find pairs of nodes (and larger sets of nodes) that can be modelled by regularly varying RMLMs, study it by simulation and apply it to nutrition intake data. A new functional CLT shows that the estimators of the extremal dependence measure are consistent and asymptotically normally distributed. 

\begin{acks}[Acknowledgments]
We thank the reviewers for unusually constructive comments that have greatly improved the work. 
\end{acks}
\begin{funding}
The work was funded by the Swiss National Science Foundation (project 200021\_178824).
%
\end{funding}

\bibliography{texrefs}
\bibliographystyle{imsart-nameyear.bst}


\newpage

\appendix
\noindent{\Large{\textbf{Appendix}}}
\\
\vspace{0mm}
\setcounter{table}{0}
\renewcommand{\thetable}{A.\arabic{table}}

\noindent The Appendix has seven sections. 

Appendix~\ref{sec:3.4} continues from Section~\ref{sec:MWP} and extends the methodology for regularly varying RMLMs by using knowledge of the identified ancestors among the observed nodes. This serves the purpose of lines 8--17 of Algorithm~\ref{datdalgrmlm}. Some of the findings in Appendix~A rely on results from multivariate regular variation (Appendix~\ref{sec:ARV}) and properties of the coefficient matrix (Lemma~\ref{ineq} in Appendix \ref{Ap:ProofsRMLM}). In order not to interrupt the argument when extending the  identification results of Section~\ref{sec:MWP}, we make forward references to these appendices.

Appendix~\ref{Ap:ProofsRMLM} gives the proofs of the main theorems of Sections~\ref{sec:RMLM},~\ref{sec:RV} and Appendix~\ref{sec:3.4}.

Appendix~\ref{sec:ARV} provides standard definitions and results on regular variation, emphasising those relevant for RMLMs.

Appendix~\ref{sec:stat} proposes an intermediate thresholding procedure that requires what seems to be a novel functional CLT for a random sample, and establishes consistency and asymptotic normality of the estimated scalings and extremal dependence measures given in Section~\ref{funcCLT}.

Appendix~\ref{scalest} uses the previous results to estimate the inputs of Algorithm \ref{datdalg2}.

Appendix~\ref{sec:performance} evaluates the performance of Algorithm~\ref{datdalg2} by simulation, in terms of true and false positive rates for the estimated max-weighted paths enriched by various categories of causal dependence. 

Appendix~\ref{Ap:sim} contains boxplots and results on the sensitivity of Algorithm~\ref{datdalg2} from both the simulation study and the data example, and a discussion concerning the effect of innovations with different tail indices on causal extremal dependence.

\section{Identifying max-weighted paths: the set ${\rm MWP}(\An(K)^c)$}\label{sec:3.4}

We consider the setting of Section~\ref{sec:MWP}, but investigate the influence of a subset $K$ of observed nodes on $X_i$ and $X_j$ for $i,j\in O\setminus K$.  We assume that $i\notin\an(j)$, so we can order them such {that $i<j$.} 
Similar to the discussion preceding equation~\eqref{eq:Aim}, this is ensured by the condition of Lemma~\ref{lemmackappa}, or by~\eqref{cond2} of Theorem~\ref{inc2dagmod}. This parallels the situation in lines 8--17 of Algorithm~\ref{datdalgrmlm}, once $K$ contains at least one node. The objective is to establish a criterion based on the extremal dependence measure to ensure that $(i,j)\in {\rm MWP}(\An(K)^c)$.
Throughout this section we need the following assumptions:

\medskip \noindent
{\bf {Assumptions B:}} \label{assump2}
\begin{enumerate}
	\item[(B1)]
	the random variables indexed by the nodes in $K$ can be represented as an RMLM;
	\item[(B2)]
	all observed ancestors of the nodes in $K$ lie inside $K$, i.e., $\An(K)\cap O\subset K$; and 
	\item[(B3)]
	if $i, j$ have a common hidden confounder $u$ with {$k_1\in K$, i.e., $u\in\an(i)\cap\an(j)\cap\an(k_1)\cap O^c$,} then there exist max-weighted paths $u\rightsquigarrow k_2 \rightsquigarrow i $, $u\rightsquigarrow k_2 \rightsquigarrow j $ for $k_2\in \An(k_1)\cap K$.
\end{enumerate}
The goal is to infer whether we can obtain a larger RMLM by adding the node variables $X_i$ and $X_j$ to those representing the observations on nodes in $K$.

Assumptions (B1)--(B3) follow naturally from the causal ordering of the nodes and statements (i)--(ii) of Theorem~\ref{prop.dat}, and the iterative nature of Algorithm~\ref{datdalgrmlm}. 
In particular, Assumption (B3) implies that when considering the extension of an RMLM by adding the node variables $(X_i,X_j)$,
we may disregard all innovations indexed by nodes in $\An(K)$. By definition of the set $V^K$ in Algorithm~\ref{datdalgrmlm}, this is equivalent to $i,j\in V^K$.
Therefore, the only relevant innovations are those indexed in $\An(i)\cup \An (j) \cap \An(K)^c$.

Recall that Lemma~\ref{lemmaid1} (i) states that there is a max-weighted path $k\rightsquigarrow j \rightsquigarrow i$  for a pair of nodes $(i, j)$ if and only if the relation $a_{jj}a_{ik}=a_{jk}a_{ij}$ holds for some $k\in\An(i)\cap\An(j)=\An(j)$.  
The previous paragraph implies that  it suffices that this relation now holds for all $k\in \An(i)\cap \An (j) \cap \An(K)^c$.

To make this mathematically precise, we define {the random variables} 
$$M_{K}\coloneqq \vee_{k\in K} X_k\quad\mbox{and}\quad M_{ K, r}\coloneqq \vee_{k\in K \cup \{r\}} X_k,$$ 
which, by Proposition~\ref{prop:easy}, can be represented as 
\begin{align}\label{eq:MK}
M_{ K, r} &= \bigvee_{k_1\in \An(r)} a_{r k_1}Z_{k_1}  \vee \underset{o\in\An(K) }{\bigvee} \bigvee_{k_2\in\An(o)} a_{ok_2}Z_{k_2}\nonumber \\
&=\underset{k_1\in\An(r)\setminus\An(K)}{\bigvee}a_{rk_1}Z_{k_1} \vee \underset{o\in\An(K) }{\bigvee} \bigvee_{k_2\in\An(o)} a_{ok_2}Z_{k_2}\nonumber\\
&=:\underset{k_1\in\An(r)\setminus\An(K)}{\bigvee}a_{rk_1}Z_{k_1} \vee M_{K},\quad r\in\{i,j\},
\end{align}
where for the second line we have used Assumption~(B3). 
By Lemma~\ref{lemmaid1} (i), the latter assumption implies that for $k_1\in \An(j)\cap\An(K)\cap O^c$, $a_{jk_1}a_{oo}=a_{ok_1}a_{jo}$ for some $o\in K\cap\An(j)$. Using Lemma~\ref{ineq}, this gives $a_{jk_1}<a_{ok_1}$.

With $M_{K}\coloneqq\vee_{o\in\An(K) }\vee_{k_2\in\An(o)} a_{ok_2}Z_{k_2}$ as in \eqref{eq:MK} define  
\begin{align}\label{eq:T3kappa}
\boldsymbol{T}^{K ij}_3 =({T}^{K ij}_{31},{T}^{K ij}_{32}) \coloneqq (M_{K,i}-M_{K},  M_{K,j}-M_{K}).
\end{align}

As in the discussion leading to Lemma~\ref{lemmaid1}, {we now connect} the extremal dependence measure of transformations of $\boldsymbol{T}^{K ij}_3$ with membership of $(i, j)$ to ${\rm MWP}(\An(K)^c)$. 
By Lemma~\ref{genmaxspect5CK} the atoms of the angular measure of $\boldsymbol{T}^{Kij}_3$ are obtained by normalising the non-zero columns of the matrix 
\begin{align}\label{ATim3}
A_{\boldsymbol{T}^{Kij}_3}= \begin{bmatrix}
0& \cdots & a_{ii} & \cdots & a_{i,j-1} & a_{ij} &  0  & a_{i,j+2} & \cdots & 0  & a_{ik} & 0  & \cdots & 0\\
0& \cdots & 0 & \cdots & 0 & a_{jj} &  0  & a_{j,j+2} & \cdots & 0 & a_{jk}& 0 & \cdots & 0
\end{bmatrix}\in\R^{2\times D},
\end{align}
where the zero columns indices correspond to innovations indices in $\An(K)$, and only the columns indexed in $(\An(i)\cup \An (j) )\cap \An(K)^c$ are non-zero.

This leaves us in a setting parallel to that of Section~\ref{sec:MWP}
and we note that Lemma~\ref{lemmaid1}~(i)-(ii) remain valid for the vector $\boldsymbol{T}^{Kij}_3$ and the matrix in \eqref{ATim3} with indices $k\in\An(j)\cap\An(K)^c$.
We state a modified version of Lemma~\ref{lemmaid1}~(iii) for the pair $(X_i, X_j)$ which uses $K$.

\begin{lemma}\label{lemmackappa}
Consider the subvector {$(X_i, X_j, X_K)\in{\rm RV}_+^{2+|K|}(2)$ from an RMLM satisfying Assumptions~A and~B}, and let $(a_{ij},a_{jj})$ and $(a_{ik},a_{jk})$ be the $j$-th and $k$-th columns of $A_{ij}$ in ~\eqref{eq:Aim}.
	If there exists some $a>1$ such that
	\begin{align*}
	 \sigma_{M_{i,aj,{a}K}}^2=\sigma_{M_{i,j,K}}^2 +(a^2-1)\sigma_{M_{{j,K}}}^2\quad{\rm and}\quad 
	\sigma_{M_{ai,j,{a}K}}^2<\sigma_{M_{i,j,K}}^2 +(a^2-1)\sigma_{M_{i,{K}}}^2,
		\end{align*}
		then $i\notin \an(j)$, $\An(i)\cap\An(j)\cap\An(K)^c\neq \emptyset$, and $a_{jk}\geq a_{ik}$ for all $k\in\An(j)\cap\An(K)^c$;
	otherwise, either $\An(i)\cap \An(j)\cap\An(K)^c=\emptyset$, or  there exists some $k\in\An(j)\cap\An(K)^c$ such that there are no max-weighted paths $k\rightsquigarrow j\rightsquigarrow i$.
\end{lemma}

{This suggests that we apply to $\boldsymbol{T}^{Kij}_3$ from \eqref{eq:T3kappa} the procedure as applied to $\boldsymbol{T}^{ij}$ in \eqref{def:tim}.
Table~\ref{table:atomsO} provides the atoms of the angular measure for the vectors leading to  $\boldsymbol{T}^{Kij}_3$. }

\begin{table}[htbp] 
	{\centering\caption{Vectors $\tilde a_k$ used to obtain the atoms of transformations $(X_i,X_j, M_{K})$ in equation~\eqref{eq:T3kappa}. \label{table:atomsO}}}
	\centering
	\fbox{%
		\begin{tabular}{l|l|l}
			Notation & Vector & $\tilde a_k $  \\
			\hline
			$\boldsymbol{T}^{Kij}_1$ & $(M_{K},X_i,X_j)$ & $(\underset{r\in K}{\vee}a_{rk}, a_{ik}, a_{jk})$ \\
			$\boldsymbol{T}^{Kij}_2$ & $(M_{K},M_{K,i},M_{K,j})$ & $(\underset{r\in K}{\vee}a_{rk}, \underset{r\in K\cup\{i\}}{\vee}a_{rk}, \underset{r\in K\cup\{j\}}{\vee}a_{rk})$\\
			$\boldsymbol{T}^{Kij}_3$ & $(M_{K,i}-M_{K},M_{K,j}-M_{K})$ & $(\underset{r\in K\cup\{i\}}{\vee}a_{rk}-\underset{r\in K}{\vee}a_{rk} , \underset{r\in K\cup\{j\}}{\vee}a_{rk}-\underset{r\in K}{\vee}a_{rk}, )$ \\
		\end{tabular}}
	\end{table}
	
	Now, define for $0<c_1<1$ and $c_2>0$
	\begin{align}\label{TKij}
    M^K_{c_1i,c_2j} &\coloneqq\max(c_1{T}^{Kij}_{31}, c_2{ T}^{Kij}_{32}), and\nonumber\\
	\boldsymbol{T}^{Kij} &\coloneqq(M^K_{c_1i,j}-c_1{T}^{Kij}_{31}, (1+c_2) T^{Kij}_{32}+c_2{T}^{Kij}_{31}-c_2M^K_{ij});
	\end{align}
    this is a linear function of $(M^K_{c_1i,j},M^K_{ij},{T}^{Kij}_{31},{T}^{Kij}_{32})$.
    Table~\ref{table:atomsO} is a consequence of Lemmas~\ref{breimanlemma}--\ref{genmaxspect5CK}; \ref{genmaxspectlast} gives the atoms of the angular measure of $\boldsymbol{T}^{Kij}$.

\begin{lemma}\label{lemmaid3}
		Let $\boldsymbol{T}^{K ij}$ be as in~\eqref{TKij} for $0<c_1\leq 1$ and $c_2>0$.
		If the condition of Lemma~\ref{lemmackappa} holds, 
		then $\boldsymbol{T}^{K ij}\in\mathrm{RV}_+^2(2)$ has discrete angular measure
        with atoms $\tilde{a}_k/\|\tilde{a}_k\|$ derived from Table~\ref{table:atomsO}  for
         non-zero vectors 
         $\tilde{a}_k=(c_1a_{ik}\vee a_{jk} -c_1a_{ik}, (1+c_2)a_{jk}+c_2a_{ik}-c_2(a_{ik}\vee a_{jk}))$
         for $k\in\An(j)\cap\An(K)^c$.
		Moreover, $\tilde a_k=(0,0)$ for $k\notin\An(j)\cap\An(K)^c$. 
	\end{lemma}
    
	Let $\boldsymbol{\tilde T}^{Kij}=(\tilde T_1^{Kij},\tilde T_2^{Kij})$ denote the standardised version of $\boldsymbol{T}^{Kij}$, analogous to $\boldsymbol{T}^{i j}$ and $\boldsymbol{\tilde T}^{i j}$ defined in \eqref{eq:Tscaled}. The corresponding Lemmas~\ref{breimanlemma}--\ref{genmaxspectlast} can be found in Appendix~\ref{sec:ARV}. 
	The next theorem uses the extremal dependence measure $\tau_{Kij}^2\coloneqq\sigma_{{\tilde T}^{Kij}_{1},{\tilde T}^{Kij}_{2}}^2$ between the components of $\bs{\tilde{T}}^{K ij}$ to provide necessary and sufficient conditions for a pair $(i, j)$ to lie in ${\rm MWP}(\An(K)^c)$. The proof is given in Appendix~\ref{Ap:ProofsRMLM}. 
	
	\begin{theorem}\label{inc2dagmod}
		Let {$\boldsymbol{X}\in {\rm RV}_+^D(2)$} be an RMLM satisfying Assumptions~A and~B. 
		Suppose that we observe nodes $i,j\in V$ such that $\An(K)\cap\{i,j\}=\emptyset$, and that for some $a>1$, 
		\begin{align}\label{cond2}
		\sigma_{M_{i,aj,{a}K}}^2=\sigma_{M_{i,j,K}}^2 +(a^2-1)\sigma_{M_{{j,K}}}^2,
		\quad\quad
		\sigma_{M_{ai,j,{a}K}}^2<\sigma_{M_{i,j,K}}^2 +(a^2-1)\sigma_{M_{i,{K}}}^2.
		\end{align}  
		Consider $\bs{{T}}^{K ij}$ as in \eqref{TKij} with $0<c_1\leq 1$ and $c_2>0$ and its standardised version $\bs{\tilde{T}}^{K ij}$. Then:
		\begin{enumerate}
			\item[(i)] $(i,j)\in {\rm MWP}(\An(K)^c)$ if and only if 
			$\tau_{Kij}^2=1$  
			In this case, $(X_i,X_j, \boldsymbol{X}_{K})$ can be represented as an RMLM;
			\item[(ii)] if $i\notin \emph{de}(j)$, then
			$\tau_{Kij}^2<1$, and $(X_i,X_j, \boldsymbol{X}_{K})$ cannot be represented as an RMLM.
		\end{enumerate}
	\end{theorem}

	We refer to Example~\ref{exdag} for an application of Theorem~\ref{inc2dagmod}. 
	\begin{example}
		Consider an RMLM satisfying the conditions of Theorem~\ref{inc2dagmod} and supported on the DAG in Figure~\ref{fig:M1}. 
        Let $K=\{8,9,10\}$ contain the known (source) nodes. 
		Indeed, this represents a trivial RMLM consisting of three source nodes, thus satisfying Assumptions~B. 
        Suppose we wish to know whether $(X_2, X_4, X_8, X_9, X_{10})$ can be modelled as an RMLM. 
        For the pair $(i,j)=(2,4)$, we have
		\begin{align*}
		X_4&=a_{44}Z_4\vee a_{47}Z_7\vee\frac{a_{48}}{a_{88}}X_8\vee\frac{a_{49}}{a_{99}}X_9,\\
		X_2&=a_{22}Z_2 \vee a_{27}Z_7 \vee  \frac{a_{24}}{a_{44}}X_4\vee \frac{a_{28}}{a_{88}}X_8\vee\frac{a_{29}}{a_{99}}X_9,
		\end{align*}
		with hidden node $7$.  
        First, we need to consider the known nodes in $K$. 
        By Lemma~\ref{genmaxspect5CK} the column vectors of the angular measure of  $\boldsymbol{T}^{K 24}_3=(M_{2,8,9,10}-M_{8,9,10}, M_{4,8,9,10}-M_{8,9,10} )$ arranged into the matrix  $A_{\boldsymbol{T}^{K 2 4}_3}$ as in \eqref{ATim3}, will be non-zero for those indices corresponding to $Z_2$, $Z_4$ and $Z_7$. 
        More specifically,
		$$
		A_{\boldsymbol{T}^{K 2 4}_3}=
		\begin{bmatrix}
		0 & a_{22} & 0 & a_{24} & 0 & 0 & a_{27} & 0 & 0 & 0 \\
		0 & 0 & 0 & a_{44} & 0 & 0  & a_{47} & 0 & 0 & 0
		\end{bmatrix},
		$$
		which removes the effect of $Z_8$, $Z_9$ and $Z_{10}$.
		We check whether the path $7\to4\to2$ is max-weighted, so that we can disregard the term $ a_{27}Z_7$, and hence ensure exogeneity of the innovation $a_{44}Z_4\vee a_{47}Z_7$.  
        If condition \eqref{cond2} holds, then this path is max-weighted, that is $(2,4)\in {\rm MWP}(7)$ if and only if Theorem~\ref{inc2dagmod} $(i)$ is satisfied for $\boldsymbol{\tilde{T}}^{K 24}$, in which case we can ignore the effect of node $7$.  Using Proposition~\ref{prop.incmat}, we then obtain
		\begin{align*}
		X_2=a_{22}Z_2\vee \frac{a_{24}}{a_{44}}X_4,\quad X_4=a_{44}^{*}Z^{*}_4\vee\frac{a_{48}}{a_{88}}X_8\vee\frac{a_{49}}{a_{99}}X_9,
		\end{align*}
		where $a_{44}^{*}=(a_{44}^2+a_{47}^2)^{1/2}$, and $Z_4^{*}=(a_{44}Z_4\vee a_{47}Z_7)/a_{44}^{*}$. A similar analysis applies for the remaining nodes.
	\end{example}

\section{Proofs of Sections~\ref{sec:RMLM},~\ref{sec:RV} and Appendix~\ref{sec:3.4}} \label{Ap:ProofsRMLM}

The notion of generations of a DAG proves useful throughout the proofs. Here we provide the definition, referring to \citet[Section~2.1]{KK} for further properties.
 \begin{definition}\label{def:gen}
 In a DAG $\mathcal{D}=(V, E)$, a {generation of nodes} is the set of all nodes that have a longest path of same length from any source node. {Let $G_0$ denote the source nodes}, and define the $k$-th generation of nodes by
$$G_k=\{i\in V\setminus\underset{\ell<k}{\cup} G_{\ell}: \underset{p_{ji}\in P_{ji}: \hspace{1mm}j\in G_0}{\max}  {|p_{ji}|=k}\}.$$
\end{definition}

\bproof[\textbf{Proof of Theorem~\ref{prop.dat}}]

We use induction over the number of generations of the observed DAG.  

Skipping the case of a trivial DAG with a single generation $G_0^O$, we initiate the induction by showing that for two observed generations, $G_0^O, G_1^O$, the conditions in (i) and (ii) are necessary and sufficient. 

Regarding necessity, we notice that {an RMLM on a 2-generation DAG consists} of a set of source node variables $G_0^O$ and their observed children, $G_1^O$, connected by at most a path of length one. {Because $\bs X^O$ is an RMLM, the source nodes in $G_0^O$ satisfy (i), and therefore $V_0^O=G_0^O$.}
Then the recursion \eqref{semequat1} of this RMLM reads
\begin{align*}
X_\ell &= c_{\ell \ell} Z_\ell,\quad \ell\in V_o^O,\\
X_i &= \bigvee_{k\in\pa(i)\cap O} c_ {ik} X_k \vee  c_{ii} Z_i,\quad i\in G_1^O.
\end{align*}
This means that the innovations entering the representation of any source node, say $\ell$, can have a max-weighted path to $\ch(\ell)\cap O$ only via $\ell$ 
and hence (i)(a) holds. 
For assertion (i)(b), either $j\in V_0^O$ and then $X_j=c_{jj} Z_j$, or $j\in G_1^O$, but 
$j\notin{\ch(\ell)}$, hence again $X_j=\vee_{k\in\pa(j)\cap O\setminus \{\ell\} } c_ {jk} X_k\vee c_{jj} Z_j$.
By the exogeneity of the innovations of the RMLM, there can be no common hidden ancestor $u\in\an(i)\cap\an(j)\cap O^c\setminus \an(V_0^O)$ for $i,j\in G_1^O$, so the conditions (ii) are void.

Next, we show sufficiency of (i) and (ii). 
Since the observed DAG has two generations, we use representation~\eqref{minrep} for {$i\in G_1^O$ and $\ell\in G_0^O$}, yielding
\begin{align}\label{ells}
X_{i} &= \bigvee_{k\in \an^O(i)}\frac{a_{ik}}{a_{kk}} X_k \vee \bigvee_{k\in \An^{O^c}(i)} a_{ik}Z_k,\quad i\in G_1^O,\nonumber\\
X_{\ell} &= \bigvee_{k\in\An^{O^c}(\ell)} a_{\ell k}Z_k\eqqcolon \tilde{Z}_\ell, \quad \ell\in G_0^O.
\end{align}
We show that an innovation $Z_{u}$ can appear only in one of the equations in~\eqref{ells}, either $u\in\An^{O^c}(i)$ or $u\in\An^{O^c}(\ell)$.
By definition, the former occurs only when there are no max-weighted paths $u\rightsquigarrow \ell\rightsquigarrow  i$ for $\ell\in G_0^O$, and thus, by $(i)(a)$ when $u\in\an(i)\setminus\an(\ell)$. 
Otherwise, if there is a max-weighted path $u\rightsquigarrow \ell\rightsquigarrow  i$, then $u\notin\An^{O^c}(i)$, but $u\in\An^{O^c}(\ell)$.

Now, let  {$j\in O\cap{\rm De}(\ell)^c$}, then $\an(\ell)\cap\an(j)\cap O^c=\emptyset$ and $\an(\ell)\cap O=\emptyset$. Thus, if we consider the representation of $X_j$, we have that $\An^{O^c}(j)\cap \An^{O^c}(\ell)=\emptyset$ by (i)(b). As the nodes $\ell\in G_0^O$ are the only ones that satisfy the properties in (i), we have $G_0^O=V_0^O$.

 {So far we looked at pairs $(i,\ell)$ where $\ell\in G_0^O$ and $i\in G_1^O$ or $i\in O\cap{\rm De}(\ell)^c$}. 
 It remains to investigate $(i,j)$ when both $i,j\in G_1^O$.
Consider the representation of $X_j$, similar to that of $X_i$ in~\eqref{ells}, for $j\in G_1^O\setminus\{i\}$. We want to show that $\An^{O^c}(i)\cap\An^{O^c}(j)=\emptyset$. 
Suppose that there exists $u\in\an(i)\cap\an(j)\cap O^c$ and $u\in \An^{O^c}(j)$. Then, $u\notin \an(\ell)$ for $\ell \in G_0^O$. 
Furthermore, as the observed DAG has two generations, there are no paths between $i$ and $j$ for $i,j\in G_1^O$. So, by (ii)(b) there must exist a node $k\notin G_0^O$ such that $i,j\in{\rm MWP}^k(u)$. However, as the DAG only has two generations, no such node $k$ exists, and therefore there can be no $u$ such that $u\in\an(i)\cap\an(j)\cap O^c$ and $u\in \An^{O^c}(j)$. In particular, this implies that $\An^{O^c}(i)\cap\An^{O^c}(j)=\emptyset$.


Thus, conditions (i) and (ii) suffice to show that we have obtained an RMLM on a DAG with two generations.

{Next, using the inductive hypothesis, suppose that there are $o-1<d$ such nodes belonging to $p-1$ observed generations (${\cup}_{i\leq p-1}{G^O_i}$), and that we can construct an RMLM on a DAG composed of the nodes in ${\cup}_{i\leq p-1}{G^O_i}$ if and only if (i) and (ii) are satisfied. Now, suppose we observe an additional generation, say $G^O_p$, and {assume for the moment that} $G^O_p$ consists of one node, say ${o_{p_1}}$. 
We first show the necessity of (ii). By 
\eqref{minrep} we may write 
	\begin{align}\label{proofrep}
	X_{o_{p_1}} &= \bigvee_{j\in \an^{O}({o_{p_1}})}\frac{a_{{o_{p_1}}j}}{a_{jj}}X_j\vee \bigvee_{j\in  \An^{O^c}({o_{p_1}})} a_{{o_{p_1}}j}Z_j.
	\end{align}
	As before, we need to ensure that the innovations of the hidden ancestors of node ${o_{p_1}}$ appearing in $\An^{O^c}({o_{p_1}})$ are exogenous. We do so by contradiction: suppose that the observed $o$ nodes are generated via an RMLM on a DAG, and that there exists some $u\in O^c\cap \an({o_{p_1}})\cap\an(j)$ for some $j\in{\cup}_{i\leq p-1}{G^O_i}$ such that neither (a) nor (b) of (ii) hold for the pair $({o_{p_1}}, j)$.
	Then,
	\begin{itemize}
		\item[(C1)] for $j\in\an({o_{p_1}})$, no max-weighted paths from $u$ to ${o_{p_1}}$ pass through $j$;
		\item[(C2)]  for every $k\in\an({o_{p_1}})\cap\an(j)\cap O$ there are no max-weighted paths from $u$ to both ${o_{p_1}}$ and  $j$ passing through $k$.
	\end{itemize}
	(C1) ensures that $Z_u$ will enter $X_{o_{p_1}}$ either via some node in $\an^{O}({o_{p_1}})\setminus\{j\}$ or via the set $\An^{O^c}({o_{p_1}})$. 
	In the latter case, $Z_{u}$ entering $\An^{O^c}({o_{p_1}})$ implies that  the innovations involved in $\An^{O^c}({o_{p_1}})$ are not exogenous, as $u$ enters the representation of $j$ as well, either via $\an^{O}(j)$, or $\An^{O^c}(j)$. This contradicts the assumption that the observed DAG corresponds to an RMLM. If $Z_{u}$ enters the representation of $X_{o_{p_1}}$ via $\an^{O}({o_{p_1}})\setminus\{j\}$, then there exists some node $r\in\an^{O}({o_{p_1}})\setminus\{j\}$ such that $u\in O^c\cap \an(r)\cap\an(j)$. 
	Note, however, that by the induction hypothesis, since the nodes in the first $p-1$ generations form an RMLM, and since $j, r\in {\cup}_{i\leq p-1}G^O_i$, either there is a max-weighted path $u\rightsquigarrow r\rightsquigarrow j\rightsquigarrow {o_{p_1}}$ or $u\rightsquigarrow j\rightsquigarrow r\rightsquigarrow {o_{p_1}}$, or there exists some node $k\in\an(r)\cap\an(j)\cap O \cap {\rm de}(u)$ such that there are max-weighted paths from $k$ to both $r$ and $j$. The first two such paths contradict (C1). The third path would also imply that there is a max-weighted path from $k$ to both $j$ and ${o_{p_1}}$, hence contradicting (C2). Similarly, $u\in \An^{O^c}(j)\cap\An^{O^c}({o_{p_1}})$ contradicts the exogeneity of  the innovations composing the two respective sets due to the presence of $Z_u$.}

Under (C2) $Z_{u}$ would appear in $X_j$ via either  $\an^{O}(j)$ or $\An^{O^c}(j)$, and in $X_{o_{p_1}}$ via either $\an^{O}({o_{p_1}})\setminus\{j\}$ or $\An^{O^c}({o_{p_1}})$. Clearly, $u$ appearing in $\An^{O^c}(j)\cap \An^{O^c}({o_{p_1}})$, or in $\an^{O}(j)\cap \An^{O^c}({o_{p_1}})$ contradicts the exogeneity of the innovations involved in $\An^{O^c}({o_{p_1}})$. Similarly, $u$ appearing in $\An^{O^c}(j)\cap\an^{O}({o_{p_1}})\setminus\{j\}$, would imply that there is a max-weighted path from $u$ to ${o_{p_1}}$ via $r\in\an^{O}({o_{p_1}})\setminus\{j\}$. But then $u\in \an(r)\cap \an(j)$, and since $j, r\in {\cup}_{i\leq p-1}G^O_i$, by the induction hypothesis there exists some $k$ such that there are max-weighted paths from $k\in\an(r)\cap\an(j)\cap O \cap {\rm de}(u)$ to both $j$ and $r$, and hence to both $j$ and ${o_{p_1}}$, a contradiction to (C2). Finally,  $u$ appearing through ancestors in $\an^{O}(j)\cap\an^{O}({o_{p_1}})\setminus\{j\}$ implies that there are max-weighted paths from $u$ to $j$ passing through $r_1\in \an^{O}(j)$, and to ${o_{p_1}}$ passing through $r_2\in\an^{O}({o_{p_1}})\setminus\{j\} $. Since $r_1, r_2\in {\cup}_{i\leq p-1}G^O_i$, and $u\in\an(r_1)\cap\an(r_2)$, by the induction hypothesis there exists some $m\in\an^{O}(r_1)\cap\an^{O}(r_2)$ such that there are max-weighted paths from $u$ to both $r_1$ and $r_2$,  and hence also to $j$ and ${o_{p_1}}$, passing through $m$, again a contradiction. 
	
Similar arguments apply when $G^O_p$ has more than one node, {i.e., $o_{p_1}$ and $o_{p_2}$. 
Then we can also write the variable of the latter as 
		\begin{align}\label{proofrep2}
		X_{o_{p_2}} &= \bigvee_{j\in \an^{O}({o_{p_2}})}\frac{a_{{o_{p_2}}j}}{a_{jj}}X_j\vee \bigvee_{j\in  \An^{O^c}({o_{p_2}})} a_{{o_{p_2}}j}Z_j.
		\end{align}
For pairs $(o_{p_2}, j)$ points (C1) and (C2) can be ruled out similarly as for $(o_{p_1}, j)$.  It remains to consider the hidden nodes $u\in O^c\cap \an(o_{p1})\cap\an(o_{p_2})$ for the pair $(o_{p_2}, o_{p_1})$. The point (C1) is void as there is no path between nodes $o_{p_1}$ and $j=o_{p_2}$, both of which lie in $G^O_p$, so we need only consider (C2) for $(o_{p_2}, o_{p_1})$. We only argue against the case of $u$ appearing in $\An^{O^c}(o_{p_2})\cap\an^O(o_{p_1})$; arguments for the other cases are identical to those presented for $j\in\cup_
{i\leq p-1}G_i^O$ in the preceding paragraph. This would imply that there exists a max-weighted path from $u$ to $o_{p_1}$  via some $r\in \an^O(o_{p_1})$. As $u\in\An^{O^c}(o_{p_2})$, we have that $u\in\an(r)\cap\an(o_{p_2})$.  Note, however, that this contradicts the exogeneity of $Z_u$, since $r\in\cup_{i\leq p-1}G_i^O$. We can use the arguments similar to the preceding paragraph to show that (C2) contradicts the formulation of the observed node variables as an RMLM.
This shows the necessity of condition (ii). Note that (i) is a special case of (ii) for $j\in V_0$.}

To show sufficiency when there are more than two generations we use the inductive hypothesis and suppose that (i) and (ii) suffice for a DAG with $p-1$ generations to generate an RMLM. 

Consider the generation $G^O_p$ of a DAG. Let ${o_{p}}\in G^O_p$ and focus on the representation in~\eqref{minrep}.
It suffices to show that any $u\in \An^{O^c}({o_{p}})$ does not appear in any set $\An^{O^c}({j})$ for $j\in O\setminus {o_p}$. 
For any such $j$ we have the three mutually exclusive possibilities:
\begin{itemize}
	\item[(D1)]  $j\in \an({o_{p}})\cap O $;
	\item[(D2)] there exists $ q\in \an(j)\cap \an({o_{p}})\cap O$ and $j\notin\an(o_p)$; or 
	\item[(D3)] $\an(j)\cap\an({o_{p}})\cap O=\emptyset$.
\end{itemize}

Under (D3), suppose that there exists $u\in\an(j)\cap\an({o_{p}})\cap O^c$. This immediately leads to a contradiction to (i)(a) and (ii)(b), since there can exist no node $k\in \an(i)\cap\an(j)\cap O$ such that $i,j\in {\rm MWP}^k(u)$. Thus, we cannot have $u\in\an(j)\cap\an({o_{p}})\cap O^c$ if $\an(j)\cap\an({o_{p}})\cap O=\emptyset$. It follows that $\An^{O^c}({j})\cap \An^{O^c}({o_p})=\emptyset$ for pairs under (D3).

Under (D1), if $j\in \an({o_{p}})$, then by (ii)(a) each $u\in\an(j)\cap\an({o_{p}})\cap O^c$ will have a max-weighted path to ${o_{p}}$ via $j$, or to both $j$ and ${o_{p}}$ via some $k\in\an(i)\cap\an({o_{p}})\cap O$, which ensures that  $u\notin\An^{O^c}({o_{p}})$.

Under (D2), by condition (ii)(b) for every $u$ in $\an(j)\cap\an({o_{p}})\cap O^c$ there exists some  $k$ in $\an(j)\cap \an({o_{p}})\cap O$ such that there are max-weighted paths from $u$ to both $j$ and ${o_{p}}$, respectively, passing through $k$. 
This shows that ${u}$ cannot appear in $\An^{O^c}(o_p)$ or in $\An^{O^c}(j)$.

Therefore, all innovations in $\An^{O^c}({o_{p}})$ must be exogenous to ${o_{p}}$ and cannot enter the representation of the other nodes. Thus, $X_{o_{p}}$ for ${o_{p}}\in G^O_p$ can be represented as a max-linear function of some ancestral nodes and some exogenous innovations via representation \eqref{proofrep}, and so, by induction, $\boldsymbol{X}_O$ can be formulated as an RMLM.
\eproof

\begin{proof}[Proof of Proposition~\ref{graphalgprop}]
    {We note that the membership of a pair $(i,j)$ in MWP or in MWP($\An(K)^c$) for a certain set $K$, is an asymmetric relation (by acyclicity), and implies that $(j,i)$ cannot belong to either set, nor that $\An(i)\cap\An(j)\cap \An(K^c)=\emptyset$.} 
    Suppose first that the {$d$ node variables} indexed in $O$ can be modelled as an RMLM. 
    Then, by Theorem~\ref{prop.dat} (i), in line 7 we must have that $z_j=d-1$ at the end of the initial step for all $j$ that are source nodes {and $z_j<d-1$ for every non-source node $j$.} 
    We select one {source} node, say $K=\{j_0\}$.
	
	In the second step in line 8, $V^K$ will be composed of $d-1$ remaining nodes. 
    {Because the observed node variables form an RMLM, for any source node, say $j$, that was not selected in the initial step, we have $z_j=d-2$, again by Theorem~\ref{prop.dat}(i).} 
    {For nodes $j$ such that $j\in\des(j_0)$ and $\an(j)\cap O=\{j_0\}$,} either of the conditions of Theorem~\ref{prop.dat} (ii) must be met. 
    As these correspond to the conditions in lines 12 and 13 of the algorithm, we have that $z_j=d-2$ at the end of the second step. Again, as before, it holds that $z_j<d-2$ for any other type of node $j$. {We add a new index to $K=\{j_0\}$ as requested in line 16.}
	
	The algorithm proceeds in a similar fashion, augmenting the set $K$ by one index {at each iteration step;} {indeed, because of Theorem~\ref{prop.dat} (ii), 
    at each step $p$ there exists at least one node $j$ which satisfies either of the conditions in lines 12 or 13, giving $z_j=d-p$ {and $|V^K|= d-p$} at the end of the step.} 
	
	In the last step, we have $V^K=\{i_d\}$, composed of a single node, for which  {$z_{i_d}=0$,} after which the RMLM can no longer be extended and the algorithm ends.
	
	{Suppose now that the algorithm outputs all $d$ nodes in $O$.
    As the dimension of $K$ increases by at most one, this implies that the algorithm runs for $d$ iterations, and this occurs only if at each, but the last step, there is at least one $z_j>0$. To see why, note that if $\max(\bs z)=0$ before $d$ iterations, then $|K|<d$ and $V^K=\emptyset$; however, this implies that the algorithm stops and outputs a set $K$ of dimension less than $d$, which is a contradiction.
	

In particular, at each step $p$ we must have $|V^K|=d-p$. To see why, suppose that we have selected $p-1$ nodes, and that $\max(z_j)<d-p$ in the $p$-th step, attained, say, at $j_p$. 
This implies that there exists a node $i_p\in V^K$ such that neither $(i_p, j_p)\in {\rm MWP}(\An(K)^c)$, nor $\An(i_p)\cap\An(j_p)\cap\An(K)^c=\emptyset$. 
After we update $K$ with the new node $j_p$, we see that $i_p$ cannot be a member of $V^K$, hence,
{$|V^K|\le d-p-1$ and $|K|= p$. }
The set $K$ is updated with elements from $V^K$, so that it can contain at most $d-1$ elements, however, this is a contradiction to $|K|=d$ which the algorithm outputs.}
	
	Note that for $j_{p+1}$ to be in $V^K$ with $j_p \in K$, either $\an(j_{p+1})\cap\an(K)=\emptyset$ or $j_p\in {\rm MWP}^K({\An(K)})$. 
    This ensures the exogeneity of the innovations {of node variables} not in $O$ which are common ancestors of $j_{p+1}$ and $K$, verifying conditions (i) and (ii) of Theorem~\ref{prop.dat} for the pairs $(j_{p+1}, k)$, $k\in K$. 
    For this fixed $j_{p+1}$, and for the remaining $d-p-1$ nodes $i\in V^K$, either $\An(i)\cap\An(j_{p+1})\cap\An(K)^c=\emptyset$, or $(i,j_{p+1})\in{\rm MWP}^K({\An(K)^c})$, {or $(j_{p+1}, i)\in{\rm MWP}^K({\An(K)^c})$}. 
    This verifies the relations established in conditions (i) and (ii) of Theorem~\ref{prop.dat} for all pairs that contain nodes from the set {$V^K$, and thus also from $K$. In the last step of the algorithm, we have $K=O$, therefore showing that the nodes in $O$ can be modelled as an RMLM.}
	
	
	Assume now the setting when only strict subsets $K$ of $O$ can be modelled as an RMLM, and assume without loss of generality that $|K|>1$. Note that these may be several, possibly disjoint, subsets as illustrated in Example~\ref{subsets}. Then, the set $K$ in line 7 will be non-empty, since there exists $j$ which satisfies either of the conditions in lines 3 and 4 for at least one node $i\neq j$, giving $z_j>0$; thus we can extend the DAG by at least one node.  In case no non-trivial RMLMs, i.e., with more than one node, can be formed, then none of the conditions in lines 3 and 4 is met, giving $z_j=0$ for all $j\in O$; from these line 7 selects $j_0$, say, $K=\{j_0\}$, and line 8 gives $V^K=\emptyset$, returning a trivial one dimensional RMLM.  
	
 To show that $K$ is well-ordered, suppose that for $K\neq \emptyset$, there exists $i,j\in V^K$, such that $j\in\an(i)$, and $(i,j)\in {\rm MWP}(\An(K)^c)$, so that $\An(j)\subset \An(i)$.
 If $k\in V^K$ is such that $\An(k)\cap\An(i)\cap \An(K)^c=\emptyset$, then it follows that $\An(k)\cap\An(j)\cap \An(K)^c=\emptyset$. Furthermore, because $(i,j)\in {\rm MWP}(\An(K)^c)$, if $(k,i)\in {\rm MWP}(\An(K)^c)$,   it follows directly that $(k,j)\in {\rm MWP}(\An(K)^c)$. Therefore, $z_j>z_i$, implying that $j$ will be selected before node $i$.
	
	That Algorithm~\ref{datdalgrmlm} outputs the DAG constructed via the minimal representation in~\eqref{minrep}, follows directly, since, if $j\in \des(k)\cap\an(i)$ and $(i,j)\in {\rm MWP}(\An(k)\setminus (\an(k) \cap K))$, then there is a max-weighted path $k\rightsquigarrow j \rightsquigarrow i$, and thus $k\notin \An^K(i)$ in representation~\eqref{minrep} of node $i$. 
	\end{proof}

\begin{proof}[\textbf{Proof of Proposition~\ref{prop.incmat}.}]
We consider only the observed nodes $O$ and proceed via induction over the generations {(see Definition~\ref{def:gen})} of the DAG of the observed part $\bsx_O$ of the RMLM.
{Under the conditions of Theorem~\ref{prop.dat}, the sets $V_0^O$ and  $G_0^O$ are the same and consist of the source nodes.}
		We start with the source nodes $V_0^O$.  Now $\An(i)=\An^{O^c}(i)$ for $i\in V^O_0$, and we find from~\eqref{minrep} that
  \begin{align}\label{proofMLrep}
  X_i = \bigvee_{k\in \An^{O^c}(i)} a_{ik} Z_k\, = \, a^{*}_{ii} Z^{*}_{i},
  \end{align}
		where $a_{ii}^{*2}=\sum_{j\in \An (i)} a_{ij}^2$, and $Z^{*}_i= \vee_{j\in \An{(i)}} (a_{ij}/a_{ii}^{*}) Z_j$ by Lemma~\ref{mlcomb}. Moreover $a_{ij}^{*2}=0$ for $j\neq i$ by Theorem~\ref{prop.dat} (i)(b), because the summation covers all ancestors of $X_i$.
  Furthermore the source nodes and hence $Z^{*}_i$ are independent by Theorem~\ref{prop.dat}(i)(b). 
 As $\bsx_O$ lives on a well-ordered DAG, the source nodes correspond to the last components of $\bsx_O$ and the corresponding rows in $A^{*}_O$ have non-zero entries only on the diagonal.

Now consider a node  $i_1$ in generation $G_1^O$, which consists of the children of nodes in $V_0^O$ in $O$. By~\eqref{minrep}, we find
  \begin{align}\label{G1}
		X_{i_1}{=}\bigvee_{k\in \an^{O}(i_1)}\frac{a_{i_1,k}}{a_{kk}} X_k \vee  a_{i_1,i_1}^{*}Z_{i_1}^{*}, 
   \end{align}
   		where $a_{i_1,i_1}^{*2}=\sum_{k\in \An^{O^c}(i_1)}a_{i_1,k}^2$ and, by Lemma~\ref{mlcomb}, $Z^{*}_{i_1}= \vee_{j\in \An^{O^c}(i_1)} Z_j/a_{i_1,i_1}^{*}$.
  
  We now show that $Z^{*}_{i_1}$ is independent, first, of $X_k\in V_0^O$ and, second, of $Z^{*}_{j_1}$, with $i_1, j_1 \in G_1^O$.
 
To prove the first, note that, if 
  $i_1\in G_1^O$ and $i\in V_0^O$ have a common hidden ancestor $u$, then by Theorem~\ref{prop.dat}(i)(a) there must be a max-weighted path $u  \rightsquigarrow i \rightsquigarrow i_1$.
  Therefore, representation~\eqref{G1} for $X_{i_1}$ contains $X_i$ in the first maximum, such that the innovations indexed in $\An^{O^c}(i)$ and $\An^{O^c}(i_1)$ must have different indices, and therefore be independent.  
  
To prove the second, we take two different nodes $i_1, j_1 \in G_1^O$. Then by
 Theorem~\ref{prop.dat} (ii)(a), any max-weighted path to $i_1$ and $j_1$ from a common hidden ancestor would have to pass through a common source node variable $X_k$ for $k\in V_0^O\cap\pa^O(i_1)\cap \pa^O(j_1)$. Hence, in representations~\eqref{G1} for $X_{i_1}$ and $X_{j_1}$, the innovations would be the scaled innovations with different indices in $\An^{O^c}(i_1)$ and $\An^{O^c}(j_1)$, respectively; thus these innovations are independent.

To compute the entries of $A_O^{*}$ from those of $A_O$, we use the right-hand side of \eqref{proofMLrep} for $X_k$ in \eqref{G1}, which gives 
  \begin{align}\label{eq:i1}
  X_{i_1} = \bigvee_{k\in \an^{O}(i_1)}\frac{a_{i_1,k}}{a_{kk}} a_{kk}^{*} Z_k^{*} \vee  a_{i_1,i_1}^{*}Z_{i_1}^{*} 
  \, = \, \bigvee_{k\in \an(i_1) \cap O }a_{i_1,k}^{*} Z_k^{*} \vee  a_{i_1,i_1}^{*}Z_{i_1}^{*},
  \end{align}
  where $a_{i_1,k}^{*}=({a_{i_1,k}}/{a_{kk}})a_{kk}^{*}$, and $\Des(k)\cap\an^O(i_1)=\{k\}$. 
  This gives the matrix elements of the row for generation $G_{i_1}$ in $A_O^{*}$ for all $k\ge i_1$; we set $a^{*}_{i_1k}=0$ for all $k<i_1$.

  We have now established that for the first two generations, $V_0^O$ and $G_1^O$, the coefficient matrix can be represented by an upper-triangular matrix $A^{*}_O$, and that the innovations $Z_i^{*}$ are independent.  We now suppose that this is true for generations up to $G^O_{p-1}$, and argue by induction, assuming that we have obtained the coefficients $a_{i,k}^{*}$, where $i\in \cup_{j<p}G_j^O$. 
     For the inductive step,
      we select  $i_p$ from generation $G_p^O$, and note that~\eqref{minrep} implies that
		$$X_{i_p}=\bigvee_{k\in \an^{O}(i_p)}\frac{a_{i_p,k}}{a_{kk}} X_k \vee  \bigvee_{j\in \An^{O^c}(i_p)} a_{i_p,j}Z_{j},
            $$
            where, by Lemma~\ref{mlcomb}, the innovations can be encapsulated into a single standardised innovation $Z^{*}_{i_p}= \vee_{j\in \An^{O^c}(i_p)} Z_j/a_{i_p,i_p}^{*}$, with $a_{i_p,i_p}^{*2}=\sum_{k\in \An^{O^c}(i_p)}a_{i_p,k}^2$.  To complete the proof we must establish independence of the innovations and compute the entries of the matrix $A_O^{*}$.  
            
We first prove that the innovations in $\An^{O^c}(i_p)$ are independent of those in $\An^{O^c}(i_q)$, or, equivalently, that $Z^{*}_{i_p}$ is independent of $Z^{*}_{i_q}$ for an arbitrary node $i_q\in \cup_{j\leq p}G_j^O$ that belongs to some generation up to or including that of $i_p$. Without loss of generality let $i_q>i_p$. If $i_q$ and $i_p$ have a common hidden ancestor $u$, then by Theorem~\ref{prop.dat}(ii)(a) and (b), either there must be a max-weighted path $u\rightsquigarrow i_p \rightsquigarrow i_q$, or there must exist $k\in \an(i_p)\cap \an(i_q) \cap O$ such that there are paths $u\rightsquigarrow k\rightsquigarrow i_q$, and $u\rightsquigarrow k \rightsquigarrow i_p$. 
Therefore, the innovations indexed in $\An^{O^c}(i_p)$ and $\An^{O^c}(i_q)$ must have different indices; thus they are independent.

Finally we use the induction hypothesis for the ML representation of $X_k$ for $k\in \an^O(i_p)$, to obtain, similarly as in~\eqref{eq:i1},
\begin{align*}
X_{i_p}&=\bigvee_{k\in \an^{O}(i_p)}\frac{a_{i_p,k}}{a_{kk}}\bigvee_{j\in \An(k)\cap O}a_{kj}^{*} Z_j^{*} \vee  a_{i_p,i_p}^{*}Z_{i_p}^{*}\\
&=\bigvee_{k\in \an^{O}(i_p)}\bigvee_{j\in \An(k)\cap O}\frac{a_{i_p,k}}{a_{kk}} a_{kj}^{*} Z_j^{*} \vee  a_{i_p,i_p}^{*}Z_{i_p}^{*}\\
&=\bigvee_{j\in \an(i_p)\cap O}\bigvee_{k\in \an^{O}(i_p)\cap {\Des(j)}}\frac{a_{i_p,k}}{a_{kk}} a_{kj}^{*} Z_j^{*} \vee  a_{i_p,i_p}^{*}Z_{i_p}^{*},
\end{align*}
and this yields $a_{i_p,j}^{*}=\bigvee_{k\in \an^{O}(i_p)\cap {\Des(j)}}a_{kj}^{*}{a_{i_p,k}}/{a_{kk}} $. The exchange of the maximum operators in the last equality follows in a similar fashion to Lemma A.1 in \citet{gk}. 
The last expression gives the matrix elements of the rows corresponding to nodes $i_p \in G_{p}$ in $A^{*}_O$ for all $j\ge i_p$, and we set  {$a^{*}_{i_p,j}=0$} for all $j<i_p$.  This ends the second step of the proof, and establishes the result.  
	\end{proof}

Standardization of the innovation coefficient matrix $A$ allows us to use Lemma~2.1 of \citet{GKO}, re-stated below, and important in the subsequent proofs.

\ble\label{ineq}
If the RMLM $\boldsymbol{X}$ is supported on a well-ordered DAG, then $\bar{a}_{jj}>\bar{a}_{ij}$ for $i\in V$ and $j\in \an(i)$.
\ele

\begin{proof}[\textbf{Proof of Lemma~\ref{lemmaid1}}.] \, (i) Equivalence between the first equality and {the existence of a max-weighted path $k\rightsquigarrow j \rightsquigarrow i$ is a direct consequence} of Theorem~3.10 of \citet{gk}. Next, we expand
\begin{align*}
   \tilde{a}_{jj}\tilde{a}_{ik} 
   &= (a_{jj}+c_2a_{ij})(a_{jk}-c_1a_{ik})=a_{jj}a_{jk}-c_1a_{jj}a_{ik}+c_2a_{ij}a_{jk}-c_1c_2a_{ij}a_{ik}\\
   &=a_{jj}a_{jk}-c_1a_{jk}a_{ij}+c_2a_{jj}a_{ik}-c_1c_2a_{ij}a_{ik}
   = (a_{jk}+c_2a_{ik})(a_{jj}-c_1a_{ij})\\
   &= \tilde{a}_{jk}\tilde{a}_{ij},
\end{align*}
where the step from the first to the second line holds if and only if $a_{jj}a_{ik}=a_{jk}a_{ij}$.

\quad (ii) The proof follows immediately from (i) and the fact that for $k\notin\An(j)$ both vectors have {different} zero entries.  
In particular, the proof of (i) implies that, {for a pair $(i,j)$ in MWP}, $\tilde{a}_{ik}=\tilde{a}_{ij}\tilde{a}_{jk}/\tilde{a}_{jj}$ for all $k\in \An(j)$. 
Therefore, since $\tilde{a}_{ik}=0$, for $k\notin\An(j)$, the vector $(\tilde{a}_{i1},\ldots, \tilde{a}_{iD})$ is a scalar multiple of $(\tilde{a}_{j1},\ldots, \tilde{a}_{jD})$. 
By equivalence, the same holds for $k\in \An(j)$, when considering the vectors $A_i$, $A_j$. However, for such vectors we also have $a_{ii}>a_{ji}=0$ because $i\notin\An(j)$, implying they cannot be linearly dependent.

\quad (iii) {That $i\notin \an(j)$, follows from Theorem~2 of \citet{KK}.} For the second statement, suppose $\An(i)\cap\An(j)\neq \emptyset$, and that there exists $k\in\An(j)$ such that $a_{jk}<a_{ik}$. Then $\sigma_{M_{i,aj}}^2-\sigma_{M_{ij}}^2<a^2-1$, since $(a^2a_{jk}^2)\vee a_{ik}^2-a_{ik}^2< \max\big((a^2-1)a_{jk}^2,0\big)$, by arguments similar to equations~(23)--(24) of \citet{KK}, giving a contradiction.

If $\An(i)\cap\An(j)=\emptyset$, then $X_i$ and $X_j$ are independent, and by symmetry of the extremal dependence measure of $(X_i,X_j)$, $\sigma_{M_{i,aj}}^2=\sigma_{M_{ai,j}}^2=1+a^2$. Furthermore, $\sigma_{M_{ij}}^2=2$, and, hence, $\sigma_{M_{i,aj}}^2-\sigma_{M_{ij}}^2=\sigma_{M_{ai,j}}^2-\sigma_{M_{ij}}^2=(1-a^2)$, contradicting (iii).

To show the final statement by contradiction, suppose that {there exist max-weighted paths} $k\rightsquigarrow j \rightsquigarrow i$ for all $k\in\An(j)$. By equation~\eqref{eq:equiv}, $a_{jk}>a_{ik}$ for all $k\in\An(j)$, which yields $\sigma_{M_{i,aj}}^2=\sigma_{M_{ij}}^2 +a^2-1$. For the other difference, similar to equation~(23)--(24) of \citet{KK}, we obtain
\begin{align*}
\sigma_{M_{ai,j}}^2-\sigma_{M_{ij}}^2=\sum_{k_1\notin\An(j)}(a^2-1)\, a_{ik_1}^2+\sum_{k_2\in\An(j)} \big((a^2\, a_{ik_2}^2)\vee a_{jk_2}^2-a_{jk_2}^2\big) \, < \, a^2-1.
\end{align*}
These correspond to the two identities in the first statement in (iii), giving a contradiction.
\end{proof}

\begin{proof}[\textbf{Proof of Theorem~\ref{incdagmod}.}]
	By Lemma~\ref{lemmaid1} (iii), equation~\eqref{cond1} implies that $a_{jk}\geq a_{ik}$ for $k\in\An(j)$, with strict inequality for $k=j$ by Lemma~\ref{ineq}. As $0<c_1\leq 1$, also $a_{jk}\geq c_1 a_{ik}$ for such $k$. 
 Furthermore, equation~\eqref{cond1} and Theorem~2 of \citet{KK} imply that $i\notin \an(j)$, and therefore $i<j$. 
 By Lemma~\ref{lemmaid2}, the angular measure of $\boldsymbol{{T}}^{ij}$ is given by 
 \begin{align*}
	H_{\boldsymbol{{T}}^{ij}}(\cdot) = \sum_{k\in\An(j)}\norm{\tilde{a}_k}^2 \delta_{\big\{ \frac{{\tilde{a}_k}}{\norm{\tilde{a}_k}} \big\}}(\cdot),
	\end{align*}
for (non-zero) vectors $\tilde a_k$ given in Table~\ref{table:atoms}.
 Thus,  $\tilde{a}_k=(a_{jk} -c_1a_{ik}, a_{jk}+c_2a_{ik})$ for $k\in\An(j)$, and  $\tilde{a}_k=(0, 0)$ for $k\notin\An(j)$. 

We first prove (i). {Let $(i,j)\in$ MWP.}
Lemma~\ref{lemmaid1} (i) implies that $a_{ik}a_{jj}=a_{ij}a_{jk}$ for all $k\in\An(i)\cap\An(j)$ if and only if {there exists a max-weighted path from every}
$k\in\An(i)\cap\An(j)=\An(j)$ to $i$ {that} passes through $j$. 
  In this case, it also holds that $a_{ik}=ba_{jk}$ with $b={a_{ij}}/{a_{jj}}$ such that $\tilde{a}_k=((1-c_1b)a_{jk}, (1+c_2b)a_{jk})$ for $k\in\An(j)$.
	
	The components of $\boldsymbol{X}$ are standardised, that is, $\sum_{k\in\An(r)}a_{rk}^2=1$ for $r\in\{i,j \}$, and the squared scalings of ${T}^{ij}_{1}$ and ${T}^{ij}_{2}$ become $\sigma_{{T}^{ij}_{1}}^2=(1-c_1b)^2$ and $\sigma_{{T}^{ij}_{2}}^2=(1+c_2b)^2$. 
 Standardization of ${{T}^{ij}_{1}}$ and ${{T}^{ij}_{2}}$ to unit scalings amounts to normalising them via~\eqref{eq:Tscaled} by the respective factors $1/(1-c_1b)$ and $1/(1+c_2b)$,  which map the vectors $\tilde{a}_k$ to $\bar{a}_k=(a_{jk}, a_{jk})$ for $k\in\An(j)$.
 Hence, by Corollary~\ref{completedep} (iii),
 $\tau_{ij}^2={\sigma}_{ \tilde{T}^{ij}_{1},\tilde{T}^{ij}_{2}}^2=1$.
	
For the reverse implication, recall first that $\sigma_{\tilde{T}^{ij}_{1}}^2=\sigma_{\tilde{T}^{ij}_{2}}^2=1$ for the standardised vector $\bs{\tilde{T}}^{ij}$.
 Assume that there exist scalars $\sigma_{{T}^{ij}_{1}},\sigma_{{T}^{ij}_{2}}$ such that, after standardization,  $\tau_{ij}^2=1$. 
Recall that 
 $\tilde{a}_k=(a_{jk} -c_1a_{ik}, a_{jk}+c_2a_{ik})$ for $k\in\An(j)$ and $\tilde a_k=(0,0)$ for $k\notin\An(j)$ and 
Proposition~\ref{completedep} implies
	\begin{align}\label{CSineq}
	1= \tau_{ij}^2={\sigma}_{ \tilde{T}^{ij}_{1},\tilde{T}^{ij}_{2}}^2
 & =\frac{1}{\sigma_{{{T}^{ij}_{1}}} \sigma_{{{T}^{ij}_{2}}}}\sum_{k\in\An(j)}(a_{jk}-c_1a_{ik})(a_{jk}+c_2a_{ik})\nonumber\\
 & =\Big[\frac{1}{\sigma_{{T}^{ij}_{1}}^2}\sum_{k\in\An(j)}(a_{jk}-c_1a_{ik})^2\frac{1}{ \sigma_{{T}^{ij}_{2}}^2}\sum_{k'\in\An(j)}(a_{mk'}+c_2a_{ik'})^2\Big]^{1/2}\\
 &= \sigma_{\tilde{T}^{ij}_{1}}\,\sigma_{\tilde{T}^{ij}_{2}}.\nonumber
	\end{align}
	However, by the Cauchy--Schwarz inequality, equation~\eqref{CSineq} holds if and only if for some $b>0$ we have $(a_{jk}-c_1a_{ik})=b(a_{jk}+c_2a_{ik})$ for $k\in\An(j)$. Furthermore, since 
 $a_{jj}>0$, it must be the case that $a_{ij}>0$ also, and therefore 
 $(1-b)a_{jk}=(c_1+bc_2)a_{ik}$ for $k\in\An(j)$, and so ${a_{ik}}{a_{jj}}={a_{ij}}{a_{jk}}$ for such $k$. This implies that for all $k\in\An(i)\cap\An(j)$ there are max-weighted paths from $k$ to $i$ that pass through $j$. This proves (i). 

To establish (ii), suppose for a contradiction that~\eqref{CSineq} holds for the standardised variables $\tilde{T}^{ij}_{1},\tilde{T}^{ij}_{2}$, and that $a_{ij}=0$. Similar to the argument in the previous paragraph, by the Cauchy--Schwarz inequality,~\eqref{CSineq} implies that ${a_{ik}}{a_{jj}}={a_{ij}}{a_{jk}}$ for {all} $k\in\An(j)$. By~\eqref{cond1} and Lemma~\ref{lemmaid1} (iii), there exists some $k\in\An(i)\cap\An(j)\neq \emptyset$ such that $a_{ik},a_{jk}>0$, and moreover $a_{jj}>0$, which contradicts the equality ${a_{ik}}{a_{jj}}={a_{ij}}{a_{jk}}$, and implies that the path $k\rightsquigarrow j \rightsquigarrow i$ is not max-weighted. Therefore, we must have $\tau_{ij}^2<1$.
\end{proof}

\begin{proof}[\textbf{Proof of Theorem~\ref{inc2dagmod}.}]
		{The proof of (i)} closely follows that of Theorem~\ref{incdagmod}. From~\eqref{cond2} it holds that $a_{jk}\geq a_{ik}$ for $k\in \An(j)\setminus\An(K)$ with strict inequality $a_{jj}>a_{ij}$, and likewise for $\boldsymbol{T}^{K ij}$ when $c\leq1$.  Hence  $\boldsymbol{T}^{K ij}$ has angular measure 
	\begin{align*}
	H_{{\boldsymbol{T}^{K ij}}}(\cdot) = \sum_{k\in\An(j)\setminus\An(K)}\norm{\tilde{a}_k}^2 \delta_{\big\{ \frac{{\tilde{a}_k}}{\norm{\tilde{a}_k}} \big\}}(\cdot),
	\end{align*}
	where $\tilde{a}_k=(a_{jk} -c_1a_{ik}, a_{jk}+c_2a_{ik}),$ for $k\in \An(j)\setminus\An(K)$, and $\tilde{a}_k=(0, 0)$ otherwise, in particular for $k\in \An(i)\setminus\An(j)$, and, due to Lemmas~\ref{genmaxspect5CK} and~\ref{genmaxspectlast}, also for $k\in K$. Since for $k\in  (\An(i)\cap\An(j))\setminus\An(K)\neq \emptyset$ there is a max-weighted path from $k$ to $i$ via $j$ if and only if ${a_{ik}}{a_{jj}}={a_{ij}}{a_{jk}}$, then $a_{ik}=ba_{jk}$ for all such $k$, giving $\tilde{a}_k=((1 -c_1b)a_{jk},(1+c_2b)a_{jk})$.
	
	The squared scaling of ${T}^{Kij}_{3,2}$ equals $\sigma_{{T}^{Kij}_{3,2}}^2=\sum_{k\in\An(j)\setminus \An(K)}a_{jk}^2$, which implies that the squared scalings of ${T}^{Kij}_{1}$ and ${T}^{Kij}_{2}$ become $(1-c_1b)^2\sigma_{{T}^{Kij}_{3,2}}^2$ and $(1+c_2b)^2\sigma_{{T}^{Kij}_{3,2}}^2$. Upon standardising the components $\boldsymbol{T}^{K ij}$ to unit scalings, say, into the vector $\boldsymbol{\tilde{T}}^{K ij}$, by arguments similar to those of the proof of Theorem~\ref{incdagmod}, it is clear that $\tau_{Kij}^2=\sigma_{{\tilde T}^{Kij}_{1},{\tilde T}^{Kij}_{2}}^2=1$.
	
	The reasoning in the other direction mimics the Cauchy--Schwarz inequality argument in the proof of Theorem~\ref{incdagmod}, but with  the index of the summation ranging in $\An(j)\setminus\An(K)$. 
 Because $a_{jj}>0$ and $\An(i)\cap\An(j)\setminus\An(K)\neq \emptyset$, we must have $a_{ij}>0$. The Cauchy--Schwarz equality {then implies} that ${a_{ik}}{a_{jj}}={a_{ij}}{a_{jk}}$ for $k\in(\An(i)\cap\An(j))\setminus\An(K)$.
	
	The proof of (ii) is identical to that of Theorem~\ref{incdagmod}; the only change is again in the index $k\in(\An(i)\cap\An(j))\setminus\An(K)\neq \emptyset$. 
\end{proof}

\section{Multivariate regular variation}\label{sec:ARV}
 
\subsection{Definitions and results for regularly varying vectors}\label{sec:3.1}

 We state two equivalent definitions of multivariate regular variation taken from \citet[Theorem~6.1]{ResnickHeavy}. 
 
 \bde \label{mrv} 
 (i) \, A random vector $\boldsymbol{Y}\in\mathbb{R}^{d}_+$ is {multivariate regularly varying} if there exists a sequence of real numbers $b_n\to \infty$  as $n\to\infty$ such that
 \beam\label{eq:mrva}
 n\mathbb{P}(\boldsymbol{Y}/{b_n}\in \cdot)\overset{v}{\to}{\nu_{\bs Y}}(\cdot), \hspace{5mm}n\to\infty,
 \eeam
 where $\overset{v}{\to}$ denotes vague convergence in $M_+([0, \infty]^{ {d}}\setminus\{\boldsymbol{0}\})$, the set of non-negative Radon measures on $[0, \infty]^{ {d}}\setminus\{\boldsymbol{0}\}$, and $\nu_{\bsy}$ is called the exponent measure of $\bsy$.\\
 (ii) \,
 A random vector  $\boldsymbol{Y}\in\mathbb{R}^{ {d}}_+$ is {multivariate regularly varying} if for any norm $\|\cdot\|$ there exists a finite measure $H_{\boldsymbol{Y}}$ on the positive unit sphere $\Theta_+^{{ {d}}-1}=\{\boldsymbol{\omega}\in \mathbb{R}^{ {d}}_+: \norm{\boldsymbol{\omega}}=1\}$ and a sequence $b_n\to \infty$ as $n\to\infty$ such the {angular representation} $(R,\boldsymbol{\omega})\coloneqq(\norm{\boldsymbol{Y}}, \boldsymbol{Y}/\norm{\boldsymbol{Y}})$ of $\bsy$ satisfies 
 \begin{align}\label{eq:mrvb}
 n\mathbb{P}\big(({R}/{b_n},\boldsymbol{\omega})\in \cdot\big)\overset{v}{\to} \nu_\alpha\times H_{\boldsymbol{Y}}(\cdot), \hspace{5mm}n\to\infty,
 \end{align}
 in $M_+((0,\infty]\times\Theta_+^{{ {d}}-1})$, ${\rm d}\nu_\alpha(x)=\alpha x^{-\alpha-1}{\rm d}x$ for some $\alpha>0$, and for Borel subsets $C\subseteq \Theta_+^{{ {d}}-1}$,
 \begin{align*}
 H_{\boldsymbol{Y}}(C)\coloneqq\nu_{\boldsymbol{Y}}\big(\{\boldsymbol{y}\in\mathbb{R}^{ {d}}_+\setminus\{\boldsymbol{0}\}: \norm{\boldsymbol{y}}\geq 1, \boldsymbol{y}/\norm{\boldsymbol{y}} \in {C}\}\big).
 \end{align*}
 In this case the measure $H_{\boldsymbol{Y}}$ is called the {angular measure of $\boldsymbol{Y}$}, we write $\boldsymbol{Y}\in {\rm RV}^{ {d}}_+(\alpha)$ and call $\alpha$  the {index of regular variation}.
 \ede

For a vector $\bsy$ with standardised margins \citet[Corollary~3.2.18 (1.)]{MikWin} give a  useful characterisation of multivariate regular variation, stated in the lemma below, by studying the limiting conditional distribution of the angular components.
Similar characterisation holds under  even more general settings in the context of star-shaped metric spaces via the so-called modulus, a norm-like function \citep[equivalence between (i) and (ii) in Prop. 3.1]{segersmodulus}.


\begin{lemma}\label{def:mrvc}
Let $\boldsymbol{Y}\in \mathbb{R}^d_+$ have standardised margins such that $\mathbb{P}(Y_i>y)\sim y^{-\alpha}$  as $y\to\infty$ for $i\in\{1,\ldots, d\}$ and $\alpha>0$. Assume further that $R=\norm{\boldsymbol{Y}}\in {\rm RV}_+(\alpha)$, 
and that for $\boldsymbol{\omega}=\boldsymbol{Y}/\norm{\boldsymbol{Y}}\in {\Theta}^{d-1}_+$ there exist a sequence $b_n'\to\infty$ such that 
\beam\label{eq:mrvc}
\mathbb{P}\left(\boldsymbol{\omega} \in \cdot\, |\, R/b_{n}'>1\right)\overset{w}{\to} \tilde{H}_{\boldsymbol{Y}}(\cdot), \quad n\to\infty,
\eeam
where $\overset{w}{\to}$ denotes weak convergence 
and $\tilde{H}_{\boldsymbol{Y}}$ is a probability measure on ${\Theta}^{d-1}_+$. Then $\boldsymbol{Y}\in {\rm RV}^{ {d}}_+(\alpha)$. 
\end{lemma}

The choice of the normalising sequence $b_n'$ can vary with the choice of the norm or the dependence structure of the vector $\bs Y$ as in Lemma~\ref{def:mrvc}. 
Without loss of generality we focus on  $\alpha\geq1$ and  $\norm{\cdot}_\alpha$ as choice of norm. 
{Multivariate regularly varying vectors with index $\alpha<1$ can always be transformed to regular variation with index greater than one, for instance, by taking all components to a certain power.}

From Definition~\ref{mrv}(ii) we obtain that
\begin{align}\label{interstep_mass}
{n}\mathbb{P}\left(R/{b_{n}}> r\right)
={n}\mathbb{P}\big(R/{b_{n}}> r, \boldsymbol{\omega} \in \Theta^{d-1}_+\big) \to r ^{-\alpha} \times H_{\boldsymbol{Y}}(\Theta^{d-1}_+),\quad n\to\infty,
\end{align}
indicating that we may encode information about $H_{\boldsymbol{Y}}(\Theta^{d-1}_+)$ in $b_{n}$ to arrive at~\eqref{eq:mrvc}.
In a similar fashion to \citet[p. 592]{cooley}, we use~\eqref{interstep_mass} to compute
\begin{align}\label{eq:ncd}
H_{\boldsymbol{Y}}(\Theta_+^{{ {d}}-1})&=\int_{\Theta_+^{{ {d}}-1}} \norm{\boldsymbol{\omega}}^{\alpha} {\rm d}H_{{\boldsymbol{Y}}}(\boldsymbol{\omega})\nonumber\\
&=\int_{\Theta_+^{{ {d}}-1}}\sum_{j=1}^d \omega_j^{\alpha}{\rm d}H_{{\boldsymbol{Y}}}(\boldsymbol{\omega})=\sum_{j=1}^d\int_{\Theta_+^{{ {d}}-1}} \omega_j^{\alpha}{\rm d}H_{{\boldsymbol{Y}}}(\boldsymbol{\omega})\nonumber\\
&=\sum_{j=1}^d\int_{\Theta_+^{{ {d}}-1}} \int_{r>\omega_j^{-1}}\alpha r^{-(\alpha+1)}{\rm d}r{\rm d}H_{{\boldsymbol{Y}}}(\boldsymbol{\omega})\nonumber\\
&= \sum_{j=1}^d \lim_{n\to \infty}n\mathbb{P}\big(R/{b_n}>\omega_j^{-1}, \boldsymbol{\omega}\in \Theta_+^{{ {d}}-1}\big)\nonumber\\
&=\sum_{j=1}^d \lim_{n\to \infty}n\mathbb{P}\big(Y_j/{b_n}>1\big)=d,
\end{align}
where we have used the fact that $R\omega_j=Y_j$ and
$b_n\sim n^{1/\alpha}$ are correct normalising constants for the standardised $Y_j$.
Therefore, to normalise the mass of $H_{{\boldsymbol{Y}}}$, without loss of generality, we can fix constants $b'_n$ for $R$ such that $b'_n\sim(dn)^{1/\alpha}$ as $n\to\infty$.

{The following lemma links the scalings of $M_K=\max(X_k: k\in K)$ over components of $\bs X$ 
for $\alpha=2$ with the tail asymptotics of $M_K$.}

\begin{lemma}\label{tailMk}
	Consider a vector $\bs Y\in {\rm RV}^d_+(2)$ with standardised margins, let $K\subseteq \{1,
    \dots,d\}$ and $M_K:=\max(Y_k: k\in K)$. Then the squared scaling $\sigma_{M_K}^2$ of  $M_K$ satisfies $n\mathbb{P}({{M_K}}/\sqrt{n}>y)=\sigma_{M_K}^2y^{-2}$, $y>0$.
\end{lemma}
\begin{proof} Since $\bs Y\in {\rm RV}^d_+(2)$, we have {by Definition~\ref{mrv}(i)} that for $y>0$ 
		\begin{align*}
		\lim\limits_{n\to \infty} n\mathbb{P}({{M_{K}}}/\sqrt{n}>y)&=\nu_{\bs{Y}}\Big({\Big\{\bs{y}\in\R_+^{ {d}} : \frac{\bs{y}}{\norm {\boldsymbol{y}}}\in \Theta_+^{{d}-1}, \underset{k\in K}{\vee}y_k > y\Big\}}\Big)\\
		&=\int_{\Theta_+^{ {d}-1}} \int_{\{r> y/{\underset{k\in K}{\vee}\omega_i}\}} 2 r^{-3}{\rm d}r  {\rm d}H_{{\bs{Y}}}(\bs{\omega})\\
		&= y^{-2} \int_{ \Theta_+^{ {d}-1}} \underset{k\in K}{\vee}\omega_k^{2}  {\rm d}H_{{\bs{Y}}}(\bs{\omega})\\
		&=\sigma_{M_{K}}^2y^{-2},
		\end{align*}
		which proves the claim of the lemma.
\end{proof}

{We end this section with a remark showing that for a max-linear vector $\bs X\in {\rm RV}^d_+(2)$ with coefficient matrix $A$ and the Euclidean norm, the scaling parameters and the extremal dependence measures are invariant with respect to the dimensionality of the angular measure. 
This follows upon representing a (sub)vector $\bs X_K$ of $\bs X$ by the matrix $A_K$, which contains only the rows of $A$ indexed in $K$, and applying Definitions~\ref{scaledef} and~\ref{scalMK} using the angular measure $H_{\bs X_K}$.

\begin{remark}\label{invariant}
Consider the max-linear vector $\bs X\in {\rm RV}^d_+(2)$ with coefficient matrix $A$ and its subvector $\bs X_K$ for a subset of indices $K$. Then the extremal dependence measures $\sigma_{ij}^2$ and the scalings  $\sigma_{M_{J}}^2$, where $i,j\in K$ and $J\subseteq K$, computed from the angular measures of $\bs X$ and $\bs X_K$ under the Euclidean norm $\norm{\cdot}$ are the same.
\end{remark}
}

 \subsection{Regularly varying RMLMs}\label{sec:3.2}
 
\begin{proposition} (\citet[Lemma 3]{foug}, \citet[Section 6]{einmahl2012m}], 
)\label{discspect}\\
Let $\bsz\in {\rm RV}^D_+(\alpha)$ with independent components $Z_k\in{\rm RV}_+(\alpha)$, $A\in\R_+^{d\times D}$ and 
\begin{align}\label{eq:AZ}
\bsx= A\times_{\max} \bsz.
\end{align}
Then $\bsx\in{\rm RV}^d_+(\alpha)$ with discrete angular measure 
 \begin{equation}
 H_{\boldsymbol{X}}(\cdot) = \sum_{k=1}^D\norm{a_k}^\alpha \delta_{\big\{ \frac{{a_k}}{\norm{a_k}} \big\}}(\cdot),
 \label{discretspecteq}
 \end{equation}	
 on the positive unit sphere $\Theta_+^{d-1}$ with atoms $(a_{k}/\norm{a_k})_{k=1,\dots,D}$ for $a_k=(a_{1k},\dots,a_{dk})$; i.e., the atoms are the {normalised} columns of $A$.  
 The finite measure $H_{\bsx}$ can be normalised to a probability measure $\tilde{H}_{\bsx}$  defined as
 \begin{align}\label{specmass2}
 \tilde{H}_{\bsx}(\cdot)\coloneqq \frac{H_{\boldsymbol{X}}(\cdot)}{H_{\boldsymbol{X}}(\Theta_+^{{d}-1})}.
 \end{align}
\end{proposition}

\begin{corollary}\label{fulldep}
(i) {\rm (\citet[ Proposition~A.2]{GKO})} For a set $[\boldsymbol{0},\boldsymbol{x}]^c$ the exponent measure is given by 
\begin{align*}
\nu_{\bsx}([\boldsymbol{0},\boldsymbol{x}]^c)=\sum_{k=1}^{D}\underset{i=1}{\overset{d}{\bigvee}}\frac{a_{ik}^\alpha}{x_i^\alpha}.
\end{align*}
(ii) Let $a_{ik}=a_{jk}$ for all $k\in\{1,\dots,D\}$ as in Corollary~\ref{completedep} (iii). Then 
\begin{align*}
\nu_{\bsx}([\boldsymbol{0},\boldsymbol{x}]^c)=\sum_{k=1}^{D} a_{1k}^\alpha \underset{i=1}{\overset{d}{\bigvee}} x_i^{-\alpha} = \sum_{k=1}^{D} a_{1k}^\alpha \,\Big(\underset{i=1}{\overset{d}{\bigwedge}} x_i\Big)^{-\alpha},
\end{align*}
giving asymptotically full dependence by \citet[equation~(6.32)]{ResnickHeavy}.
\end{corollary}

In what follows we denote the set of random vectors $\bsx$ as in~\eqref{eq:AZ} by ${\rm RV}^d_+(\alpha,A)$ and as before the rows of $A$ by $A_i$ for $i\in\{1,\dots,d\}$.
Let now $\bsx\in{\rm RV}^d_+(\alpha,A)$, then each component of $\bsx$ has representation 
$$X_i= A_i \times_{\max} \bsz =\bigvee_{k\in\{1,\dots,D\}} a_{ik} Z_k,\quad i\in\{1,\dots,d\}.$$

This motivates the following. 

\begin{proposition}\label{prop:easy}
Consider the set of random variables
$$\{X=a \times_{\max}\bsz\in {\rm RV}_+(\alpha,a): a\in\R^D_+\}.$$
This set has the following properties:
\begin{enumerate}
\item[(i)]
For $c>0$ and $X\in{\rm RV}_+(\alpha, a)$ we have $cX\in {\rm RV}_+(\alpha, ca)$.
\item[(ii)]
Let $X_1\in{\rm RV}_+(\alpha, a_1)$ and $X_2 \in {\rm RV}_+(\alpha,a_2)$, then 
$(X_{1},X_{2}) \in {\rm RV}^2_+(\alpha,A_{12})$ with $A_{12}=(a_{1},a_{2})^\top$; i.e., $(X_{1},X_{2})$ has discrete angular measure with  atoms obtained by normalising the non-zero columns of $A_{12}$. 
\item[(iii)]
Let $X_1\in {\rm RV}_+(\alpha,a_1)$ and $X_2 \in {\rm RV}_+(\alpha,a_2)$, then
$\max\{X_1,X_2\} \in{\rm RV}_+(\alpha,a_1\vee a_2)$, where $a_1\vee a_2$ is taken componentwise.
\end{enumerate}
\end{proposition}

\begin{proof}
(i) follows directly from the representation of $X$.\\ 
(ii) is a simple consequence of considering the vector $(X_1, X_2)$, therefore arranging the transposed vectors $a_1,a_2$ into rows of a new matrix $A_{12}\in\mathbb{R}_+^{2\times D}$. The atoms are then derived from Proposition~\ref{discspect}.\\
(iii) is a consequence of max-linearity:
$$\hspace{3cm}X_1\vee X_2= \underset{j\in\{1,2\}}\vee\underset{k\in\{1,\ldots,D\}}\vee a_{jk}Z_k=\underset{k\in\{1,\ldots,D\}}\vee (\underset{j\in\{1,2\}}\vee a_{jk})Z_k.\hspace{1,7cm}\qedhere
$$
\end{proof}

From this we can immediately read off the first two lines of Table~\ref{table:atomsO} and of Table~\ref{table:atoms}.
For the third lines of these Tables, we apply the following multivariate version of Breiman's lemma in combination with Proposition~\ref{prop:easy}.

\begin{lemma} {\rm [\citet[Proposition~A.1]{basrak2002}]}\label{breimanlemma}
	Let $\boldsymbol{Y}\in {\rm RV}^d(\alpha)$, and $S$ be a random $q\times d$ matrix, independent of $\boldsymbol{Y}$. If $0<\mathbb{E}\norm{S}^\gamma<\infty$ for some $\gamma>\alpha$ then
	$$n\mathbb{P}(b_n^{-1}S\boldsymbol{Y}\in \cdot)\overset{v}{\to}\tilde{\nu}(\cdot)\coloneqq \mathbb{E}[\nu \circ S^{-1}(\cdot)],$$
 where $\overset{v}{\to}$ denotes vague convergence on $\R^d \setminus\{0\}$.
\end{lemma}

\begin{corollary}\label{cor:easy}
Let $\bsx=A\times_{\max} \bsz$ be as in Proposition~\ref{prop:easy}, with atoms of the angular measure of $\bsx$ obtained by normalising the non-zero column vectors of the $d\times D$-matrix $A$ on the positive unit sphere. Let $S\in\R^{q\times d}$ be a non-random matrix. Then the linear transformation $S \bsx= S (A \times_{\max} \bsz)\in {\rm RV}^d(\alpha)$ 
has discrete angular measure on the unit sphere $\Theta^{d-1}=\{\omega\in\R^d: \|\omega\|=1\}$ with atoms obtained by normalising the non-zero columns of $SA$.  
\end{corollary}

\begin{lemma}\label{genmaxspectmainCK}
	Let $\boldsymbol{X}\in {\rm RV}_+^d(\alpha,A)$ with $A\in\mathbb{R}_+^{d\times D}$. {Recall $\boldsymbol{T}_2^{ij}=(M_{c_1ij}, M_{ij},X_i,X_j)$ from Table~\ref{table:atoms} and define the matrix}
$$S=\begin{bmatrix}
	1&0&-c_1&0\\
	0&-c_2& c_2& 1+c_2\\
	\end{bmatrix}
 \in\mathbb{R}^{2\times4}.
	$$
 Then $\boldsymbol{T}^{ij}:=S \boldsymbol{T}_2^{ij} =(M_{c_1i,j}-c_1X_i, (1+c_2)X_j+c_2X_i-c_2M_{ij})$, $\boldsymbol{T}^{ij}\in  {\rm RV}^2(\alpha)$ and has discrete angular measure with  atoms obtained by normalising the vectors 
 $\tilde{a}_k=(c_1a_{ik}\vee a_{jk} -c_1a_{ik}, (1+c_2)a_{jk}+c_2a_{ik}-c_2(a_{ik}\vee a_{jk}))$.	
\end{lemma}

\begin{proof}
From Proposition~\ref{prop:easy} we read off the non-zero atoms $(\tilde{a}_k/\norm{\tilde{a}_k})$ of the angular measure of $\boldsymbol{T}_2^{ij}$ where 
$\tilde{a}_k=(c_1a_{ik}\vee a_{jk}, a_{ik}\vee a_{jk}, a_{ik}, a_{jk})$.
 Applying 
 Corollary~\ref{cor:easy}, the result follows.
\end{proof}

\begin{lemma}\label{genmaxspect5CK}
	Let $\boldsymbol{X}\in {\rm RV}_+^d(\alpha,A)$ be a $d-$dimensional subvector of an RMLM with coefficient matrix $A\in\mathbb{R}_+^{d\times D}$. 
Recall $\boldsymbol{T}_2^{K ij}\coloneqq(M_{K},M_{K,i},M_{K,j})$ from Table~\ref{table:atomsO} and define the matrix
 $$S=\begin{bmatrix}
	-1 & 1 & 0\\
	-1 & 0 & 1\\
	\end{bmatrix}
 \in\mathbb{R}^{2\times3}.
	$$
Then $\boldsymbol{T}^{Kij}_3\coloneqq(M_{K,i}-M_{K},M_{K,j}-M_{K})= S \boldsymbol{T}_2^{K ij}$, 
 $\boldsymbol{T}^{Kij}_3\in {\rm RV}^2(\alpha)$ and has discrete angular measure with atoms obtained by normalising the vectors 
	 $\tilde{a}_k=(\underset{r\in K\cup\{i\}}{\vee}a_{rk}-\underset{r\in K}{\vee}a_{rk} , \underset{r\in K\cup\{j\}}{\vee}a_{rk}-\underset{r\in K}{\vee}a_{rk} )=(a_{ik} \mathds{1}_{\{k\notin \An(K)\}}, a_{jk} \mathds{1}_{\{k\notin \An(K)\}}).$	
\end{lemma}

\begin{proof}
From Proposition~\ref{prop:easy} we read off the atoms of the angular measure of ${\boldsymbol{T}_2^{K ij}}$ by normalising the non-zero vectors
$(\underset{r\in K}{\vee}a_{rk}, \underset{r\in K\cup\{i\}}{\vee}a_{rk}, \underset{r\in K\cup\{j\}}{\vee}a_{rk})$. Applying Corollary~\ref{cor:easy}, the first representation of $\tilde{a}_k$ follows. The second representation is
	due to the causal ordering of the nodes in $K$ and Lemma~\ref{ineq}, since both $a_{ik},a_{jk}$ are strictly less than $a_{kk}$ for some $k\in K$. This results in
	$$
	\underset{r\in K\cup\{i\}}{\vee}a_{rk}=\begin{cases} \underset{r\in K}{\vee}a_{rk}, & k\in \An(K), \\
	a_{ik}, & k\notin \An(K), \end{cases} 
	$$
	and likewise for ${\vee}_{r\in K\cup\{j\}}a_{rk}$.
\end{proof}

\begin{lemma}\label{genmaxspectlast}
	Let $\boldsymbol{X}\in {\rm RV}_+^d(\alpha,A)$ satisfy the setting of Lemma~\ref{genmaxspect5CK}.
 Set $\boldsymbol{T}_4^{K ij}\coloneqq(M^K_{c_1i,j}, M^K_{i,j},$ ${T}_{3,1}^{K ij},{T}_{3,2}^{K ij})$, for components defined around~\eqref{TKij}, and define the matrix
 $$S=\begin{bmatrix}
	1 &0 & -c_1& 0\\
	0& -c_2 & c_2 & 1+c_2\\
	\end{bmatrix}
 \in\mathbb{R}^{2\times4}.
	$$
Recall $\boldsymbol{T}^{Kij}\coloneqq(M^K_{c_1i,j}-c_1{ T}^{Kij}_{31}, (1+c_2){T}^{Kij}_{32}+c_2{T}^{Kij}_{31}-c_2M^K_{ij})$ from~\eqref{TKij}.
Then $\boldsymbol{T}^{Kij} =S\boldsymbol{T}^{Kij}_4$, $\boldsymbol{T}^{Kij}\in {\rm RV}^2(\alpha)$ and has discrete angular measure with atoms obtained by normalising the vectors 
	 $\tilde{a}_k=(c_1a_{ik}\vee a_{jk} -c_1a_{ik}, (1+c_2)a_{jk}+c_2a_{ik}-c_2(a_{ik}\vee a_{jk}))\mathds{1}_{\{k\notin \An(K)\}}.$	
\end{lemma}

\begin{proof} The proof follows by applying Lemmas~\ref{genmaxspectmainCK} and~\ref{genmaxspect5CK} to the vector $\boldsymbol{T}^{K ij}_3$, instead of $(X_i, X_j)$.
\end{proof}

\section{Statistical theory for regularly varying innovations}\label{sec:stat}

We start with some notation and results from \citet{ResnickHeavy}, as in Section~6 of \citet{KK}. 

Let $\boldsymbol{Y}_1,\ldots,{\boldsymbol{Y}_n}$ for $n\in\N$ be independent replicates of $\boldsymbol{Y}\in {\rm RV}^d_+(\alpha)$ with standardised margins,   and consider the angular decomposition of $\bsy$ given by
\begin{align}\label{angular.eq}
R\coloneqq\norm{\boldsymbol{Y}} , \quad  \boldsymbol{\omega}=(\omega_1,\ldots,\omega_d)=\frac{\boldsymbol{Y}}{R}. 
\end{align}
We call $R$ the radial and $\boldsymbol{\omega}$ the angular component. 
Similarly for $\bsy_\ell$ we write $R_\ell=\|\bsy_\ell\|$ and $\boldsymbol{\omega_\ell}=\bsy_\ell/R_\ell$ for  $\ell\in\{1,\ldots, n\}$.

The standardised angular measure $\tilde{H}_{\bsy}$ from~\eqref{specmass2} {provides a way to obtain a} consistent estimator from the empirical angular measure \cite[see, e.g., (9.32) of][]{ResnickHeavy}, given for known normalising functions $b_{n/k}$ by 
\begin{align}\label{eq:estang}
\tilde{H}_{\bsy, n/k} (\cdot) = \frac{\sum_{\ell=1}^{n}\mathds{1}{\{(R_\ell/b_{\frac{n}{k}},{\boldsymbol{\omega}}_\ell) \in [1,\infty] \times \cdot\}}}{\sum_{\ell=1}^{n}\mathds{1}{\{R_\ell/b_{\frac{n}{k}}\geq1\}}}\stw \tilde{H}_{\bsy} (\cdot),
\end{align}
as $n\to\infty$, $k\to\infty$, $k/n\to 0$. 

Let $R^{(n)}\leq \cdots\leq R^{(1)}$ denote the order statistics of $R_1,\ldots,R_n$.  
If we choose normalising functions $b_t$ such that 
\beam\label{eq:normfct}
t\P(\bsy/b_t\in\cdot) & \stv &\nu_{\bsy}(\cdot),\quad t\to\infty,
\eeam
then 
from \citet[p.~308]{ResnickHeavy}, we know that 
$R^{(k)}/b_{ n/k}\overset{P}{\to}1$, which suggests setting $b_{ n/k}=R^{(k)}$ in~\eqref{eq:estang} and gives the estimator 
\begin{align}
\tilde{H}_{\bsy, n/k} (\cdot) = \frac{1}{k}\sum_{\ell=1}^{n}\mathds{1}\{R_\ell \geq R^{(k)}, \boldsymbol{\omega}_\ell\in \cdot\},
\end{align}
where $k=\sum_{\ell=1}^{n}\mathds{1}\{R_\ell \geq R^{(k)}\}$.  

Our goal is to estimate extremal dependence measures and squared scalings as in Definition~\ref{scaledef}, and we define for a continuous function $f:\Theta^{d}_+\to \R_+$ the quantity
$$\E_{\tilde{{H}}_{\bsy}}[f(\boldsymbol{\omega})] :=\lim_{x\to\infty} \E[f(\boldsymbol{\omega})\mid R>x] =\int_{\Theta^d_+} f(\boldsymbol{\omega})d\tilde{H}_{\bsy}(\boldsymbol{\omega}).
$$
Thus a natural estimator for $\mathbb{E}_{\tilde{{H}}_{\bsy}}[f(\boldsymbol{\omega})]$ is  equation~(29) of \citet{KK}, given by
\begin{align}\label{initialest}
\hat{\mathbb{E}}_{\tilde{{H}}_{\bsy}}[f(\boldsymbol{\omega})]=\frac{1}{k}\sum_{\ell=1}^{n}f(\boldsymbol{\omega}_\ell)\mathds{1}\{R_\ell \geq R^{(k)}\}.
\end{align}
The function $f(\cdot)$ will depend on whether we want to estimate extremal dependence measures or squared scalings.

\subsection{Intermediate thresholding}\label{sec:thresh}

\citet{KK} use the setting from the previous section to estimate squared scalings of partial maxima of selected components of an RMLM in their Section~6. 
In the present paper we want to estimate the extremal dependence measure of the components of $\boldsymbol{\tilde T}^{ij}=(\tilde T_1^{ij},\tilde T_2^{ij})$ as defined in Section~\ref{sec:MWP}. 
Transformation of the sample variables to $\boldsymbol{\tilde T}^{ij}$ creates many small values near $\boldsymbol{0}$, corrupting the estimator~\eqref{initialest} significantly.
As a remedy, we have implemented a two-step procedure using besides $k$ as in~\eqref{initialest} an additional intermediate threshold.
 
For a given large sample $\boldsymbol{Y}_1,\ldots,\boldsymbol{Y}_n$ in ${\rm RV}^d_+(\alpha)$, with standardised margins and with angular decomposition ~\eqref{angular.eq} under the norm $\norm{\cdot}_{\alpha}$, we choose a threshold $k_1$. Consider for $\ell\in\{1,\dots,n\}$ only those observations whose radial components satisfy $R_\ell\geq b_{n/{k_1}}$ for normalising constants $b_{n/{k_1}}$ as in~\eqref{eq:estang}, and define
\beam\label{eq:Nn}
N_n=\sum_{\ell=1}^{n}\mathds{1}\{R_\ell\geq b_{n/{k_1}}\}.
\eeam
Following the d\'ecoupage de L\'evy \citep[p.~15]{ResnickHeavy},
these observations are also independent and identically distributed. 
Assume that $k_1=k_1(n)\to \infty$ and $k_1/n\to 0$ as $n\to\infty$, and choose normalising constants 
$b_{n/k_1}\sim (dn/k_1)^{1/\alpha}$. 
Here $d$ corresponds to the total mass of the angular measure $H_{\boldsymbol{Y}}$ from Definition~\ref{mrv}(ii), and including it in the normalising constant leads to the normalisation of the angular measure as shown in \eqref{eq:ncd} in Appendix~\ref{sec:3.1}.
By definition,
$\frac1{n}N_n$ is the empirical estimator of  
$\P(R\geq b_{n/k_1}) \sim d b_{n/k_1}^{-\alpha}\sim  k_1/n$ giving 
\beam\label{eq:convp}
\frac1{k_1}N_n & \stp & 1,\quad n\to\infty.
\eeam
Assume that as $n\to\infty$, $k_1=k_1(n)\to \infty$, and select a sequence {$k_2=k_2(k_1)\to\infty$} such that $k_2/k_1\to 0$.
We modify the estimator~\eqref{initialest} by first disregarding all small observations and only take the $N_n$ observations with radial component larger than $b_{n/k_1}$  into account.
For fixed $k_1$ define conditional random vectors 
\beam\label{eq:Rrandom}
\tilde{\boldsymbol{Y}}_\ell = d^{1/\alpha}\boldsymbol{Y}_\ell/b_{n/k_1}\mid R_\ell \geq b_{n/k_1},
\quad \ell\in\{1,\dots,N_n\}.
\eeam

\begin{lemma}\label{lem:Ytrans}
Let $\bsy\in{\rm RV}_+^d(\alpha)$ with angular decomposition~\eqref{angular.eq}.
Assume that as $n\to\infty$, $k_1=k_1(n)\to \infty$, and select a sequence {$k_2=k_2(k_1)\to\infty$} such that $k_2/k_1\to 0$. 
Choose the normalising constants $b_{n}/{k_i}$ such that for $i\in \{1,2\},$
\begin{align*}
\mathbb{P}(R\geq b_{{n}/{k_i}})\sim{\frac{k_i}{n}.}
\end{align*}
Consider the conditional random vectors $\tilde{\bs Y}$ as in~\eqref{eq:Rrandom}.
Then $\tilde{\bsy}\in{\rm RV}_+^d(\alpha)$ with angular decomposition $(\tilde \bsy/\tilde R, \tilde R)$ and angular measure $\tilde{H}_{\boldsymbol{Y}}$ normalised to a probability measure  on ${\Theta}^{d-1}_+$. 
Furthermore, convergence to the limiting angular measure $\tilde{H}_{\boldsymbol{Y}}$ is preserved as we condition on $b_{ k_1/k_2}$ such that $\P(\tilde R\geq b_{ k_1/k_2})\sim {k_2/k_1}$, namely,
\begin{align}
\mathbb{P}\Big(\frac{\tilde{\boldsymbol{Y}}}{\tilde R}\in \cdot \mid \tilde{R}\geq  b_{k_1/k_2} \Big)
&\overset{w}{\to} \tilde H_{\boldsymbol{Y}}(\cdot),\quad n\to\infty.
\end{align}
\end{lemma}

\begin{proof}
	We first show that $\tilde{\bsy}\in{\rm RV}_+^d(\alpha)$. From~\eqref{eq:Rrandom} we obtain  $\tilde{R}=\norm{\tilde{\boldsymbol{Y}}}$. {The following exploits Lemma~\ref{def:mrvc}.}  To prove regular variation, since $\tilde{\boldsymbol{Y}}$ is conditioned on $n$, we consider as constants $b_t$, indexed by $t$ and independent of $n$, as $t\to \infty$. Regular variation of $\tilde{R}$, for fixed $n$, follows along the same lines as the argument used for changing $n$, in~\eqref{conditionalR}:
	\begin{align} \label{conditionalRt}
	\mathbb{P}(\tilde R >  b_{t}) &= \mathbb{P}(d^{1/\alpha}R/b_{n/k_1} > b_{t} \mid R \geq b_{n/k_1})=\frac{\mathbb{P}(R > d^{-1/\alpha}b_{n/k_1}b_{t} , R \geq b_{n/k_1})}{\mathbb{P}(R \geq b_{n/k_1})} \nonumber\\
	&\sim\frac{\mathbb{P}(R > b_{nt/k_1}, R > b_{n/k_1})}{\mathbb{P}(R_\ell \geq b_{n/k_1})}=\frac{\mathbb{P}(R_\ell > b_{nt/k_1})}{\mathbb{P}(R_\ell \geq b_{n/k_1})}\sim \frac{db_{nt/k_1}^{-\alpha}}{db_{n/k_1}^{-\alpha}}\sim t^{-1} {\sim \P(R>b_{t}),\quad t\to\infty.}
	\end{align}
We next show~\eqref{eq:mrvc}:
		\begin{align}\label{refmrv}
		\mathbb{P}\Big(\frac{\tilde{\boldsymbol{Y}}}{\tilde R}\in \cdot \mid \tilde{R}\geq  b_t \Big)
		&=\mathbb{P}\Big(\frac{ \boldsymbol{Y}}{ R}\in \cdot \mid d^{1/\alpha} R/ b_{n/k_1} \geq b_t, R \geq b_{n/k_1}\Big)\nonumber\\
		&=\mathbb{P}\Big(\frac{ \boldsymbol{Y}}{ R}\in \cdot \mid {R}\geq d^{-1/\alpha}b_{n/k_1} b_{t}, R \geq b_{n/k_1} \Big) \nonumber\\
		&\sim\mathbb{P}\Big(\frac{ \boldsymbol{Y}}{ R}\in \cdot \mid {R}\geq  b_{nt/k_1}, R \geq b_{n/k_1}  \Big) \nonumber\\
		&\sim\mathbb{P}\Big(\frac{ \boldsymbol{Y}}{ R}\in \cdot \mid {R}\geq  b_{nt/k_1} \Big) \nonumber\\
		&\overset{w}{\to} \tilde H_{\boldsymbol{Y}}(\cdot),\quad t\to\infty.
		\end{align}
  This proves the first claim, that $\tilde{\bsy}\in{\rm RV}_+^d(\alpha)$ by Lemma~\ref{def:mrvc}.

	The constants $b$ appearing in the estimators constructed from the empirical angular measure, i.e.,~\eqref{eq:estang}, are naturally indexed by the sample size $n$. To this end, we show that both convergence results,~\eqref{conditionalRt} and~\eqref{refmrv}, remain valid for the conditional random vector $\tilde{\bs{Y}}$ as we let the constants $b$ change with $n$. We first consider the normalising constants. 
	Choose normalising constants $b_{k_1/k_2}\sim (dk_1/k_2)^{1/\alpha}$ and note that  
$d^{-1/\alpha}b_{n/k_1} b_{k_1/k_2}\sim (n/k_1)^{1/\alpha}(dk_1/k_2)^{1/\alpha} = (dn/k_2)^{1/\alpha}\sim b_{n/k_2}$. Note that, by~\eqref{eq:ncd} and the discussion thereafter, the choice of the normalising constants is correct for $\bsy$.
Then, for the conditional radial component we find
\begin{align} \label{conditionalR}
\mathbb{P}(\tilde R >  b_{k_1/k_2}) &= \mathbb{P}(d^{1/\alpha}R/b_{n/k_1} > b_{k_1/k_2} \mid R \geq b_{n/k_1})=\frac{\mathbb{P}(R > d^{-1/\alpha}b_{n/k_1}b_{k_1/k_2} , R \geq b_{n/k_1})}{\mathbb{P}(R \geq b_{n/k_1})} \nonumber\\
&\sim\frac{\mathbb{P}(R > b_{n/k_2}, R > b_{n/k_1})}{\mathbb{P}(R_\ell \geq b_{n/k_1})}=\frac{\mathbb{P}(R_\ell > b_{n/k_2})}{\mathbb{P}(R_\ell \geq b_{n/k_1})}\sim \frac{db_{n/k_2}^{-\alpha}}{db_{n/k_1}^{-\alpha}}\sim \frac{k_2}{k_1} {\sim \P(R>b_{k_1/k_2}).}
\end{align}
Finally, we show that the convergence to the angular measure of ${\boldsymbol{Y}}$ is preserved in the limit as we condition on $b_{k_1/k_2}$, where both $k_1,k_2$ are functions of $n$, as choice of normalising constants:
\begin{align}\label{rel3}
\mathbb{P}\Big(\frac{\tilde{\boldsymbol{Y}}}{\tilde R}\in \cdot \mid \tilde{R}\geq  b_{k_1/k_2} \Big)
&=\mathbb{P}\Big(\frac{ \boldsymbol{Y}}{ R}\in \cdot \mid d^{1/\alpha} R/ b_{n/k_1} \geq b_{k_1/k_2}, R \geq b_{n/k_1}\Big)\nonumber\\
&=\mathbb{P}\Big(\frac{ \boldsymbol{Y}}{ R}\in \cdot \mid {R}\geq d^{-1/\alpha}b_{n/k_1} b_{k_1/k_2}, R \geq b_{n/k_1} \Big) \nonumber\\
&\sim\mathbb{P}\Big(\frac{ \boldsymbol{Y}}{ R}\in \cdot \mid {R}\geq  b_{n/k_2}, R \geq b_{n/k_1}  \Big) \nonumber\\
&=\mathbb{P}\Big(\frac{ \boldsymbol{Y}}{ R}\in \cdot \mid {R}\geq  b_{n/k_2} \Big) \nonumber\\
&\overset{w}{\to} \tilde H_{\boldsymbol{Y}}(\cdot),\quad n\to\infty. 
\end{align}
Hence,  $ b_{k_1/k_2}$ is a correct normalising constant for $\tilde\bsy$, and the normalised angular measure of $\tilde{\bsy}$ converges to the same angular measure of $\tilde{H}_{\boldsymbol{Y}}$ as $\boldsymbol{Y}$.
\end{proof}

We now replace the normalising constants $b_{k_1/k_2}$ in~\eqref{conditionalR} by the order statistic $\tilde R^{(k_2)}$, the $k_2-$th largest order statistic of $\tilde R_\ell$ over th random number $N_n$ of observations. 
In Lemma~\ref{randomRb} we show that $\tilde R^{(k_2)}/b_{k_1/k_2}\overset{P}{\to}1$ for $k_1,k_2$ as above, and thus  choosing $\tilde R^{(k_2)}$ is a consistent choice.
Hence, for a continuous function $f:\Theta_+^1\to\mathbb{R}$, we consider the estimator based on the random number of observations given in~\eqref{eq:Nn}
\begin{align}\label{finalest}
\hat{\mathbb{E}}_{\tilde{{H}}_{\bsy}}[f(\boldsymbol{\omega})]=\frac{1}{k_2}\sum_{\ell=1}^{N_n}f(\boldsymbol{\omega}_\ell)\mathds{1}\{\tilde R_\ell \geq \tilde R^{(k_2)}\}.
\end{align}


The following lemma is a consequence of Theorem~4.3.2 of \citet{EKM}. In order to keep the paper self-contained, we provide a proof.

 \begin{lemma}\label{randomRb} 
	Let $\boldsymbol{X}_1,\ldots, \boldsymbol{X}_{N_n}$ be independent replicates of $\boldsymbol{X}\in{\rm RV}^{d}_+(\alpha)$. Choose $k_1=o(n)$, $k_2=o(k_1)$ and such that $n\to\infty, k_1\to\infty,k_2\to\infty$. Let $N_n\in\mathbb{N}$ be a random process such that $N_n/k_1{\scriptstyle \overset{ P}{\to}}1$. Let $(R_\ell, \boldsymbol{\omega}_\ell)$ be the angular decomposition of $\boldsymbol{X}_\ell$, and $R_{N_n}^{(k_2)}$ be the $k_2-$th largest order statistics amongst $R_\ell$, $\ell\in\{1,\ldots,N_n\}$. Let $b_{k_1/k_2}\coloneqq  (dk_1/k_2)^{1/\alpha}$. Then $R_{N_n}^{(k_2)}/b_{k_1/k_2} \overset{P}{\to}1$.
\end{lemma} 

 \begin{proof}
 For $\eps_1, \eps_2>0$, we consider
 	 \begin{align*}
 	\mathbb{P}\left(\Big| \frac{R_{N_n}^{(k_2)}}{b_{k_1/k_2}}-1\Big|>\eps_1\right)
  &=\mathbb{P}\left(\Big| \frac{R_{N_n}^{(k_2)}}{b_{k_1/k_2}}-1\Big|>\eps_1, \Big| \frac{N_n}{k_1}-1\Big|\geq\eps_2\right)+ \mathbb{P}\left(\Big| \frac{R_{N_n}^{(k_2)}}{b_{k_1/k_2}}-1\Big|>\eps_1, \Big| \frac{N_n}{k_1}-1\Big|<\eps_2\right)\\
  & =: I_n + II_n.
  \end{align*}
  We first estimate $I_n$ by 
  \begin{align*}
  \lim_{n\to\infty} I_n \leq\lim_{n\to\infty}\mathbb{P}\Big(\Big| \frac{N_n}{k_1}-1\Big|\geq\eps_2\Big) \to 0.
  \end{align*}
Now, we turn to $II_n$:
   \begin{align}
 	II_n &=\mathbb{P}\left( \frac{R_{N_n}^{(k_2)}}{b_{k_1/k_2}}>1+\eps_1, \Big| \frac{N_n}{k_1}-1\Big|<\eps_2\right)
  +\mathbb{P}\left( \frac{R_{N_n}^{(k_2)}}{b_{k_1/k_2}}<1-\eps_1, \Big| \frac{N_n}{k_1}-1\Big|<\eps_2\right)\nonumber\\
	&\leq \mathbb{P}\left( \frac{R_{N_n}^{(k_2)}}{b_{k_1/k_2}}>1+\eps_1, {N_n}<k_1(1+\eps_2)\right)
 +\mathbb{P}\left( \frac{R_{N_n}^{(k_2)}}{b_{k_1/k_2}}<1-\eps_1, {N_n}>k_1(1-\eps_2)\right)\nonumber\\
	&\leq \mathbb{P}\left( \frac{R_{k_1(1+\eps_2)}^{(k_2)}}{b_{k_1/k_2}}>1+\eps_1, {N_n}<k_1(1+\eps_2)\right)
 +\mathbb{P}\left( \frac{R_{k_1(1-\eps_2)}^{(k_2)}}{b_{k_1/k_2}}<1-\eps_1, {N_n}>k_1(1-\eps_2)\right)\label{eq:***}\\
	&\leq\mathbb{P}\left( \frac{R_{k_1(1+\eps_2)}^{(k_2)}}{b_{k_1/k_2}}>1+\eps_1\right)
 +\mathbb{P}\left( \frac{R_{k_1(1-\eps_2)}^{(k_2)}}{b_{k_1/k_2}}<1-\eps_1\right)\nonumber
 	\end{align}
   The inequality {\eqref{eq:***}} follows from the fact that $R_{N_n}^{(k_2)}$ is the $k_2-$th largest radial component amongst a sample of size $N_n$, and therefore it is an increasing function of $N_n$.
 	 
 	 Now, $R_{k_1(1+\eps_2)}^{(k_2)}$ is the $k_2-$th largest radial component amongst a sample of size $k_1(1+\eps_2)$, thus we have 
 	 $$\frac{R_{k_1(1+\eps_2)}^{(k_2)}}{b_{k_1/k_2}}=\frac{R_{k_1(1+\eps_2)}^{(k_2)}}{b_{k_1(1+\eps_2)/k_2}}\frac{b_{k_1(1+\eps_2)/k_2}}{b_{k_1/k_2}}\overset{ P}{\to}(1+\eps_2)^{1/\alpha},
 	 $$
 	 and similarly,
 	 $$\frac{R_{k_1(1-\eps_2)}^{(k_2)}}{b_{k_1/k_2}}=\frac{R_{k_1(1-\eps_2)}^{(k_2)}}{b_{k_1(1-\eps_2)/k_2}}\frac{b_{k_1(1-\eps_2)/k_2}}{b_{k_1/k_2}}\overset{ P}{\to}(1-\eps_2)^{1/\alpha}.
 	 $$
 	 Therefore,  for every $\eps_1>0$ we can find $0<\eps_2<\min\Big((1+\eps_1)^{\alpha}-1,1-(1-\eps_1)^{\alpha}\Big)$, such that 
 	 $$\lim_{n\to\infty}\mathbb{P}\left( \frac{R_{k_1(1+\eps_2)}^{(k_2)}}{b_{k_1/k_2}}>1+\eps_1\right) = \lim_{n\to\infty}\mathbb{P}\left( \frac{R_{k_1(1-\eps_2)}^{(k_2)}}{b_{k_1/k_2}}<1-\eps_1\right) = 0,
 	 $$
 	 implying $R_{N_n}^{(k_2)}/b_{k_1/k_2} \overset{ P}{\to}1$.
 \end{proof}

\subsection{Asymptotic normality}\label{sec:4.2}

For independent replicates $\boldsymbol{Y}_1,\ldots,{\boldsymbol{Y}_n}$ of $\boldsymbol{Y}\in {\rm RV}^d_+(\alpha)$ with standardised margins and angular decomposition $(R,\boldsymbol{\omega})$ as in~\eqref{angular.eq}, \citet{KK} investigate the asymptotic properties of the estimator $\hat{\mathbb{E}}_{{\tilde{H}}_{\bsy}}[f(\boldsymbol{\omega})]$ as in \eqref{initialest}.
The intermediate thresholding procedure of Section~\ref{sec:thresh}, however,  chooses observations with the largest radial components, and then starts from a sample of random size $N_n$ leading to the estimator \eqref{finalest}.
To this estimator, however, we cannot apply the CLT of \cite[Theorem~4]{KK}. 

Because of~\eqref{conditionalR} we rewrite~\eqref{finalest} as
\begin{align}\label{finalest1}
\hat{\mathbb{E}}_{\tilde{{H}}_{\bsy}}[f(\boldsymbol{\omega})]=\frac{1}{k_2}\sum_{\ell=1}^{N_n}f(\boldsymbol{\omega}_\ell)\mathds{1}\{R_\ell \geq R^{(k_2)}\}.
\end{align}
Now we shall use the technique of \citet{lars}, modifying {their} arguments to allow for the random sample size due to the intermediate thresholding.  

As we want to estimate the extremal dependence measure and the squared scalings given in Definition~\ref{scaledef} of the two components of a vector like $\boldsymbol{\tilde{T}}^{ij}$ from~\eqref{def:tim}, we assume that $d=2$.

To keep the section self-contained, we repeat the main theorems from Section~\ref{funcCLT}.  Recall that, under~appropriate conditions, for $\sigma^2={\rm Var}_{\tilde{{H}}_{\bsy}}(f({{\boldsymbol{\omega}}_1}))$, we first want to show that the~two-parameter~process 
\begin{align}\label{w1}
W_{k_1}(t,s)=\frac{1}{\sigma\sqrt{k_2}}\sum_{i=1}^{\lfloor k_1t\rfloor}\Big(f(\boldsymbol{\omega}_i)-\mathds{E}_{\tilde{H}_{\boldsymbol{Y}}}[f(\boldsymbol{\omega}_1)]\Big)\mathds{1}\{R_i/{b_{{k_1}/{k_2}}}\geq s^{-1/\alpha}\}
\end{align}
converges weakly as $n\to\infty$ {in $D([0,\infty)^2)$ to a Brownian sheet $W$, a Wiener process on $[0,\infty)^2$ with covariance function $(t_1\wedge t_2)(s_1\wedge s_2)$ for $(t_1,s_1), (t_2,s_2)\in\R_+^2$.}
This relies on a Donsker-type CLT. The Theorem~and its proof is motivated by Theorem~1 of \citet{lars}, however, needs to be extended from a deterministic sample size $n$ to a random sample size $N_n$.

\begin{theorem}\label{asnorma}
	Let $\{\bsy_i: i\geq 1\}$ be independent replicates of the standardised vector $\bsy\in {\rm RV}^2_+(\alpha)$ with angular decomposition $(R,\boldsymbol{\omega})$. 
	Let $R$ have have distribution function $F$ and survival function $\bar{F}=1-F$. 
	Choose $k_1(n), k_2(k_1)\to\infty$ and $k_1=o(n), k_2=o(k_1)$ as $n\to\infty$.
	Let the normalising constants $b_{k_1/k_2}$ be chosen such that $\bar F(b_{ k_1/k_2} ) \sim k_2/k_1$.
	Assume that  
	\begin{align}\label{assump}
	\lim\limits_{n\to \infty}\sqrt{k_2}\bigg(\frac{k_1}{k_2}
	\mathbb{E}[f( \boldsymbol{\omega}_1 )\mathds{1} {\{R_1\geq b_{ k_1/k_2} s^{-1/ \alpha} \}}] 
	- \mathbb{E}_{\tilde{H}_{\boldsymbol{Y}}} [f(\boldsymbol{\omega}_1)] \frac{k_1}{k_2} \bar{F}(b_{ k_1/k_2}
	s^{-1/\alpha}) \bigg)=0
	\end{align}
	holds locally uniformly for $s\in[0, \infty)$, and assume that
	$\sigma^2={\rm Var}_{\tilde{{H}}_{\bsy}}(f({{\boldsymbol{\omega}}_1}))>0$. 
	Then
	\begin{align}\label{eq:limitest}
	W_{k_1}(t,s)\overset{w}{\to}W(t,s), \hspace{5mm} n\to\infty.
	\end{align}
\end{theorem}


\begin{proof}[\textbf{Proof of Theorem~\ref{asnorma}.}]

To prove the convergence $W_{k_1}\overset{ w}{\to} W$ in $D([0,\infty)^2)\times [0,\infty)^2$ we use a classical proof technique to show finite dimensional convergence and then tightness. 
 
We start with finite dimensional convergence and define functions $h(\boldsymbol{\omega}_i):=f(\boldsymbol{\omega}_i)-\mathds{E}_{\tilde{H}_{\boldsymbol{Y}}}[f(\boldsymbol{\omega}_1)]$. For the sake of notational simplicity we use $\gamma=1/\alpha$.
	For a given interval $(t_2,t_1]\times (s_2,s_1]$ such that $t_1\geq t_2, s_1\geq s_2$ consider
	\begin{align*}
	& W_{k_1}((t_2,t_1]\times (s_2,s_1]) =W_{k_1}(t_1,s_1)+W_{k_1}(t_2,s_2)-W_{k_1}(t_1,s_2)-W_{k_1}(t_2,s_1)\\
	&=\frac{1}{\sigma\sqrt{k_2}}\sum_{i=1}^{ \lfloor k_1 t_1\rfloor }h(\boldsymbol{\omega}_i)\mathds{1}\{R_i/{b_{{{k_1}/{k_2}}}}\geq s_1^{-\gamma}\}+\frac{1}{\sigma\sqrt{k_2}}\sum_{i=1}^{ \lfloor k_1 t_2\rfloor }h(\boldsymbol{\omega}_i)\mathds{1}\{R_i/{b_{{{k_1}/{k_2}}}}\geq s_2^{-\gamma}\}\\
	&\quad -\frac{1}{\sigma\sqrt{k_2}}\sum_{i=1}^{ \lfloor k_1 t_1\rfloor }h(\boldsymbol{\omega}_i)\mathds{1}\{R_i/{b_{{{k_1}/{k_2}}}}\geq s_2^{-\gamma}\}-\frac{1}{\sigma\sqrt{k_2}}\sum_{i=1}^{ \lfloor k_1 t_2\rfloor }h(\boldsymbol{\omega}_i)\mathds{1}\{R_i/{b_{{{k_1}/{k_2}}}}\geq s_1^{-\gamma}\}\\
	&=\frac{1}{\sigma\sqrt{k_2}}\sum_{i=1}^{ \lfloor k_1 t_1\rfloor }h(\boldsymbol{\omega}_i)\mathds{1}\{R_i/{b_{{{k_1}/{k_2}}}}\in [s_1^{-\gamma},s_2^{-\gamma}) \}
	 -\frac{1}{\sigma\sqrt{k_2}}\sum_{i=1}^{ \lfloor k_1 t_2\rfloor }h(\boldsymbol{\omega}_i)\mathds{1}\{R_i/{b_{{{k_1}/{k_2}}}}\in [s_1^{-\gamma},s_2^{-\gamma}) \}\\
	&=\frac{1}{\sigma\sqrt{k_2}}\sum_{i=\lfloor k_1 t_2\rfloor+1}^{ \lfloor k_1 t_1\rfloor }h(\boldsymbol{\omega}_i)\mathds{1}\{R_i/{b_{{{k_1}/{k_2}}}}\in [s_1^{-\gamma},s_2^{-\gamma}) \}.
	\end{align*}
	Now write 
	\begin{align}\label{eq:nk2}
	N_{k_2}\coloneq N_{k_1}(t_1,t_2,s_1,s_2)=\sum_{i=\lfloor k_1 t_2\rfloor+1}^{ \lfloor k_1 t_1\rfloor} \mathds{1}\{R_i/{b_{{{k_1}/{k_2}}}}\in [s_1^{-\gamma},s_2^{-\gamma}) \}.
	\end{align}
 Furthermore, let $i(j,k_1)$ be the $j$-th index $i\in \{\lfloor k_1 t_2\rfloor+1,\ldots, \lfloor k_1 t_1\rfloor\}$ for which $R_i/{b_{{{k_1}/{k_2}}}}\in [s_1^{-\gamma},s_2^{-\gamma})$.
	{Then we rewrite $W_{k_1}$ with~\eqref{eq:nk2} as}
	\begin{align*}
	&W_{k_1}((t_2,t_1]\times (s_2,s_1]) =\frac{1}{\sigma\sqrt{k_2}} \sum_{j=1}^{N_{k_2}} \Big( f(\boldsymbol{\omega}_{i(j,k_1)})- \mathds{E}_{\tilde{H}_{\boldsymbol{Y}}}[f({\boldsymbol{\omega}_1})]\Big)\\
	&=\frac{1}{\sigma\sqrt{k_2}} \sum_{j=1}^{N_{k_2}} \Big( f(\boldsymbol{\omega}_{i(j,k_1)})- \mathds{E}_{\tilde{H}_{\boldsymbol{Y}}}[f({\boldsymbol{\omega}}_{i(j,k_1)})]\Big)
	+\frac{1}{\sigma\sqrt{k_2}} N_{k_2} \Big( \mathds{E}_{\tilde{H}_{\boldsymbol{Y}}}[f({\boldsymbol{\omega}}_{i(j,k_1)})]-\mathds{E}_{\tilde{H}_{\boldsymbol{Y}}}[f({\boldsymbol{\omega}}_1)]\Big)\\
	&=: {A_{k_1}((t_2,t_1]\times (s_2,s_1])+B_{k_1}((t_2,t_1]\times (s_2,s_1])} =: A_{k_1}+B_{k_1}.
	\end{align*}
	We first show that $B_{k_1}\overset{ P}{\to}0$ as $n\to\infty$. 
    To this end, {we set $F([a_1,a_2))):= \bar F(a_1)-\bar F(a_2)$} and note that 
	\begin{align*}
	\sigma B_{k_1} =& \frac{N_{k_2}}{k_2}\sqrt{k_2}\Big(\mathds{E}[h(\boldsymbol{\omega}_1)\mid R_1/b_{k_1/k_2}\in [s_1^{-\gamma},s_2^{-\gamma}) ]\Big)\\
	=& \frac{N_{k_2}/k_2}{{\scriptscriptstyle \frac{k_1}{k_2}} F(b_{k_1/k_2}
 [s_1^{-\gamma},s_2^{-\gamma}))}\\
  & \times \sqrt{k_2}\Big(\mathds{E}\left[f(\boldsymbol{\omega}_1)\frac{k_1}{k_2}\mathds{1}\{R_1\in {b_{{{k_1}/{k_2}}}}[s_1^{-\gamma},s_2^{-\gamma}) \}\right]-\mathds{E}_{\tilde{H}_{\boldsymbol{Y}}}[f({\boldsymbol{\omega}_1})]\frac{k_1}{k_2} F(b_{k_1/k_2}[s_1^{-\gamma},s_2^{-\gamma}))\Big)\\
	=& \frac{N_{k_2}/k_2}{{\scriptscriptstyle \frac{k_1}{k_2}}F(b_{k_1/k_2}[s_1^{-\gamma},s_2^{-\gamma}))}(C_{k_1}(s_1)-C_{k_1}(s_2)).
	\end{align*}
	As the observations are independent and identically distributed, Theorem~6.2 (9) in \citet{ResnickHeavy} implies {by standardisation and $\gamma=1/\alpha$} that 
	\begin{align*}
	\frac{N_{k_2}}{k_2} 
	&{=} (t_1-t_2)\frac{N_{k_2}}{\lfloor k_1t_1\rfloor-\lfloor k_1t_2\rfloor} \frac{\lfloor k_1t_1\rfloor-\lfloor k_1t_2\rfloor}{ k_2(t_1-t_2)}\\
	&\overset{ P}{\to} (t_1-t_2)\nu_\alpha[s_1^{-\gamma},s_2^{-\gamma})
	=(t_1-t_2)(s_1-s_2),
	\end{align*}
	where we have used {\eqref{eq:nk2} and the fact} that $\lfloor a \rfloor/a \to 1$ as $a\to\infty$.
	The regular variation of $\bar F$ likewise implies that
	\begin{align*}
	\lim_{n\to\infty} \frac{k_1}{k_2}F(b_{{k_1}/{k_2}}[s_1^{-\gamma},s_2^{-\gamma})) = \nu_\alpha[s_1^{-\gamma},s_2^{-\gamma})= s_1-s_2.
	\end{align*}
	Assumption \eqref{assump} implies that $C_{k_1}(s_1)-C_{k_1}(s_2)\to 0$ locally uniformly, so $B_{k_1}\overset{ P}{\to}0$.
	
	We now consider the process $A_{k_1}$. 
 {First note that the d\'ecoupage de L\'evy \citep[e.g., p.~212]{sres} implies that the sequence {$(\boldsymbol{\omega}_{i(j,k_1)}: j\in \{1,\ldots, N_{k_2}\})$} consists of independent and identically distributed random variables.}
 Let 
	\begin{align*}
	\sigma_{k_1}^2=\var(f(\boldsymbol{\omega}_{i(j,k_1)}))=\var(f(\boldsymbol{\omega}_1) \mid  R_1/b_{k_1/k_2}\in[s_1^{-\gamma},s_2^{-\gamma})  ),
	\end{align*}
	and consider the process 
	\begin{align*}
	Z_{k_1}(r)=\frac{1}{\sigma_{k_1}\sqrt{k_2}}\sum_{j=1}^{\lfloor k_2r\rfloor} \Big(f(\boldsymbol{\omega}_{i(j,k_1)})-\mathds{E}_{\tilde{H}_{\boldsymbol{Y}}}[f(\boldsymbol{\omega}_{i(j,k_1)})]\Big),\quad {r>0.}
	\end{align*}
 Following the proof of \citet{lars} and Theorem~3 of \citet{edms}, a functional central limit Theorem~for triangular arrays gives that $Z_{k_1}\overset{ w}{\to}Z$ in $D[0,\infty)$, where $Z$ is {Brownian motion.}
 By the joint convergence of $Z_{k_1}$ and that of $N_{k_2}/k_2\overset{ P}{\to}(t_1-t_2)(s_1-s_2)$,  we obtain by composition,
	\begin{align*}
	Z_{k_1}(N_{k_2}/k_2)=\frac{1}{\sigma_{k_1}\sqrt{k_2}}\sum_{j=1}^{N_{k_2}}\Big(f(\boldsymbol{\omega}_{i(j,k_1)})-\mathds{E}_{\tilde{H}_{\boldsymbol{Y}}}[f(\boldsymbol{\omega}_{i(j,k_1)})]\Big)\overset{ \mathcal{D}}{\to}Z\Big((t_1-t_2)(s_1-s_2)\Big).
	\end{align*}
	To obtain the limit for $A_{k_1}$, note that
 $$A_{k_1}((t_2,t_1]\times (s_2,s_1])=({\sigma_{k_1}}/{\sigma})Z(N_{k_2}/k_2){\std} Z((t_1-t_2)(s_1-s_2)),$$
where we have used that $\sigma_{k_1}\to \sigma$ owing to the regular variation and the fact that $\sigma>0$. {This implies then that
 $$W_{k_1}((t_2,t_1]\times (s_2,s_1]) {\std} N\Big(0,(t_1-t_2)(s_1-s_2)\Big),$$
i.e., the limit variable is centered normal with variance $(t_1-t_2)(s_1-s_2)$.}

	Now, we repeat the procedure for disjoint sets in $[0,\infty)^2$, say $(t_{2m},t_{1m}]\times(s_{2p}, s_{1p}]$ for $m\in\{1,\ldots,M\}$ and $p\in\{1,\ldots,P\}$. For each such combination, let 
	\begin{align*}
	N_{k_2}^{m,p}\coloneqq N_{k_1}(t_{1m},t_{2m},s_{1p},s_{2p})=\sum_{i=\lfloor k_1 t_{2m}\rfloor+1}^{ \lfloor k_1 t_{1m}\rfloor} \mathds{1}\{R_i/{b_{{{k_1}/{k_2}}}}\in [s_{1p}^{-\gamma},s_{2p}^{-\gamma}) \},
	\end{align*}
	and similarly define $i_{m,p}(j,k_1)$ as the $j$-th index {$i\in \{\lfloor k_1 t_{2m}\rfloor+1,\ldots, \lfloor k_1 t_{1m}\rfloor\}$} such that $R_i/{b_ {{{k_1}/{k_2}}}}\in [s_{1p}^{-\gamma},s_{2p}^{-\gamma})$.
	
	We again decompose $W_{k_1}((t_{2m},t_{1m}]\times(s_{2p}, s_{1p}])$ into
	\begin{align*}
	W_{k_1}((t_{2m},t_{1m}]\times(s_{2p}, s_{1p}])=A_{k_1}^{m,p}+B_{k_1}^{m,p},
	\end{align*}
	where $B_{k_1}^{m,p}\overset{ P}{\to}0$ again for all combinations $(m,p)$, using the same arguments as for $B_{k_1}$. 
 
 Similarly, for the processes $Z_{k_1}^{m,p}$, with corresponding exceedance  indices  $i_{m,p}(j,k_1)$, the d\'ecoupage de L\'evy yields that for each fixed $k_1$ the $P$ sequences $(\boldsymbol{\omega}_{i_{m,p}(j,k_1)}: j\in\{1,\ldots,N_{k_2}^{m,p}\})$ are independent for $p\in\{1,\dots,P\}$, implying that the processes $Z_{k_1}^{m,p}$ are also independent. {Independence across the $m$ index follows by the independence of the observations in time.}
 
 Hence, the convergence established previously for the single process $Z_{k_1}$, holds also jointly, giving 
	\begin{align*}
	(Z_{k_1}^{1,1},Z_{k_1}^{1,2},\ldots,Z_{k_1}^{M,P})\overset{ w}{\to} (Z_1,Z_2,\ldots,Z_{MP}),
	\end{align*}
	where the limit is an $M \times P$-dimensional Brownian motion. 
 
 Finally, composing with $N_{k_1}^{m,p}/k_2$, which converges in probability to $(t_{1m}-t_{2m})(s_{1p}-s_{2p})$, for $m\in\{1,\ldots,M\}$, $p\in\{1,\ldots,P\}$ gives
	\begin{align*}
	(A_{k_1}^{1,1},A_{k_1}^{1,2},\ldots, A_{k_1}^{M,P})\overset{ \mathcal{D}}{\to} N(0,{\rm diag}\Big((t_{11}-t_{21})(s_{11}-s_{21}),\ldots,(t_{1M}-t_{2M})(s_{1P}-s_{2P}) \Big)).
	\end{align*}
{This implies then that
 \begin{align*}
 &(W_{k_1}((t_{21},t_{11}]\times (s_{21},s_{11}]),\dots,W_{k_1}((t_{2M}, t_{1M}]\times (s_{2P},s_{1P}]) \\
 {\std} & (N_1\Big(0,(t_{11}-t_{21})(s_{11}-s_{21})\Big),\dots,N_{MP}\Big(0,(t_{1M}-t_{2M})(s_{1P}-s_{2P})\Big) ,
 \end{align*}
i.e., the limit vector has independent centered normal components}
 with variances $(t_{2m}-t_{1m})(s_{2p}-s_{1p})$ for all $m$ and $p$, which entails finite-dimensional convergence of $W_{k_1}$ to a Brownian sheet. Indeed, for $t_1\geq t_2$, $s_1\geq s_2$, 
 the only contribution to the limiting covariance between $W_{k_1}(t_1,s_1)$ and $W_{k_1}(t_2,s_2)$ 
 is due to the variance of $W_{k_1}(t_2,s_2)$, whose limit equals $t_2s_2=(t_1\wedge t_2)(s_1\wedge s_2)$, corresponding to a Brownian sheet. 

	It remains to show that the process $W_{k_1}$ is tight. Since \citet{lars} consider a one-parameter process {with continuous sample paths}, they apply the moment condition of Theorem~13.5 of \citet{billingsley}.
Such moment estimates have been extended to a multi-parameter process in $D([0,1]^q)$ for an arbitrary finite dimension $q$ and \citet[equation~(3)]{BiWi} provides a condition for tightness similar to equation~(13.14) of \citet{billingsley}.
 Theorem~3 of that paper states that if an appropriate moment condition holds, and if the process vanishes on the lower bound of the domain space, then tightness of the process follows. Clearly, in our case if any of $t$ or $s$ are set to $0$, it follows that $W_{k_1}=0$ almost surely also. Hence, we could apply this condition to obtain tightness {on compacts in $D([0,\infty)^2)$. }

 This theory has been extended to tightness on $D([0,\infty)^q)$ in \citet[Theorem~4.1]{ivanoff}. 
 Adapted to our two-parameter framework, \citet[Theorem~4.1]{ivanoff}  
 states that weak convergence on $D[0,\infty)^2\mapsto\mathbb{R}$ is equivalent to weak convergence on $D([0,b_1]\times [0,b_2])\mapsto\mathbb{R}$ for finite $b_1,b_2>0$.   
In order to prove this, we apply inequality~(3) of \citet{BiWi} for $\gamma_1=\gamma_2=2$ and $B$ and $C$ intervals in $\R_+^2$.
Hence, it remains to show that there exists a finite non-negative measure, $\mu$, that assigns zero to the zero vector in $\R^2$ such that 
	\begin{align*}
	\mathds{E}[|W_{k_1}((t_2,t_1]\times (s_2,s_1])|^2|W_{k_1}((t_3,t_2]\times (s_3,s_2])|^2]\leq \mu((t_2,t_1]\times (s_2,s_1])\mu((t_3,t_2]\times (s_3,s_2]).
	\end{align*}
 Using similar arguments as \citet{lars}, we show that  
	\begin{align}\label{biwibound}
	& \underset{n\to\infty }{\rm lim\, sup}\, \mathds{E}[|W_{k_1}((t_2,t_1]\times (s_2,s_1])|^2|W_{k_1}((t_3,t_2]\times (s_3,s_2])|^2]\nonumber\\
 & \le(t_1-t_2)(s_1-s_2)(t_2-t_3)(s_2-s_3).
	\end{align} 
We first consider disjoint intervals in $\R_+^2$.
	The independence of the observations in the subsets $(t_2,t_1]$ and $(t_3,t_2]$ implies that the expectation here factorises as 
	\begin{align}\label{disjointtimeindep}
	\mathds{E}[|W_{k_1}((t_2,t_1]\times (s_2,s_1])|^2]\mathds{E}[|W_{k_1}((t_3,t_2]\times (s_3,s_2])|^2],
	\end{align}
	and we deal separately with these two terms.  Let us write
	\begin{align*}
	W_{k_1}((t_2,t_1]\times (s_2,s_1])=\frac{1}{\sigma\sqrt{k_2}}\sum_{i=\lfloor k_1 t_2\rfloor+1}^{ \lfloor k_1 t_1\rfloor }\alpha_i,
	\quad \alpha_i=h(\boldsymbol{\omega}_i)\mathds{1}\{R_i/{b_{{{k_1}/{k_2}}}}\in [s_1^{-\gamma},s_2^{-\gamma})\},
	\end{align*}
	and write $W_{k_1}(\Big(t_3,t_2]\times (s_3,s_2\Big])$ as a similar sum of terms $\beta_j$. {As the $\alpha_i$ and the $\beta_i$ are i.i.d.,} squaring the product of the sums and taking the expectation gives
	\begin{align*}
	\frac{1}{\sigma^2k_2}\Big((\lfloor k_1t_1\rfloor-\lfloor k_1t_2\rfloor)\mathds{E}[\alpha_1^2]+ (\lfloor k_1t_1\rfloor-\lfloor k_1t_2\rfloor)(\lfloor k_1t_1\rfloor-\lfloor k_1t_2\rfloor-1)\mathds{E}[\alpha_1]\mathds{E}[\alpha_2] \Big),
	\end{align*}
	where we have used independence to factorise $\mathds{E}[\alpha_1\alpha_2]$. Now as $n\to\infty$, 
	\begin{align*}
	\frac{1}{\sigma^2k_2}(\lfloor k_1t_1\rfloor-\lfloor k_1t_2\rfloor)\mathds{E}[\alpha_1^2]&\sim \frac{k_1}{\sigma^2k_2} (t_1-t_2)\mathds{E}[\alpha_1^2]\\
	&\to \frac{1}{\sigma^2} (t_1-t_2)\mathds{E}_{\tilde{H}_{\boldsymbol{Y}}}[h({\boldsymbol{\omega}_1})^2]\nu_\alpha[s_1^{-\gamma},s_2^{-\gamma})\\
	&= {\frac{\sigma^2}{\sigma^2}}(t_1-t_2)(s_1-s_2).
	\end{align*}
	Proceeding similarly for the second summand, we find 
	\begin{align*}
	\frac{1}{\sigma^2\sqrt{k_2}}(\lfloor k_1t_1\rfloor-\lfloor k_1t_2\rfloor)\mathds{E}[\alpha_1]\sim \frac{k_1}{\sigma^2\sqrt{k_2}} (t_1-t_2)\mathds{E}[\alpha_1],
	\end{align*}
	which equals
	\begin{align*}
	(t_1-t_2)\sqrt{k_2}\Big(\mathds{E}\left[f(\boldsymbol{\omega}_1)\frac{k_1}{k_2}\mathds{1}\{R_1/{b_{{{k_1}/{k_2}}}}\in [s_1^{-\gamma},s_2^{-\gamma}) \}\right] -\mathds{E}_{\tilde{H}_{\boldsymbol{Y}}}[f({\boldsymbol{\omega}_1})]\frac{k_1}{k_2}F(b_{k_1/k_2}[s_1^{-\gamma},s_2^{-\gamma}) ) \Big). 
	\end{align*}
	This converges by assumption \eqref{assump} to 0 as $n\to\infty$, and likewise the second factor involving $\Big(\lfloor k_1t_1\rfloor-\lfloor k_1t_2\rfloor-1\Big)\mathds{E}[\alpha_2]/\sqrt{k_2}$. One may proceed with $\mathds{E}[|W_{k_1}((t_3,t_2]\times (s_3,s_2])|^2]$, to then finally obtain the bound $(t_1-t_2)(s_1-s_2)(t_2-t_3)(s_2-s_3)$, 
 which is~\eqref{biwibound}. 
 
	Following similar arguments one can obtain a bound for the neighbouring sets $(t_2,t_1]\times (s_2,s_1]$ and $(t_3,t_2]\times (s_2,s_1]$ by using \eqref{disjointtimeindep}.
 Thus, instead, we focus on the sets $(t_2,t_1]\times (s_2,s_1]$ and $(t_2,t_1]\times (s_3,s_2]$ which share one face along the first dimension, namely
	\begin{align*}
	\mathds{E}[|W_{k_1}((t_2,t_1]\times (s_2,s_1])|^2|W_{k_1}((t_2,t_1]\times (s_3,s_2])|^2],
	\end{align*}
	and, as before, write
	\begin{align*}
	W_{k_1}\Big(\big(t_2,t_1]\times (s_2,s_1\big]\Big)=\frac{1}{\sigma\sqrt{k_2}}\sum_{i=\lfloor k_1 t_2\rfloor+1}^{ \lfloor k_1 t_1\rfloor }\alpha_i,\quad \alpha_i=h(\boldsymbol{\omega}_i)\mathds{1}\{R_i/{b_{{{k_1}/{k_2}}}}\in [s_1^{-\gamma},s_2^{-\gamma})\}, 
	\end{align*}
	and with a similar expression for $W_{k_1}((t_2,t_1]\times (s_3,s_2]$ involving $\beta_j$ as summands, instead of $\alpha_i$.
	
	For notational convenience, we set $K_i\coloneqq\lfloor k_1t_i\rfloor$ for $i\in\{1,2\}$.
	Finally, using that $\alpha_i$ is independent of $\alpha_j$ and $\beta_j$ for $i\neq j$, and that $\alpha_i\beta_i=0$, we compute
	\begin{align*}
	&\frac{1}{\sigma^4k_2^2}\mathds{E}\Big[|\sum_{i=K_2+1}^{ K_1 }\alpha_i|^2|\sum_{i=K_2+1}^{ K_1 }\beta_i|^2\Big] =\frac{1}{\sigma^4k_2^2}\Big[(K_1-K_2)(K_1-K_2-1)\mathds{E}[\alpha_1^2]\mathds{E}[\beta_2^2]\\
	&\quad+(K_1-K_2)(K_1-K_2-1)(K_1-K_2-2)(\mathds{E}[\alpha_1]\mathds{E}[\alpha_2]\mathds{E}[\beta_3]+\mathds{E}[\alpha_1]\mathds{E}[\beta_2]\mathds{E}[\beta_3])\\
	&\quad+(K_1-K_2)(K_1-K_2-1)(K_1-K_2-2)(K_1-K_2-3)\mathds{E}[\alpha_1]\mathds{E}[\alpha_2]\mathds{E}[\beta_3]\mathds{E}[\beta_4]\Big]\\
	&\leq\frac{1}{\sigma^4k_2^2}\Big[(K_1-K_2)^2\mathds{E}[\alpha_1^2]\mathds{E}[\beta_2^2]+ (K_1-K_2)^3\Big((\mathds{E}[\alpha_1])^2\mathds{E}[\beta_2]+\mathds{E}[\alpha_1](\mathds{E}[\beta_2])^2\Big)\\
	&\quad+(K_1-K_2)^4(\mathds{E}[\alpha_1])^2(\mathds{E}[\beta_2])^2 \Big].
	\end{align*}
	As with the bound in equation~\eqref{biwibound} we obtain
	\begin{align*}
	\lim_{n\to\infty} \frac{(K_1-K_2)^2}{\sigma^4k_2^2}\mathds{E}[\alpha_1^2]\mathds{E}[\beta_2^2] = (t_1-t_2)^2(s_1-s_2)(s_2-s_3).
	\end{align*}
	All the remaining terms go to zero as $n\to\infty$ by arguments similar to those leading to the bound $(t_1-t_2)^2(s_1-s_2)(s_2-s_3)$, which
is~\eqref{biwibound}.  
So the proof is complete. 
\end{proof}

Next, we establish consistency and asymptotic normality of the estimator~\eqref{finalest1}. 

\begin{theorem}\label{asnorm}
Assume the setting in Theorem~\ref{asnorma}. 
For $n\in\N$ let $N_n\in\mathbb{N}$ be a random process satisfying $N_n/k_1 { \stp }1$.
Furthermore, let $\bsy_1,\dots,\bsy_{N_n}$ be a random number of independent replicates of $\bsy\in {\rm RV}^2_+(\alpha)$ with angular decomposition $(R,\boldsymbol{\omega})$. Define $\hat{\mathbb{E}}_{\tilde{{H}}_{\bsy}}[f(\boldsymbol{\omega})]$ as in~\eqref{finalest1}. Then \begin{align}\label{eq:limitestf}
	\sqrt{k_2}(\hat{\mathbb{E}}_{\tilde{H}_{\boldsymbol{Y}}}[f({\boldsymbol{\omega}_1})]-\mathds{E}_{\tilde{H}_{\boldsymbol{Y}}}[f({\boldsymbol{\omega}_1})])\overset{\mathcal{D}}{\to} \mathcal{N}(0, \sigma^2), \hspace{5mm} n\to\infty.
	\end{align}
\end{theorem}

\begin{proof}[Proof of Theorem~\ref{asnorm}]
	We use that 
	\begin{align}\label{eq:constlimits}
	R^{(k_2)}/b_{k_1/k_2}\overset{ P}{\to}1 ,\quad  N_n/k_1\overset{ P}{\to} 1,\quad n\to\infty,
	\end{align}
	where the first convergence follows from Lemma~\ref{randomRb} {and the second from~\eqref{eq:convp}.
		As the probability limits} in~\eqref{eq:constlimits} are constants, the three processes in {\eqref{w1} and~\eqref{eq:constlimits}} converge jointly as $n\to\infty$, and we apply the composition map $D([0,\infty)^2)\times [0,\infty)^2\mapsto\mathbb{R}$ given by {$(W,N,R)\mapsto W(N,R)$} to obtain  
	\begin{align}\label{compmap}
	W_{k_1}\Big(\frac{N_n}{k_1}, \Big(\frac{R^{(k_2)}}{b_{{k_1}/{k_2}}}\Big)^{-1/\gamma}\Big) & \overset{ \mathcal{D}}{\to}W(1,1),\quad n\to\infty,
	\end{align}
	where $W(1,1)$ has a standard normal distribution. ~\eqref{eq:limitestf} then follows upon noting that  $k_2\wedge N_n \sim k_2$ as $n\to\infty$, thus  $W_{k_1}(N_n/k_1, (R^{(k_2)}/b_{k_1/k_2})^{-\alpha})=\sqrt{k_2}/\sigma (\hat{\mathbb{E}}_{\tilde{{H}}_{\bsy}}[f(\boldsymbol{\omega})]-\mathbb{E}_{\tilde{H}_{\boldsymbol{Y}}} [f(\boldsymbol{\omega}_1)])$. 
\end{proof}

In practice, when $n$ is finite, we replace the threshold $R_{N_n}^{(k_2)}$ in \eqref{finalest1} by $R_{N_n}^{(k_2\wedge N_n)}$, in order to always obtain a well-defined quantity. {This does not affect the limit result of Theorem~\ref{asnorm} as $k_2\wedge N_n \sim k_2$ as $n\to\infty$.}

We can now deduce asymptotic normality when $\alpha=2$ and $\boldsymbol{Y}$ is max-linear.  
Recall that our goal is to estimate a squared scaling $\sigma_{12}^2$ as in Definition~\ref{scaledef}, which leads to setting $f(\boldsymbol{\omega})=2\omega_1\omega_2$ and the estimator 
\begin{align}\label{dmest}
\hat{\sigma}_{12}^2=\frac{1}{k_2}\sum_{\ell=1}^{N_n}2\omega_1\omega_2\mathds{1}\{R_\ell \geq R^{(k_2)}\}.
\end{align}

\begin{theorem} \label{thm_ci}
	{Let $\bsx \in {\rm RV}^2_+(2)$} be a max-linear model with coefficient matrix $A\in \mathbb{R}^{2\times D}_+$. 
	Assume that $k_1=k_1(n)\to\infty$, $k_2=k_2(n)\to\infty$ and $k_1=o(n)$, $k_2=o(k_1)$ as $n\to\infty$.
	 Suppose the setup and the assumptions for Theorem~\ref{asnorma} hold. Then 
	\begin{align*}
	\sqrt{k_2}(\hat{\sigma}_{12}^2-{\sigma}_{12}^2)\overset{\mathcal{D}}{\to} \mathcal{N}(0, \sigma^2), \hspace{5mm} n\to\infty,
	\end{align*}
	where $\sigma^2=2\sum_{k=1}^{D} {a_{1k}^2a_{2k}^2}/{\norm{a_k}^2}-\sigma_{12}^4$.
\end{theorem}

\begin{proof}
Asymptotic normality  is a direct consequence of Theorem~\ref{asnorm}. 
For the variance, we use the exponent measure given in~\eqref{discretspecteq}, with $a_k=(a_{1k}, a_{2k})$ being the $k-$th column of $A$, and compute
$$\mathbb{E}_{\tilde{H}_{\boldsymbol{X}}}(4\,\omega_1^2\omega_2^2)=\frac{4}{2}\sum_{k=1}^D\norm{a_k}^2\frac{a_{1k}^2}{\norm{a_k}^2}\frac{a_{2k}^2}{\norm{a_k}^2}=2\sum_{k=1}^D\frac{a_{1k}^2a_{2k}^2}{\norm{a_k}^2},$$
 and find for the variance  
 \beao
 \var(2{\omega}_1 {\omega}_2) &=& \mathbb{E}_{\tilde{H}_{{\boldsymbol{X}}}}(4\,\omega_1^2\omega_2^2)-\mathbb{E}_{\tilde{H}_{{\boldsymbol{X}}}}(2\,\omega_1\omega_2)^2\\
 &=& 2\sum_{k=1}^{D}{a_{1k}^2a_{2k}^2}/{\norm{a_k}^2}-(\sum_{k=1}^{D}a_{1k}a_{2k})^2=2\sum_{k=1}^{D}{a_{1k}^2a_{2k}^2}/{\norm{a_k}^2}-\sigma_{12}^4.\hspace{2.4cm}\qedhere
 \eeao
\end{proof}

In most situations, it is the case that the regular variation index $\alpha$ is not known and different across the marginals, and also likely to be different from $\alpha=2$, which is required by Theorem~\ref{thm_ci}. To this end, the first step would be to marginally standardise the data to standard Fr\'echet or Pareto with $\alpha=2$.
The procedure we follow for the simulation study and the data application in the next section, is to transform the observations non-parametrically by employing the empirical probability integral transform. We remark, however, that applying such transformations is likely to affect the asymptotic distribution of the estimates to a certain extent, since the rank transformations induce dependence amongst the observations. Nevertheless, there has been recent work that provides finite-sample concentration bounds which quantify the deviation of the angular measure of the rank-transformed observations from the true angular measure over certain classes of sets of the positive unit simplex \citep{rankspectralconcentration}.

\section{Estimation of the scalings and choice of the input parameters for Algorithm~\ref{datdalg}}\label{scalest}

Assume $n$ i.i.d. observations of an RMLM $\bsx$ on a DAG $\D$.
For the estimation of the scalings in the matrices $\hat{C}^{(1)}$ and $\hat{\Delta}^{(1)}$, for each pair of nodes $(i,j)$ we consider the bivariate vector $\boldsymbol{X}_{ij}=(X_i,X_j)$, with standardised  Fr\'echet$(2)$ margins.
Standardization to the latter is done by using the empirical integral transform (p. 338 in \citet{beirlant}) giving for each $i\in V$,
\begin{align}\label{empstand}
X_{\ell i}\coloneqq\Big\{-\log\Big(\frac{1}{n+1}\sum_{j=1}^{n}\mathds{1}_{\{X^{*}_{ji}\leq X^{*}_{\ell i}\}}\Big)\Big\}^{-1/2},\quad  \ell\in\{1,\dots,n\},
\end{align}
where $\bs{X^{*}}$ is the vector of a simulated RMLM. 
In the data application $\bs{X^{*}}$ simply corresponds to the original data vector.

We then compute for each observation {$\bsx_{\ell}$ for} $\ell\in\{1,\dots,n\}$ the angular decomposition as in~\eqref{angular.eq} given by
\begin{align}\label{bh}
R_\ell\coloneqq\norm{\boldsymbol{X}_{\ell,ij}}_2 =  \Big(\sum_{j\in\{i,j\}} X_{\ell, j}^2\Big)^{1/2}, 
\quad {\boldsymbol{\omega}_\ell=({\omega}_{\ell, j}: j\in \{i,j\})\coloneqq\frac{\boldsymbol{X}_{\ell, ij}}{R_\ell}}, \hspace{5mm} \ell\in\{1,\ldots,n\}.
\end{align}

\subsection{Estimation of the scalings and extremal dependence measures}\label{sec:estscal}

We estimate first the squared scalings $\hat{\sigma}_{M_{ij}}^2$ in the matrix $\hat{C}^{(1)}$ for an appropriate $1\le k_1\le n$ by 
\begin{align}\label{estscal2}
\hat{\sigma}_{M_{ij}}^2 =\frac{2}{k_1}\sum_{\ell=1}^{n}\underset{r\in \{i,j\}}{\bigvee}{\omega}_{\ell r}^2\mathds{1}{\{R_\ell \geq R^{(k_1)}\}}.
\end{align}
For the univariate scalings $\hat{\sigma}_{M_{i}}^2=\hat{\sigma}_{X_{i}}^2$ we simply consider the average of ${\omega}_{i}^2$ rather than the maximum ${\vee}_{r\in \{i,j\}}{\omega}_{\ell r}^2$  in~\eqref{estscal2}.

The matrix $\hat{\Delta}^{(1)}$ needs estimates for the squared scalings ${\sigma}_{M_{ij}}^2$ and also ${\sigma}_{M_{i,aj}}^2$, which is based on $(X_i, a X_j)$ for $a>1$ with angular decomposition $(R_{a_\ell },\boldsymbol{\omega}_{a_\ell })$, say. 
We use in principle the same estimator as in~\eqref{estscal2}, but may choose a higher radial threshold, i.e., {$k_2\le k_1$.} We remark that the estimates in the matrix $\hat{C}^{(1)}$ are used only for computing the scalars $c_1,c_2$, which enter into the transformation \eqref{def:tim}, whereas the estimated scalings in $\hat{\Delta}^{(1)}$ are used to verify condition \eqref{cond1}. To this end, a higher radial threshold leads to a lower bias and better performance of the algorithm. 
Since the estimates computed via \eqref{afterscalscal} and \eqref{scalTim} are used to verify conditions of Theorem~\ref{incdagmod}, we decide to use for these estimates the same number of threshold exceedances. 
We estimate for $1\le k_2\le k_1\le n$,
\begin{align}
\hat{\sigma}_{M_{ij}}^2 &=\frac{2}{k_2}\sum_{\ell=1}^{n}\underset{r\in \{i,j\}}{\bigvee}{\omega}_{\ell r}^2\mathds{1}{\{R_\ell \geq R^{(k_2)}\}},\nonumber\\
\hat{\sigma}_{M_{i, aj}}^2 
&=\frac{a^2+1}{k_2}\sum_{\ell=1}^{n}\underset{r\in\{i,j\}}{\bigvee}{\omega}_{a_\ell r}^2\mathds{1}{\{R_{a_\ell  }\geq R_a^{(k_2)}\}},
\label{afterscalscal}
\end{align}
where we account for the new total mass of the angular measure in~\eqref{afterscalscal} by re-weighting by $a$.

For the matrices $\hat{\Delta}^{(2)}$ and $\hat{\Delta}^{(3)}$ we 
estimate the extremal dependence measure of the vector  $\boldsymbol{T}^{ij}$ as in~\eqref{def:tim}. 
To compute the latter, we use the parameter estimates of $c_1, c'_1, c_2$, which are based on threshold values $k_1\ge k_2$, respectively, with the largest radial norms as plug-in parameters. 
This yields a smaller and random number $N_n=\sum_{\ell=1}^{n}\mathds{1}\{\norm{\bs X_{\ell, ij}}_2\geq b_{n/{k_1}}\}$ of i.i.d. pseudo random variables $\bst^{ij}_\ell$ for $\ell\in\{1,\dots,N_n\}$; recall from~\eqref{eq:convp} that $N_n$ satisfies $N_n/k_1\to 1$ in probability.

To estimate the required extremal dependence measure of the components of the standardised vector, $\boldsymbol{\tilde{T}}^{ij}_\ell$ for $\ell\in \{1,\dots,N_n\}$ 
we use the empirical integral transform to standardise $\boldsymbol{T}^{ij}_\ell$ to standard Pareto(2) random variables $\boldsymbol{\tilde{T}}^{ij}_\ell$ with unit scalings.
Let further denote by
$(R_{\ell,\boldsymbol{\tilde{T}}^{ij}}, \boldsymbol{\omega}_{\ell, \boldsymbol{\tilde{T}}^{ij}})$ the angular decomposition of $\boldsymbol{\tilde{T}}^{ij}_\ell$ for $\ell \in\{1,\dots,N_n\}$.
Then, we compute the estimate $\hat{\tau}_{ij}^2$ using~\eqref{finalest1} as
\begin{align}\label{scalTim}
\hat{\tau}_{ij}^2=\frac{2}{k_2}\sum_{\ell=1}^{N_n}{\omega}_{\ell, \tilde{T}^{ij}_{1}}{\omega}_{\ell, \tilde{T}^{ij}_{2}}\mathds{1}{\{R_{\ell,\boldsymbol{\tilde{T}}^{ij}} \geq R^{(k_2\wedge N_n)}_{\boldsymbol{\tilde{T}}^{ij}}\}}.
\end{align}
Our goal is to apply Theorem~\ref{asnorm} to the estimator in \eqref{scalTim}.
By Lemma~\ref{lemmaid2} it follows that $\boldsymbol{\tilde{T}}^{ij}\in{\rm RV}^2_+(2)$ and, thus, we may apply Theorem~\ref{asnorm} to the independent replicates $\boldsymbol{\tilde{T}}^{ij}_1,\ldots,\boldsymbol{\tilde{T}}^{ij}_{N_n}$ where $N_n$ satisfies the assumptions of Theorem~\ref{asnorm}.

\subsubsection{Calibrating the input parameters}\label{sec:c}

We briefly comment on the choice of the parameters $c_1,c'_1,c_2$. Recall that if the conditions of  Lemma~\ref{lemmaid2} hold, the atoms of the angular measure of $\boldsymbol{T}^{ij}$ for the pair of nodes $(i,j)$ are obtained by normalising the non-zero vectors $\tilde{a}_k=(a_{jk} -c_1a_{ik}, a_{jk}+c_2a_{ik})$. Due to the opposite $-/+$ signs, the parameters $c_1$, $c_2$ have opposite effects on the entries of $\tilde{a}_k$. Since the latter vectors determine the angular measure of $\boldsymbol{\tilde{T}}^{ij}$, the choices of $c_1$, $c_2$ can alter the dependence structure; thus they are important for Algorithm~\ref{datdalg}.
 
From Theorem~\ref{incdagmod}{(i)}, we select those pairs $(i,j)$ for which the margins of $\boldsymbol{\tilde{T}}^{ij}$ are asymptotically fully dependent {(see Remark~\ref{rem:two_obs})}. Therefore, ideally, we would choose $c_1$ such that the asymptotic extremal dependence is reduced substantially relative to asymptotically full dependence between those transformed pairs of nodes which are not in MWP. 
Setting $c_1$ to a fixed value for all pairs $(i,j)$ might lead either to a high number of false positives or a low number of true positives for the estimated set MWP. We choose  $c_1=(\sigma_i^2+\sigma_j^2-\sigma_{M_{ij}}^2)^{1/2}$, and $c_2=1/c_1$, and we note that it satisfies the condition that $0<c_1\leq 1$.

\begin{corollary}\label{csigim}
For a pair of node variables $(X_i, X_j)$ satisfying Theorem~\ref{incdagmod} $(i)$, it holds that $\sigma_{ij}^2=\sigma_i^2+\sigma_j^2-\sigma_{M_{ij}}^2$.
\end{corollary}

\begin{proof}
By Theorem~\ref{incdagmod} (i), $(X_i, X_j)$ can be modelled as an RMLM, and using Proposition~\ref{prop.incmat}, they can be represented in terms of a $2\times2$ coefficient matrix, say $A_{ij}^{*}\in \mathbb{R}^{2\times 2}_+$, where $a_{22}^{*}=1$. Since the rows of $A$ are standardised, $\sigma_i^2=\sigma_j^2=1$, and $\sigma_{M_{ij}}^2=a^{*2}_{11}+a^{*2}_{22}$. The difference $\sigma_i^2+\sigma_j^2-\sigma_{M_{ij}}^2=2-\sigma_{M_{ij}}^2=2-(a^{*2}_{11}+a^{*2}_{22})=a^{*2}_{12}=\sigma_{ij}^2.$
\end{proof}

From Corollary~\ref{csigim}, if $(i,j)\in$MWP, then $c_1=\sigma_{ij}$, therefore, we use it as a proxy for the strength of asymptotic dependence.
 We motivate this choice of $c_1$ in the following example. 
 
\begin{example}\label{ex:c1}
Consider the DAG $\D_3$ in Figure~\ref{introex}.
Let $i=1,j=2$, so that there is a hidden confounder $X_3$ and no causal link. 
Hence $a_{12}=0$ giving the coefficient matrix $A$ and $A_{\boldsymbol{T}^{ij}}$ from Lemmas~\ref{lemmaid1} and~\ref{lemmaid2} as
$$A=\begin{bmatrix}
a_{11}&0 & a_{13}\\ 0& a_{22}&a_{23}
\end{bmatrix}, \hspace{1cm} A_{\boldsymbol{T}^{ij}}=\begin{bmatrix}
	0& a_{22}&a_{23}-c_1a_{13}\\ 0& a_{22}&a_{23}+\frac{1}{c_1}a_{13}
\end{bmatrix}.
$$
Suppose that the conditions \eqref{cond1} are satisfied, so {that by Lemma~\ref{lemmaid1} $(iii)$, $a_{23}\geq a_{13}$.}
Using Lemma~\ref{scalcoll}, we find that the chosen $c_1=(\sigma_1^2+\sigma_2^2-\sigma_{M_{12}}^2)^{1/2}=\sigma_{12}=a_{13}$.
Therefore, the matrix $A_{\boldsymbol{T}^{ij}}$ becomes
$$A_{\boldsymbol{T}^{ij}}=\begin{bmatrix}
	0& a_{22}&a_{23}-a_{13}^2\\ 0& a_{22}&a_{23}+1
\end{bmatrix}.
$$
Now note first that, since $a_{12}=0$, the second column is not affected by $c_1$. Furthermore, the matrix entry $(2, 3)$, $a_{23}$, increases by 1 regardless of the strength of asymptotic extremal dependence $a_{13}$.

{For asymptotically full dependence, the second and the third column of $A_{\boldsymbol{T}^{ij}}$  must correspond to the same atom of the angular measure; i.e., where both entries of the third column are the same. Therefore, a large difference between the two entries in the third column indicates low extremal dependence.  From $A_{\boldsymbol{T}^{ij}}$ we see that $a_{23}-a_{13}^2$ is much smaller than $a_{23}+1$, 
indicating a substantial distance of asymptotic extremal dependence away from asymptotically full dependence.}
\end{example}

{In general, the situation may not be as clear as in Example~\ref{ex:c1}, however with $c_1$ chosen above,}
 we reach a substantial reduction in the extremal dependence relative to asymptotically full dependence between the margins of $\boldsymbol{T}^{ij}$. Hence, we opt for this choice of $c_1$.
Finally, in order to avoid problems when $\sigma_{12}\in\{0,1\}$, we use the truncated version $c_1=\min\{0.1+({\sigma}_{i}^2+{\sigma}_{j}^2-{\sigma}_{M_{ij}}^2)^{1/2}, 0.8\}$ to ensure  that $c_1\in[0.1,0.8]$. 
{Since $c_1\in [0,1]$ in Theorem~\ref{incdagmod}, this choice of truncation is not too restrictive and works reasonably well in the simulation study.}

The choice of $c'_1$ concerns the difference in the extreme dependencies in Corollary~\ref{cor37}. Based on simulation experience, we see that in general a large difference between $c_1$ and $c_1'$ leads to larger values for $\tilde{\Delta}_c$ and we select $c'_1=0.1\,c_1$, which also satisfies $0<c_1, c_1'\leq 1$.

\section{Performance of the Algorithm} \label{sec:performance}

We assess the performance of Algorithm~\ref{datdalg} by simulation. As we are only interested in estimating the set MWP we use the following reduced version of Algorithm~\ref{datdalg2}, which simply outputs all pairs such that $\hat P_{ij}=1$:
 
	\begin{algorithm}[H]
		\caption{Estimation of pairs of nodes in MWP} 
		\label{datdalg}
\flushleft 
\textbf{Parameters}:\quad $a>1, \eps_1, \eps_2, \eps_3, \eps_4, \eps_5>0$  \\
\textbf{Input}:\quad ${\hat C^{(1)}}, \hat{\Delta}^{(1)}, \hat{\Delta}^{(2)}, \hat{\Delta}^{(3)}$  \\
\textbf{Output}: Matrix $\hat{P}\in \{0,1\}^{d\times d}$ indicating the pairs of nodes in MWP\\
\textbf{Procedure:} 
		\begin{algorithmic}[1]
                 \State \textbf{set} $S_1 = \{(i,j)\in V\times V:\hat{\Delta}^{(1)}_{ij}\geq-\eps_1 \quad {\rm and} \quad \hat{\Delta}^{(1)}_{ij}-\hat{\Delta}^{(1)}_{ji}\geq-\eps_2\}$ \quad [{\rm \textcolor{blue}{Conditions~\eqref{cond1}}}]
                \State \hspace{4.5mm} $S_2=\{(i,j)\in V\times V: \hat{\Delta}^{(2)}_{ij}>1-\eps_3\}$\quad  [{\rm \textcolor{blue}{Theorem~\ref{incdagmod} (i)}}]
                \State \hspace{4.5mm} $S_3=\{(i,j)\in V\times V: \hat{\Delta}^{(2)}_{ij}>\hat{\Delta}^{(2)}_{ji}+\eps_4\}$ \quad [{\rm \textcolor{blue}{Remark~\ref{remarkai}}}] 
                \State \hspace{4.5mm} $S_4=\{(i,j)\in V\times V: \hat{\Delta}^{(3)}_{ij}<\eps_5 \, {\hat C^{(1)}_{ij}}\}$ \quad [{\rm \textcolor{blue}{Corollary~\ref{cor37}}}]
			\State\textbf{for} $(i,j)\in V\times V$  
			
			\State\hspace{5mm}\textbf{if} 
			$(i,j)\in S_1\cap S_2 \cap S_3 \cap S_4$, \textbf{set} $\hat{P}_{ij}=1$
			\State\hspace{8mm}\textbf{else set} $\hat{P}_{ij}=0$
			\State \textbf{end for}
			
		\end{algorithmic}
 \textbf{Return} $\hat{P}$ 
	\end{algorithm}

\subsection{Simulation study}\label{sec:sim}

We first simulate i.i.d. random DAGs $\mathcal{D}=(V, E)$ of dimension $|V|=d$. In a second step, for each DAG, we simulate an RMLM $\bsx$ supported on it. The simulation setup is outlined in Section~\ref{simul}. 
The objective is to apply for each DAG and its RMLM our Algorithm~\ref{datdalg} to all pairs of nodes $(i,j)$ with the goal to determine whether there are max-weighted paths $q\rightsquigarrow j\rightsquigarrow i$ for all $q\in \An(i)\cap\An(j)$.
If this holds, we say that the pair $(i,j)$ satisfies the max-weighted path property, i.e., $(i,j)\in {\rm MWP}$. 
{By Theorem~\ref{incdagmod}(i), if this is the case, then  $(X_i,X_j)$ can be represented as an RMLM and the effect of {possible} confounders of the two nodes can be ignored.}  

For each pair of components $(X_i,X_j)$ of $\bsx$ assume estimates of the scalings in~\eqref{cond1} as well as of the extremal dependence measure in parts (i) and (ii) of Theorem~\ref{incdagmod}.
The theoretical quantities are given in Definition~\ref{scaledef} and Lemma~\ref{scalcoll}, 
and we use the estimates from Appendix~\ref{scalest}.

\subsection{Simulation set-up}\label{simul}

The DAG is constructed via the upper triangular edge-weight matrix $C\in\R^{d\times d}$ defined in equation~\eqref{semequat1}.
The presence of an edge is sampled from a Bernoulli distribution with success probability $p\in\{0.1, 0.2\}$, $p$ representing different sparsity levels giving a DAG.
For the edge-weight matrix $C$, the diagonal is set to one, while the squares of the other non-zero entries are randomly generated as ${\rm Unif}([0.3, 1.5])$. 
The algorithm is applied to DAGs of different dimensions $d\in\{20,30,40\}$, each with specific sparsity level $p\in\{0.1, 0.2\}$ and with different edge-weights matrices $C$.

For the simulation of a corresponding RMLM, we first compute the coefficient matrix $A$ as in equation~\eqref{Rmlmequat} corresponding to the edge weight matrix $C$, and standardise the row norms according to equation~\eqref{abar}.  
We then simulate a random vector $\bs{X^{*}}$ starting from equation~\eqref{Rmlmequat}. Since the discrete angular measure makes the max-linear model unrealistic to use in practice, we introduce a noise vector $\boldsymbol{Z}_{\bs{\eps}}$ to the model and set \begin{align}\label{simustud}
\boldsymbol{X^{*}}=A\times_{\max} \boldsymbol{Z}+\boldsymbol{Z}_{\bs{\eps}},
\end{align}
where the innovation vector $\boldsymbol{Z}\in{\rm RV}^D_+(\alpha)$ for $\alpha\in\{2,3\}$. 
Each margin of $\bsx^{*}$ is generated such that each component of $\bsz$, $Z_i\overset{d}{=} |t_\alpha|$, where $t_\alpha$ is a $t$-distributed random variable with $\alpha$ degrees of freedom, and the independent margins of the noise vector $\boldsymbol{Z}_{\bs{\eps}}$ are such that ${Z_\eps}\overset{d}{=} 0.5\cdot t_{5}$ for $\alpha=2$ and ${Z_\eps}\overset{d}{=} 0.5\cdot t_{10}$ for $\alpha=3$. 
Since $\boldsymbol{Z}$ has a heavier tail than $\boldsymbol{Z}_{\bs{\eps}}$, the angular measure of $\boldsymbol{X^{*}}$ is asymptotically equivalent to that of $A\times_{\max} \boldsymbol{Z}$. 

To appreciate the task of detecting pairs in MWP, we depict two cases from bivariate vectors $(X_1,X_2)$ simulated from~\eqref{simustud} with $\alpha=3$ and with causal dependence $X_2\to X_1$. For the second pair, not in MWP, there is a third node variable, $X_3$, which is a confounder whose effect is rendered visible by the ray in the max-linear case. The perturbated model~\eqref{simustud}, more likely to be encountered in real-life scenarios, makes it very difficult to distinguish between these cases. Such cases of non-MWP pairs contribute to the high FCCPR metric in the boxplots of Appendix~\ref{Ap:sim}.

\begin{figure}[htbp]\centering 
	\includegraphics[ height=7cm, width=7cm, clip]{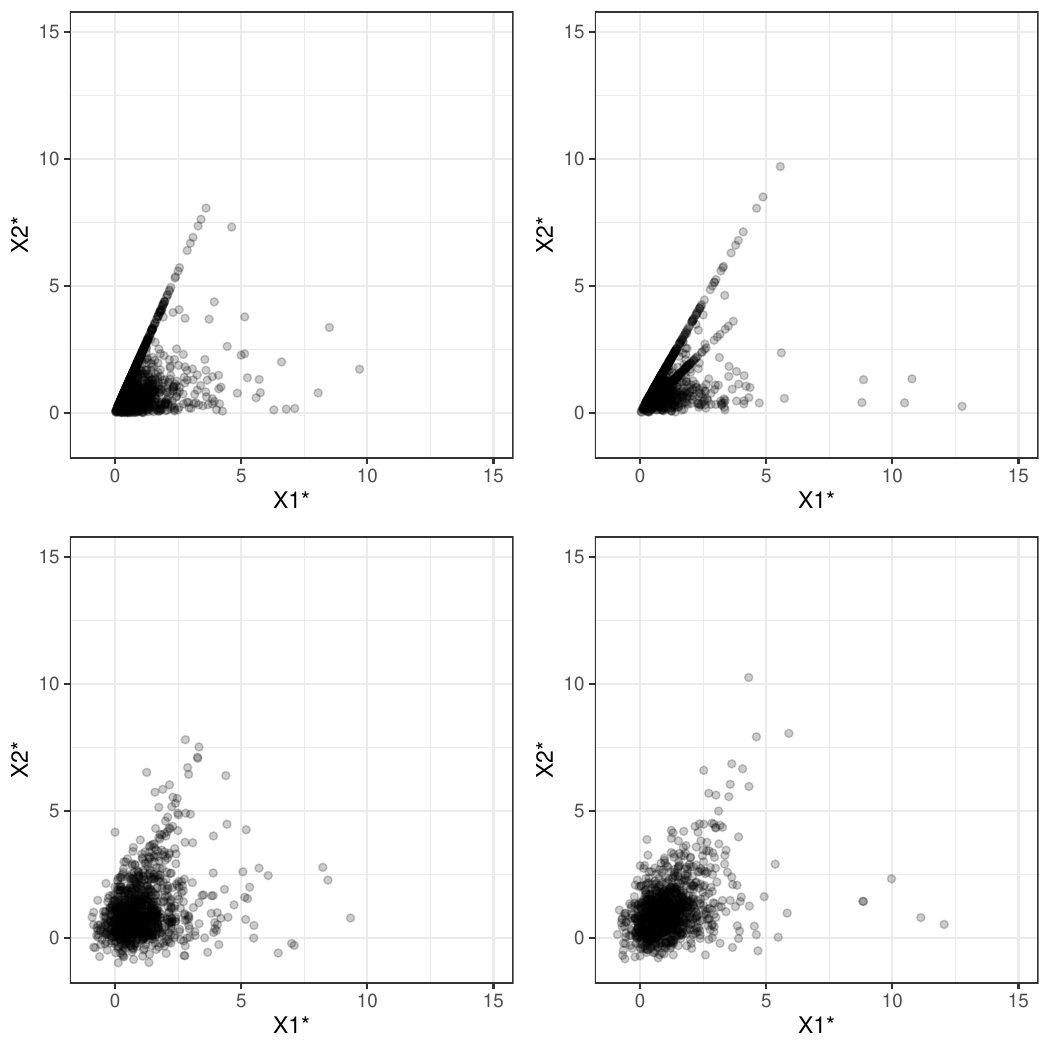}\vspace{-.2cm}
	\caption{The pair $(X_1, X_2)$ from an RMLM (top) and perturbated RMLM (bottom). The pair on the left is in MWP, and the one on the right is not. In both cases $X_2\to X_1$.} \label{noisermlm}
\end{figure}

For fixed $C,\alpha$  we simulate a sample of $n$ i.i.d. realizations as explained above. As described in Appendix~\ref{scalest}, in order to standardise the margins, we transform the data by applying the empirical integral transform \eqref{empstand} to Fr\'echet$(2)$ to $\bs{X^{*+}}$, the vector of a simulated RMLM as in equation~\eqref{simustud} and truncated to the non-negative orthant, that is,  $\bs{X^{*+}}=\max(\bs{X^{*}},\boldsymbol{0})$, with the maximum taken entrywise, and estimate scalings and extremal dependence measure from eqs. \eqref{estscal2}--\eqref{scalTim} needed as input for Algorithm~\ref{datdalg} to identify the MWP pairs. Finally, we estimate the metrics given in Section~\ref{sec:metrics}.
The box-plots in Figures~\ref{a23d20p12}--\ref{a23d40p12} are then based on 50 i.i.d. simulation runs, respectively. 

For the standardised observations, as errors for both $\alpha\in\{2,3\}$,
we fix $\eps_1=0.25, \,\eps_2=0.01,\, \eps_3=0.07,\, \eps_4=0.01,\, \eps_5=0.07$.  The selected ones are mainly due to simulation experience, taking into account both the pre-asymptotic setting and the influence of the noise. Similar to \citet{KK}, we set $a=1.0001$. The simulation studies in Section~\ref{Ap:sim2}, which show the sensitivity  with respect to the parameters $a, \eps_3, \eps_4, \eps_5$, indicate that these are appropriate choices that offer a reasonable trade-off between high true positive and low false discovery rates.

Finally, regarding the choices of the thresholds, for the sample size $n=1000$ we set $k_1=200$ and $k_2=100$, while for $n=5000$, $k_1=500$ and $k_2=200$.

\subsection{Evaluation metrics and results}\label{sec:metrics}

In this section and in Appendix~\ref{Ap:sim} we present the results of our simulation study using the metrics defined below to evaluate the predictive performance for the $50$ runs. 

Our focus is on the MWP from Definition~\ref{mwpair} for every pair of nodes in a DAG. 
For the estimator $\hat{\rm MWP}_k$ we use Theorem~\ref{incdagmod} in combination with the estimates in Section~\ref{sec:est}.
{We also involve true {causal} pairs (the set CP), true dependent pairs (the set DP), and those in ICP (the pairs in CP but in inverse index order).}

We define the following quantities for each DAG $\D_k$ with nodes $V_{\mathcal{D}_k}$ for $k\in\{1,\dots,50\}$,
\begin{itemize}
\item$\hat{\rm MWP}_k$-- the set of estimated pairs of nodes in MWP for $\mathcal{D}_k$; 
\item${\rm MWP}_{\mathcal{D}_k}$-- the set of pairs of nodes in MWP for $\mathcal{D}_k$;
{\item${\rm MWP}^c_{\mathcal{D}_k}$-- the set of pairs of nodes in MWP$^c$ for $\mathcal{D}_k$; here $(X_i,X_j)$ are not an RMLM or they could be independent, hence a degenerate RMLM;}
\item${\rm CP}_{\mathcal{D}_k}=\{(i,j)\in V_{\mathcal{D}_k}\times V_{\mathcal{D}_k}: j\in \An(i) \}$ ;
\item${\rm DP}_{\mathcal{D}_k}=\{(i,j)\in V_{\mathcal{D}_k}\times V_{\mathcal{D}_k}: \sigma_{ij}^2>0 \}$; 
\item${\rm ICP}_{\mathcal{D}_k}=\{(j,i)\in V_{\mathcal{D}_k}\times V_{\mathcal{D}_k}: (i,j)\in{\rm CP}_{\mathcal{D}_k} \}$.
\end{itemize}

The reported metrics are standard, representing True/False Positive Rates, and False Discovery Rates, see for instance \citet{fawcett} for standard formulas. However, since we measure the metrics across various categories of causal dependence, we extend them via the following formulas:
\begin{itemize}
\item {True Positive Rate:} 
$${\rm TPR}_k=\frac{\#\{\hat{{\rm MWP}}_k\cap {\rm MWP}_{\mathcal{D}_k} \}}{\#\{{\rm MWP}_{\mathcal{D}_k}\}};
$$
\item  False Causal and Confounder Positive Rate: 
$${\rm FCCPR}_k=\frac{\#\{\hat{{\rm MWP}}_k\cap {\rm MWP}^c_{\mathcal{D}_k}\cap{\rm CP}_{\mathcal{D}_k}\}}{\#\{ {\rm MWP}^c_{\mathcal{D}_k}\cap{\rm CP}_{\mathcal{D}_k}\}},
$$
which gives the proportion of estimated causal non-MWPs amongst causal non-MWPs, exemplified in $\mathcal{D}_2$ in Figure~\ref{introex};
\item  
False Dependence Causal Positive Rate: 
$${\rm FDCPR}_k=\frac{\#\{\hat{{\rm MWP}}_k\cap{\rm DP}_{\mathcal{D}_k}\cap {\rm CP}^c_{\mathcal{D}_k} \}}{\#\{{\rm DP}_{\mathcal{D}_k}\cap {\rm CP}^c_{\mathcal{D}_k}\}},
$$
which gives the proportion of estimated dependent and non-causal MWPs amongst dependent non-causal pairs, exemplified in $\mathcal{D}_3$ in Figure~\ref{introex};
\item False Discovery Rate: 
$${\rm FDR}_k=\frac{\#\{\hat{{\rm MWP}}_k\cap {\rm MWP}^c_{\mathcal{D}_k} \}}{\#\{\hat{{\rm MWP}_k}\}};
$$ 
\item False Dependence Discovery Rate:
$${\rm FDDR}_k=\frac{\#\{\hat{{\rm MWP}}_k\cap {\rm MWP}^c_{\mathcal{D}_k}\cap{\rm DP}_{\mathcal{D}_k}\}}{\#\{\hat{{\rm MWP}}_k\cap{\rm DP}_{\mathcal{D}_k}\}},
$$  
which gives the proportion of estimated dependent non-MWPs amongst estimated dependent MWPs;
\item 
False Dependence Causal Discovery Rate: 
$${\rm FDCDR}_k=\frac{\#\{\hat{{\rm MWP}}_k\cap{\rm DP}_{\mathcal{D}_k}\cap{\rm CP}^c_{\mathcal{D}_k}\}}{\#\{\hat{{\rm MWP}}_k\cap {\rm DP}_{\mathcal{D}_k}\}},
$$ 
which gives the proportion of estimated non-causal relations amongst estimated dependent MWPs;
\item False Causal Direction Discovery Rate:
$${\rm FCDDR}_k=\frac{\#\{\hat{{\rm MWP}}_k\cap{\rm ICP}_{\mathcal{D}_k}\}}{\#\{\hat{{\rm MWP}}_k\cap{\rm DP}_{\mathcal{D}_k}\}},
$$ 
which gives the proportion of inversely estimated causal relations amongst estimated dependent MWPs.
\end{itemize}

The metrics provide {knowledge} 
for each specific type of causal dependence. 
One can expect an error from a non-causal non-MWP pair to be considered as more severe relative to the case of a causal non-MWP. 
Hence it is important to distinguish between the various sources of errors.  
For instance, the difference between FDDR and FDCDR gives the contribution to the false discovery rate by the causal non-MWP pairs.
{The False Positive Rates report whether our new methodology is able to differentiate between the different} categories of causal dependence. 


Finally, we provide some comments on the simulation results as shown in Figures~\ref{a23d20p12}, ~\ref{a23d30p12} and ~\ref{a23d40p12} in Appendix~\ref{Ap:sim}. We observe a similar trend across the True Positive Rate (TPR), and the two False Positive Rates (FCCPR, FDCPR) for different levels of sparsity and regular variation index. The False Discovery Rates (FDR, FDDR, FDCDR) change in a similar fashion likewise. In general TPR lies above $80\%$ and at a similar level across the three dimensions $20,30,40$, despite a mild reduction as the latter increases. The large variation in the false positive rate, FCCPR, when $d=20$ and particularly for $p=0.1$, is due to the very small number of causal non-MWPs. The setting with $d=40$ and $p=0.1$ implies a large number of non-causal dependent pairs, leading to an increase in both FDCPR and FDCDR. Similarly, for $p=0.2$, the number of non-MWP causal pairs can be very large, in some cases more than double that of MWP, which leads to the increase in FCCPR and FDDR. Overall, however, we see that the methodology is able to distinguish between the different categories, even when a causal relation is present.

The performance of the algorithm is also affected by the regular variation index $\alpha\in\{2,3\}$, which influences the rate of convergence of the respective component maxima to their limiting Fr\'echet distributions; see for instance Prop.~2.12 in \citet{sres}. There is a slight decrease in the true positive rate from $\alpha=2$ to $\alpha=3$, and a larger decrease in the false positive rates; see for instance the metrics FCCPR and FDCPR in the box-plots of Appendix~\ref{Ap:sim}.
Amongst the false discovery rates, we notice that they become slightly lower from $\alpha=2$ to $\alpha=3$ in general. 

Concerning the level of sparsity $p$, we observe that the metric FDCPR is higher for lower values of $p$, corresponding to a higher number of non-causal pairs with confounders relative to the causal ones. The difference between FDR and FDDR indicates the presence of independent estimated non-MWPs, and we see that this increases for a lower value of $p$, as expected from a higher level of sparsity.

Finally, a larger sample size leads to an improvement in TPR, and generally to a decrease across {all} false positive rates and false discovery rates.

\section{Box-plots from the simulation study and data application}\label{Ap:sim}

{Figures \ref{a23d20p12}, \ref{a23d30p12} and \ref{a23d40p12} in this section present box-plots from the simulation study based on data with marginals transformed to Fr\'echet(2) as described after~\eqref{simustud}.}

\vspace{-.5cm}
\begin{figure}[htbp]\centering 
\includegraphics[ height=7cm, width=12cm, clip]{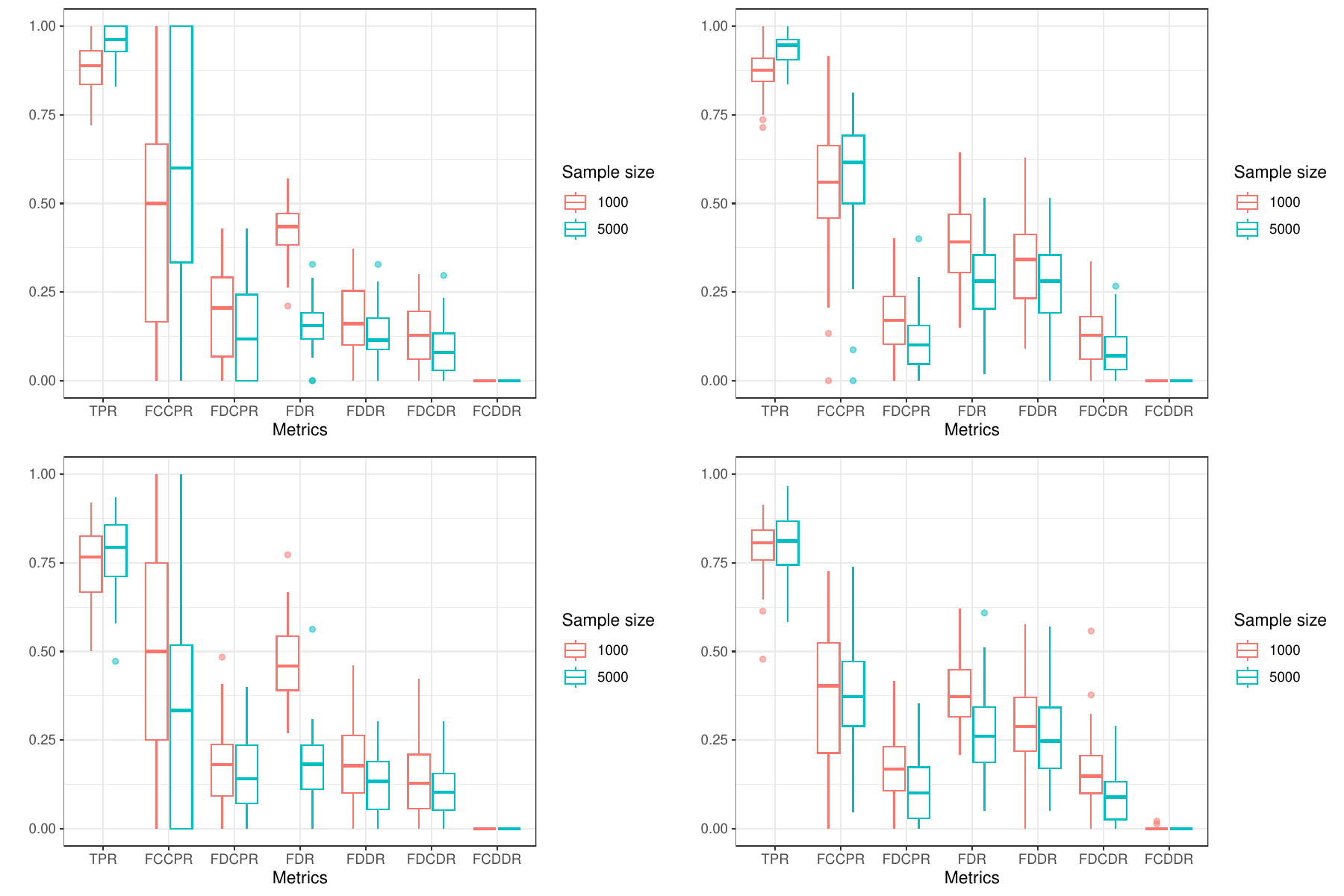}\vspace{-.2cm}
	\caption{Box-plots of the metrics over 50 {i.i.d.} DAGs with $d=20$ {nodes, sparsities} $p=0.1$ (left) and $p=0.2$ (right), {regular variation indices} $\alpha=2$ (top) and $\alpha=3$ (bottom).} \label{a23d20p12}
\end{figure}

\begin{figure}[htbp]\centering 
\includegraphics[ height=7cm, width=12cm, clip]{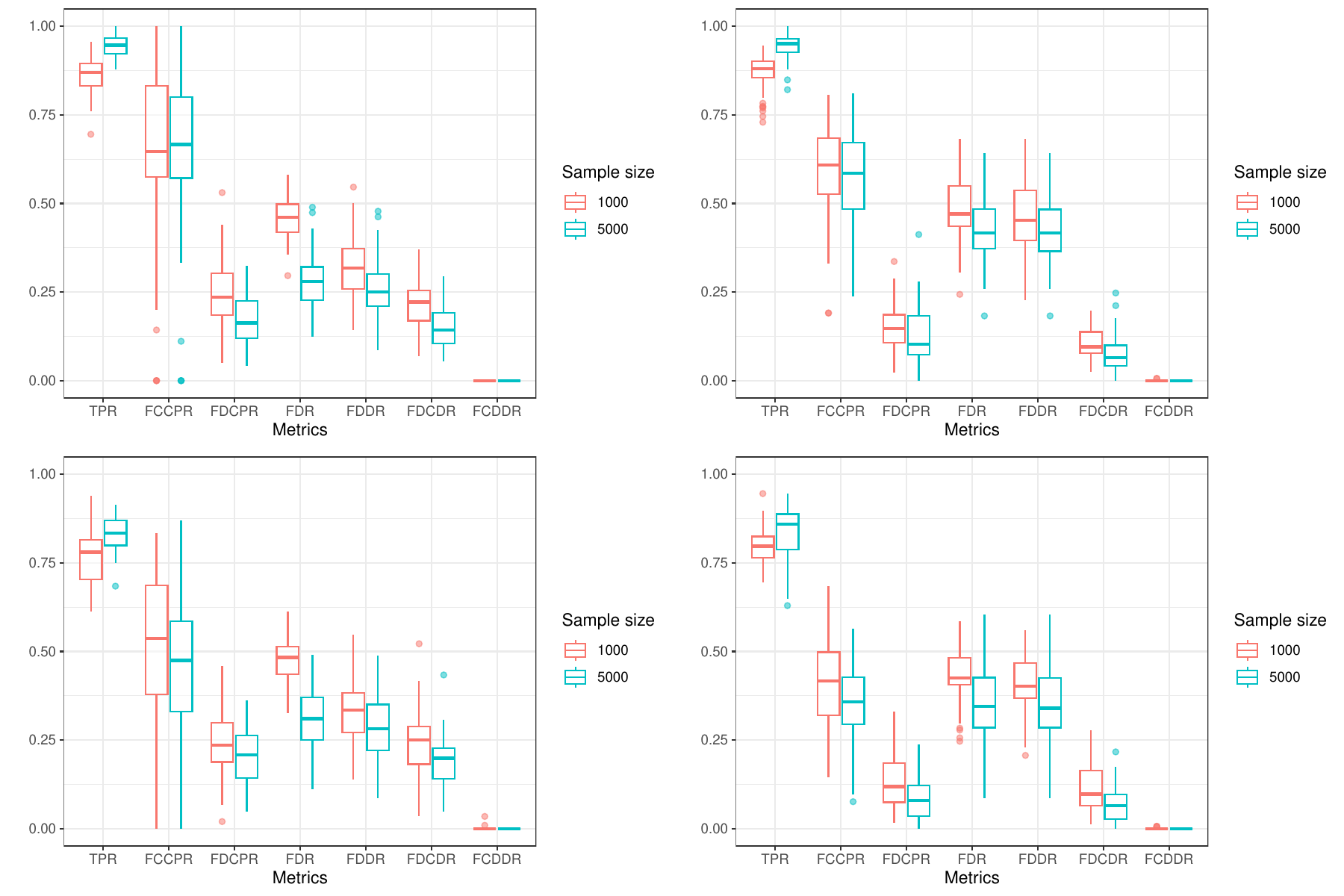}
	\caption{Box-plots of the metrics over 50 {i.i.d.} DAGs with $d=30$ {nodes, sparsities} $p=0.1$ (left) and $p=0.2$ (right), {regular variation indices} $\alpha=2$ (top) and $\alpha=3$ (bottom).} \label{a23d30p12}
\end{figure}

\begin{figure}[htbp]\centering 
\includegraphics[ height=7cm, width=12cm, clip]{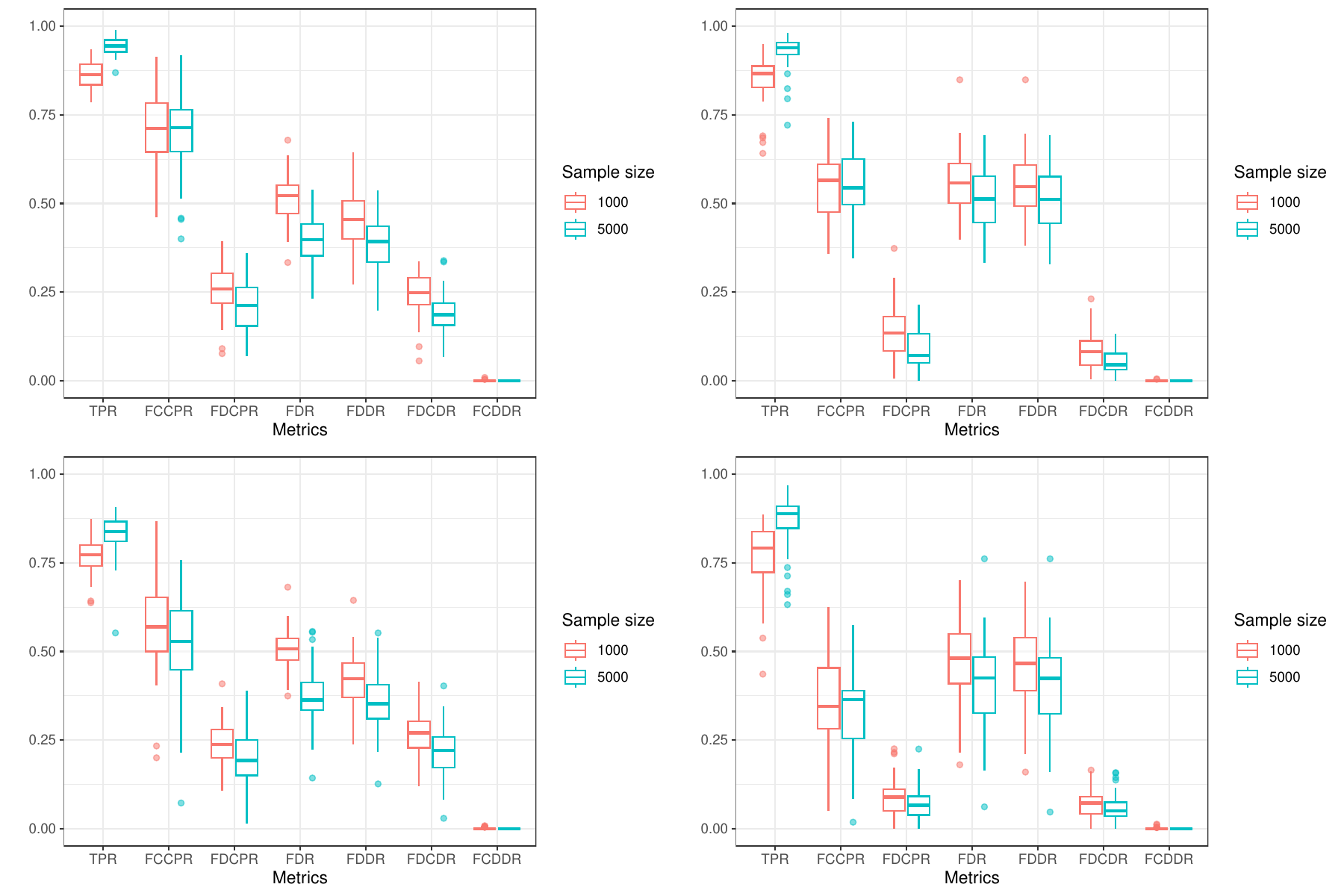}
	\caption{Box-plots of the metrics over 50 {i.i.d.} DAGs with $d=40$ {nodes, sparsities} $p=0.1$ (left) and $p=0.2$ (right), {regular variation indices} $\alpha=2$ (top) and $\alpha=3$ (bottom).} \label{a23d40p12}
\end{figure}

\subsection{Sensitivity to the $\eps$ error terms and to  $a$}\label{Ap:sim2}
In order to assess the sensitivity of Algorithm~\ref{datdalg} with respect to the error terms and the scalar $a$ we conduct a separate simulation study. We focus only on the error terms $\eps_3, \eps_4$ and $\eps_5$, whose choice is more critical . The choices $\eps_1, \eps_2$ selected in Section~\ref{simul} are rather non-restrictive, and values of $\eps_1$ in $[0.15, 0.25]$ and $\eps_2$ in $[0, 0.01]$ do not lead to significant changes in the TPR or FDR. To this end, in this simulation study we fix $\eps_1=0.25, \eps_2=0.01$.

For $\eps_3, \eps_4$ and $\eps_5$, we choose from the following sets $S_{\eps_3}=\{0.1, 0.07, 0.05\}$, $S_{\eps_4}=\{0,0. 01, .05\}$, $S_{\eps_5}=\{0.03, 0.07, 0.08\}$. Throughout the simulation we fix two of the $\eps$ terms to the values used in  Section~\ref{simul}, and let the remaining error take values in its respective set.  

For every combination of $\eps'$s, we apply Algorithm~\ref{datdalg} to 20 i.i.d. random DAGs, randomly generated with $d\in\{20, 30, 40\}$, $p\in\{0.1, 0.2\}$, $\alpha\in \{2,3\}$. Realizations are then obtained as outlined in Section~\ref{simul}. To every such DAG, we apply Algorithm~\ref{datdalg} for the combinations of $\eps'$s. The box plots of Figure~\ref{sensitivity} summarise the results for the metrics TPR, FDR, FCCPR, and FDCPR. The choice of parameters $(\eps_3, \eps_4, \eps_5)=(0.07, 0.01, 0.07)$ in the fourth bar leads to a high TPR, and to a lower FDR compared with the combinations that yield higher TPR. In general, we see that lowering $\eps_3$ leads to a significant decrease in both TPR, and FDR, whereas higher levels of $\eps_3$ lead to significant increases in both. As for $\eps_4$, which is aimed at removing independent pairs, we see that larger values decrease the TPR and also have a significant effect on the FDCPR,  which accounts for the pairs with confounders. Finally,  lowering $\eps_5$ leads to a significant decrease in TPR, whereas values slightly larger than $0.07$ only cause minor changes. The value we have selected, $(\eps_3, \eps_4, \eps_5)=(0.07, 0.01, 0.07)$, corresponding to the fourth column in each plot, offers a balanced trade-off between the TPR and FDR, and the remaining relevant metrics.

\begin{figure}[htbp]\centering 
\includegraphics[ height=10cm, width=10cm, clip]{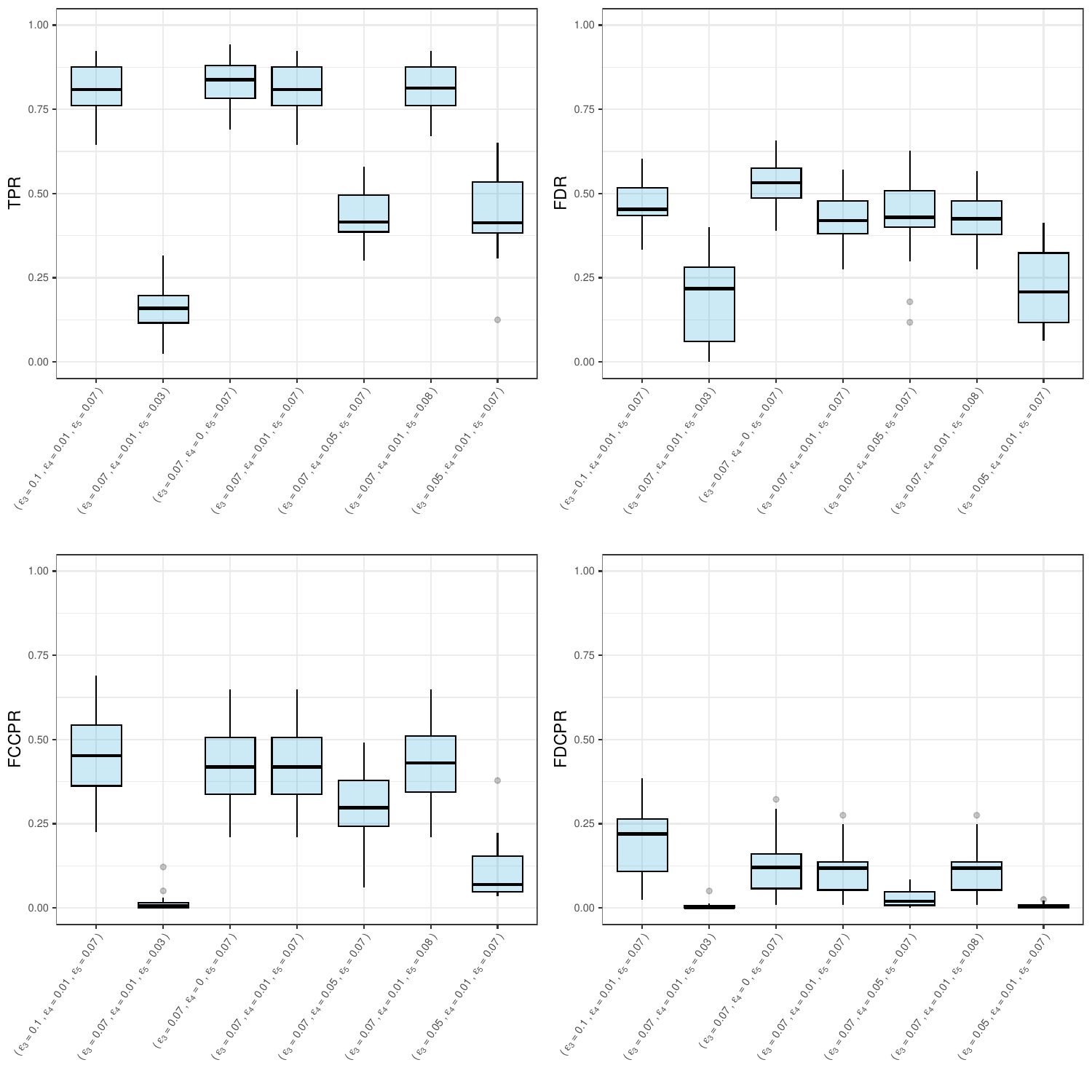}
	\caption{Box-plots of the metrics over 20 random DAGs for different choices of the $\eps'$s.} \label{sensitivity}
\end{figure}

\begin{figure}[htbp]\centering 
	\includegraphics[ height=8cm, width=8cm, clip]{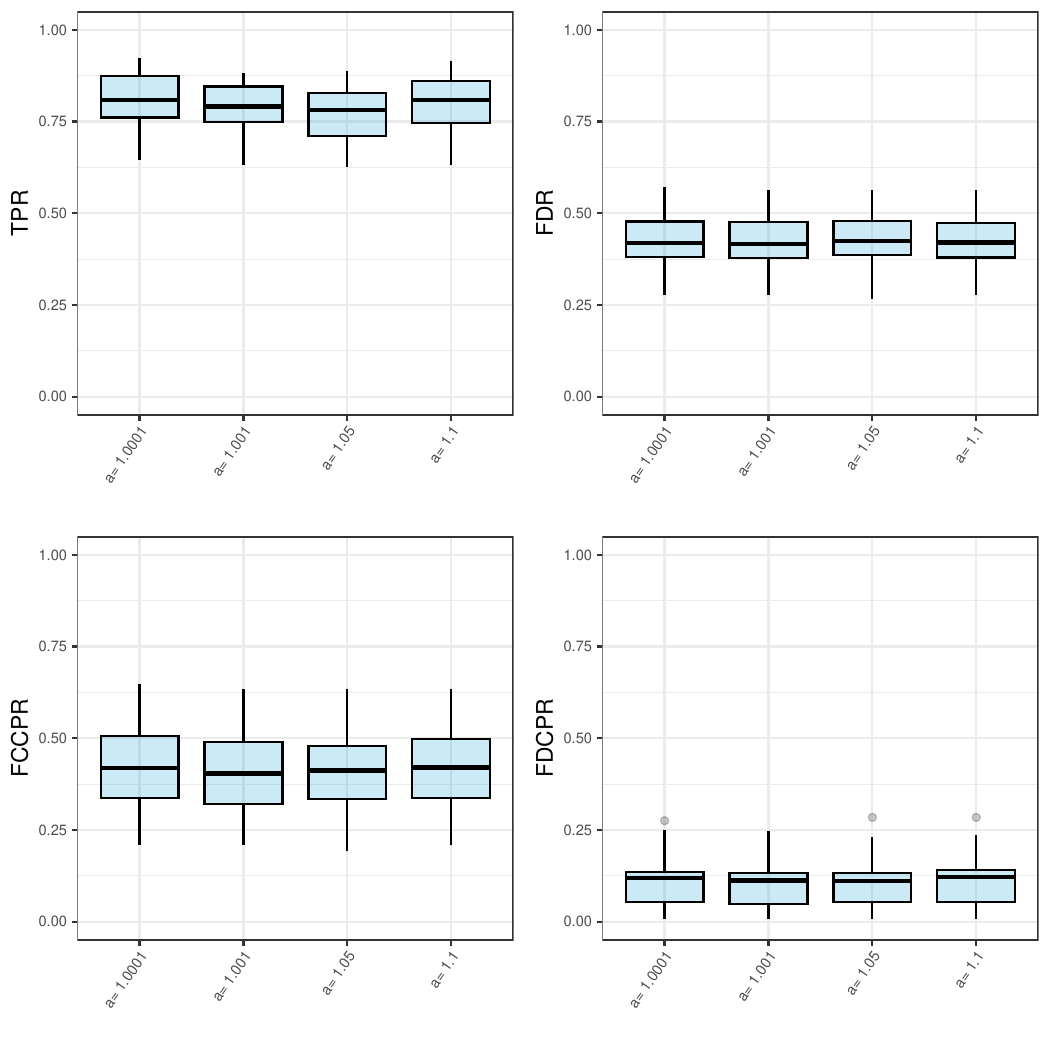}
	\caption{Box-plots of the metrics over 20 random DAGs for different choices of $a$.} \label{sensitivitya}
\end{figure}   

Finally, we investigate the performance of Algorithm~\ref{datdalg} for various choices of the scalar $a$. Here we fix the $\epsilon$'s to the same values as in Section~\ref{simul}. The random DAGs and the realizations are obtained following the procedure outlined at the beginning of this section. The resulting boxplots, depicted in Figure~\ref{sensitivitya}, show that the performance of estimating MWP is stable across the different choices of $a$.

\subsection{Food nutrient data example}\label{nutrientsensitivity}

In this section we follow the same practice as with the simulation study in Section~\ref{Ap:sim2}, and investigate how the entries of the matrix $\hat{\Delta}^{(2)}$ from the data example in Figure~\ref{fig:foodpairs}, with entries $\hat{\tau}_{ij}^2$ for all pairs $(i,j)$,  change as we let the $\eps$-terms vary. The resulting matrices are shown in Figure~\ref{sensitivity2}. On the one hand, as observed with the simulation setup in the previous section, lowering $\eps_3, \eps_5$, and increasing $\eps_4$ leads to fewer estimated pairs in MWP. In particular, the choice of $\eps_3=0.05$ yields only the pair (VA, RET) from the DAG of Figure~\ref{founddag3}, whereas the remaining pairs have a very weak level of dependence. On the other hand, higher values for $\eps_3, \eps_5$, and lower values of $\eps_4$ generally lead to many pairs estimated in MWP, including weakly dependent ones. We see that the pairs of nodes from the DAG in Figure~\ref{founddag3} are in all but two of the combinations, namely when $\eps_3=0.05$ and when $\eps_4=0.03$.

Figure~\ref{sensitivity3} provides a similar sensitivity analysis where we fix the $\epsilon$'s to the same values as in Section~\ref{simul} and let $a$ vary. 
We see here that the pairs composing the DAG in Figure~\ref{founddag3} are in all but the second matrix ($a=1.001$), in which case the nutrient pair (VA, LZ) is not in MWP. 

\begin{figure}[H]\centering 
	\includegraphics[ height=20cm, width=14cm, clip]{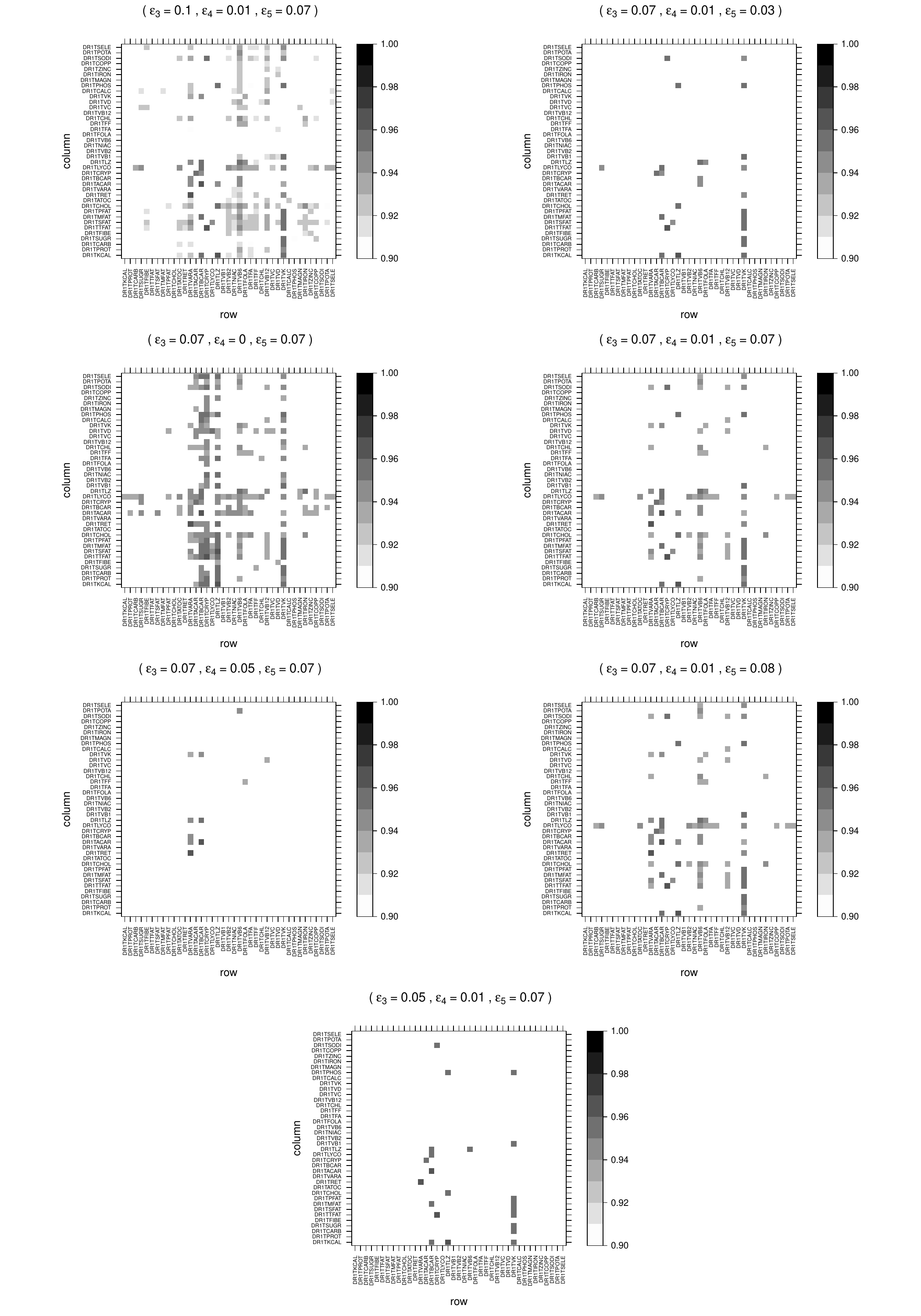}
	\hspace*{-0.5cm}\caption{Matrices $\hat{\Delta}^{(2)}$ {with entries $\hat{\tau}_{ij}^2$} for all pairs $(i,j)$, where Algorithm~\ref{datdalg2} outputs $\hat{P}_{ij}=1$ for the {chosen} $\eps'$s.  
}\label{sensitivity2}
\end{figure}

\begin{figure}[H]\centering 
	\includegraphics[ height=12cm, width=12cm, clip]{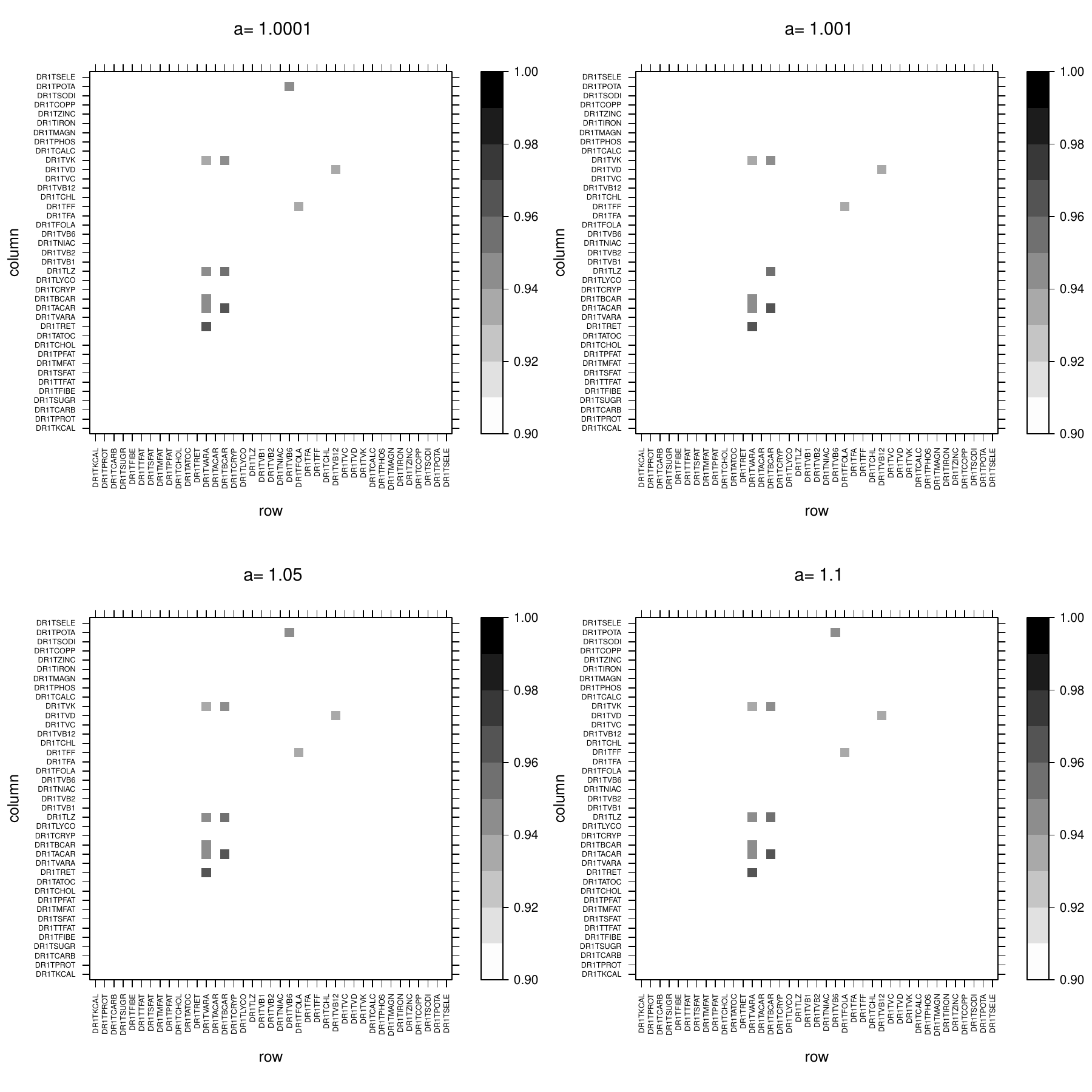}
	\caption{Matrices $\hat{\Delta}^{(2)}$ {with entries $\hat{\tau}_{ij}^2$} for all pairs $(i,j)$, where Algorithm~\ref{datdalg2} outputs $\hat{P}_{ij}=1$ for the {chosen} $a$.  
	}\label{sensitivity3}
\end{figure}

\subsection{A word of warning concerning standardisation}\label{discuss}

It is unrealistic to assume that different variables observed in real data exhibit the same tail behaviour
and we recall that 
standardisation is often used before applying statistical methods to extremes \citep[Chapter~8]{beirlant}.
As we shall see below, if variables have different regularly varying indices, max-linearity of the RMLM may no longer be suitable for capturing the causal behaviour of the data at the original scale.
Nevertheless, as we illustrate in the following example, the RMLM may still be an appropriate model for describing some extremal causality of the standardised variables. 

Let the true DAG be the one on the left of Figure~\ref{fig_app}, and suppose that we observe two structural equation models generated by different mechanisms that are supported on the left-hand DAG in Figure~\ref{fig_app}. 
Let the innovations $Z_1,Z_2\in{\rm RV}_+(\alpha_1)$ and $Z_3\in{\rm RV}_+(\alpha_3)$ with $\alpha_3>\alpha_1$ such that $Z_3$ has a lighter tail. 
Consider the two causal mechanisms in~\eqref{m1app} and~\eqref{m2app} defined recursively on the left-hand DAG of Figure~\ref{fig_app}:\\
\hspace*{-1cm}
\begin{tabularx}{\textwidth}{XX}
	{\begin{align}\label{m1app}
	X_3&=Z_3\\\nonumber
	X_2&=c_{22}Z_2\vee  c_{23}X_3\\\nonumber
	X_1&=c_{11} Z_1\vee c_{12}X_2\vee  c_{13}X_3,
	\end{align}}\hspace*{-1.5cm}
	&{\begin{align}\label{m2app}
	X_3&=Z_3\\\nonumber
	X_2&=c_{22}Z_2\vee  c_{23}X_3^{\alpha_3/\alpha_1}\\\nonumber
	X_1&=c_{11} Z_1\vee c_{12}X_2\vee  c_{13}X_3^{\alpha_3/\alpha_1}.
	\end{align}}
\end{tabularx}\\
In~\eqref{m1app}, $X_1,X_2\in {\rm RV}_+(\alpha_1)$ and $X_3\in {\rm RV}_+(\alpha_3)$. 
Note, however, that because of the heavier tail of the innovations $Z_1$ and $Z_2$, $\mathbb{P}(X_1>x)\sim \mathbb{P}(c_{11} Z_1\vee c_{12}X_2>x)$ as $x\to\infty$, and similarly $\mathbb{P}(X_2>x)\sim \mathbb{P}(c_{22} Z_2>x)$, implying that $X_1$ and $X_2$ are asymptotically independent of $X_3$. 
If we now standardise $X_3$ to 
$\tilde X_3=X_3^{\alpha_3/\alpha_1}\in {\rm RV}_+(\alpha_1)$, the tail behaviour of the variables will be equivalent to that of
\begin{align}\label{extsem1}
\tilde X_3 =\tilde Z_3=Z_3^{\alpha_3/\alpha_1},\quad\quad
\tilde X_2 =c_{22}Z_2,\quad\quad 
\tilde X_1 =c_{11} Z_1\vee c_{12}\tilde X_2,
\end{align}
which is a structural equation model supported on the middle DAG of Figure~\ref{fig_app}.
Consider now~\eqref{m2app}. Although the original variables of the system are not represented by an RMLM, the standardised variables $\tilde X_1, \tilde X_2, \tilde X_3\in{\rm RV}_+(\alpha_1)$, can be represented as
\begin{align}\label{extsem2}
\tilde X_3 =\tilde Z_3,\quad\quad
\tilde X_2 =c_{22}Z_2\vee c_{23}\tilde X_3,\quad\quad
\tilde X_1 =c_{11} Z_1\vee c_{12}\tilde X_2\vee c_{13}\tilde X_3.
\end{align}


The examples~\eqref{extsem1} and~\eqref{extsem2} above show that standardisation of the original variables to the same index of regular variation may result in extremal causal mechanisms that are different from the original ones; see the middle and left-hand DAGs in Figure~\ref{fig_app}.
We also observe, however, that causal dependencies among those node variables that share the largest risks are still preserved, providing valuable insight in risk analysis.

			\begin{figure}[H]
		\begin{center}
					\resizebox{3cm}{2cm}{\begin{tikzpicture}[
		> = stealth,
		shorten > = 1pt, 
		auto,
		node distance = 2cm, 
		semithick 
		]
		\tikzstyle{every state}=[
		draw = black,
		thick,
		fill = white,
		minimum size = 4mm,scale=1
		]
		\node[state] (3) {$3$};
		\node[state] (2) [below right=1cm and 1cm of 3] {$2$};
		\node[state] (1) [below left=1cm and 1cm of 3] {$1$};
		
		\path[->][blue] (3) edge node {} (2);
		\path[->][blue] (3) edge node {} (1);
		\path[->][blue] (2) edge node {} (1);
		\end{tikzpicture}}
            \hspace{1cm}
					\resizebox{3cm}{2cm}{\begin{tikzpicture}[
						> = stealth,
						shorten > = 1pt, 
						auto,
						node distance = 2cm, 
						semithick 
						]
						\tikzstyle{every state}=[
						draw = black,
						thick,
						fill = white,
						minimum size = 4mm,scale=1
						]
						\node[state] (3) {$3$};
						\node[state] (2) [below right=1cm and 1cm of 3] {$2$};
						\node[state] (1) [below left=1cm and 1cm of 3] {$1$};
						
						\path[->][blue] (2) edge node {} (1);;
						\end{tikzpicture}}	
					\hspace{1cm}
					\resizebox{3cm}{2cm}{\begin{tikzpicture}[
						> = stealth,
						shorten > = 1pt, 
						auto,
						node distance = 2cm, 
						semithick 
						]
						\tikzstyle{every state}=[
						draw = black,
						thick,
						fill = white,
						minimum size = 4mm,scale=1
						]
						\node[state] (3) {$3$};
						\node[state] (2) [below right=1cm and 1cm of 3] {$2$};
						\node[state] (1) [below left=1cm and 1cm of 3] {$1$};
						
						\path[->][blue] (3) edge node {} (2);
						\path[->][blue] (3) edge node {} (1);
						\path[->][blue] (2) edge node {} (1);
						\end{tikzpicture}}		
				\end{center}
				\caption{
                True DAG (left). The DAGs that support the causal mechanisms underlying the systems~\eqref{extsem1} (middle) and~\eqref{extsem2} (right).} \label{fig_app} 
			\end{figure}

\end{document}